\documentclass[12pt,a4paper,oneside]{amsart}
\usepackage[utf8]{inputenc}
\usepackage{amsmath,amscd,hyperref, amsfonts, amssymb, amsthm}
\usepackage{setspace,kantlipsum}
\usepackage{enumerate,enumitem}
\usepackage{tikz-cd}
\usepackage{mathrsfs}
\usepackage{textcomp}
\usepackage[margin=1in]{geometry}
\usepackage{colonequals}
\usepackage{mathtools}
\DeclarePairedDelimiter{\abs}{\lvert}{\rvert}

\newcommand{\Sing}{\mathrm{Sing}}
\newcommand{\mult}{\mathrm{mult}}

\newcommand{\Jac}{\mathrm{Jac}}
\newcommand{\lct}{\mathrm{lct}}
\renewcommand{\d}{\partial}

\usepackage{expl3}

\ExplSyntaxOn
\prop_new:N \g_cite_map_prop
\tl_new:N \l_citekey_result_tl

\cs_new:Npn \mapcitekey #1#2 {
  \clist_map_inline:nn {#2}
       {  \prop_gput:Nnn  \g_cite_map_prop  {##1} {#1}   }
}

\cs_new:Npn \getcitekey #1 {
   \prop_get:NoN \g_cite_map_prop{#1}  \l_citekey_result_tl
   \quark_if_no_value:NF \l_citekey_result_tl
       {  \tl_set_eq:NN #1  \l_citekey_result_tl  }
}

\cs_new:Npn \showcitekeymaps {\prop_show:N  \g_cite_map_prop }
\ExplSyntaxOff

\usepackage{etoolbox}
\makeatletter
\patchcmd{\@citex}{\if@filesw}{\getcitekey\@citeb \if@filesw}%
    {\typeout{*** SUCCESS ***}}{\typeout{*** FAIL ***}}
\patchcmd{\nocite}{\if@filesw}{\getcitekey\@citeb \if@filesw}%
    {\typeout{*** SUCCESS ***}}{\typeout{*** FAIL ***}}
\makeatother

\usepackage{enumitem}

\newcommand{\MHM}{\operatorname{MHM}}

\newcommand{\pt}{\mathit{pt}}
\newcommand{\Dmod}{\mathscr{D}}
\newcommand{\Mmod}{\mathcal{M}}

\newcommand{\shT}{\mathscr{T}}

\newcommand{\derR}{\mathbf{R}}

\newcommand{\dx}{\mathit{dx}}

\newcommand{\ZZ}{\mathbb{Z}}
\newcommand{\QQ}{\mathbb{Q}}
\newcommand{\RR}{\mathbb{R}}
\newcommand{\CC}{\mathbb{C}}

\newcommand{\PP}{\mathbb{P}}

\newcommand{\menge}[2]{\bigl\{ \thinspace #1 \thinspace\thinspace \big\vert%
\thinspace\thinspace #2 \thinspace \bigr\}}

\DeclareMathOperator{\Res}{Res}

\DeclareMathOperator{\id}{id}

\renewcommand{\Re}{\operatorname{Re}}

\DeclareMathOperator{\Supp}{Supp}
\DeclareMathOperator{\codim}{codim}

\DeclareMathOperator{\Sym}{Sym}
\DeclareMathOperator{\gr}{gr}
\DeclareMathOperator{\DR}{DR}

\DeclareMathOperator{\End}{End}

\DeclareMathOperator{\Ext}{Ext}
\DeclareMathOperator{\GL}{GL}

\DeclareMathOperator{\Pic}{Pic}

\newcommand{\sltwo}{\mathfrak{sl}_2(\CC)}

\newcommand{\shf}[1]{\mathscr{#1}}
\newcommand{\OX}{\shf{O}_X}
\newcommand{\OmX}{\Omega_X}

\newcommand{\defeq}{\underset{\textrm{def}}{=}}

\newcommand{\argbl}{-}

\def\overbar#1#2#3{{%
	\setbox0=\hbox{\displaystyle{#1}}%
	\dimen0=\wd0
	\advance\dimen0 by -#2 
	\vbox {\nointerlineskip \moveright #3 \vbox{\hrule height 0.3pt width \dimen0}%
		\nointerlineskip \vskip 1.5pt \box0}%
}}

\newcommand{\into}{\hookrightarrow}

\newcommand{\ju}{j^{\ast}}
\newcommand{\fu}{f^{\ast}}
\newcommand{\fl}{f_{\ast}}

\newcommand{\iu}{i^{\ast}}

\newcommand{\pl}{p_{\ast}}

\newcommand{\tl}{t_{\ast}}

\newcommand{\shE}{\shf{E}}

\newcommand{\shO}{\shf{O}}

\theoremstyle{plain}

\newcommand{\theoremref}[1]{\hyperref[#1]{Theorem~\ref*{#1}}}
\newcommand{\lemmaref}[1]{\hyperref[#1]{Lemma~\ref*{#1}}}
\newcommand{\definitionref}[1]{\hyperref[#1]{Definition~\ref*{#1}}}
\newcommand{\propositionref}[1]{\hyperref[#1]{Proposition~\ref*{#1}}}
\newcommand{\conjectureref}[1]{\hyperref[#1]{Conjecture~\ref*{#1}}}
\newcommand{\corollaryref}[1]{\hyperref[#1]{Corollary~\ref*{#1}}}
\newcommand{\exampleref}[1]{\hyperref[#1]{Example~\ref*{#1}}}
\newcommand{\exerciseref}[1]{\hyperref[#1]{Exercise~\ref*{#1}}}

\makeatletter
\let\old@caption\caption
\renewcommand*{\caption}[1]{%
	\setcounter{figure}{\value{equation}}%
	\stepcounter{equation}%
	\old@caption{#1}\relax%
}
\makeatother

\newcounter{intro}

\newtheorem{intro-conjecture}[intro]{Conjecture}
\newtheorem{intro-corollary}[intro]{Corollary}
\newtheorem{intro-theorem}[intro]{Theorem}

\newcommand{\OY}{\shO_Y}

\newcommand{\parref}[1]{\hyperref[#1]{\S\ref*{#1}}}
\newcommand{\chapref}[1]{\hyperref[#1]{Chapter~\ref*{#1}}}

\makeatletter
\newcommand*\if@single[3]{%
  \setbox0\hbox{${\mathaccent"0362{#1}}^H$}%
  \setbox2\hbox{${\mathaccent"0362{\kern0pt#1}}^H$}%
  \ifdim\ht0=\ht2 #3\else #2\fi
  }
\newcommand*\rel@kern[1]{\kern#1\dimexpr\macc@kerna}
\newcommand*\widebar[1]{\@ifnextchar^{{\wide@bar{#1}{0}}}{\wide@bar{#1}{1}}}
\newcommand*\wide@bar[2]{\if@single{#1}{\wide@bar@{#1}{#2}{1}}{\wide@bar@{#1}{#2}{2}}}
\newcommand*\wide@bar@[3]{%
  \begingroup
  \def\mathaccent##1##2{%
    \if#32 \let\macc@nucleus\first@char \fi
    \setbox\z@\hbox{$\macc@style{\macc@nucleus}_{}$}%
    \setbox\tw@\hbox{$\macc@style{\macc@nucleus}{}_{}$}%
    \dimen@\wd\tw@
    \advance\dimen@-\wd\z@
    \divide\dimen@ 3
    \@tempdima\wd\tw@
    \advance\@tempdima-\scriptspace
    \divide\@tempdima 10
    \advance\dimen@-\@tempdima
    \ifdim\dimen@>\z@ \dimen@0pt\fi
    \rel@kern{0.6}\kern-\dimen@
    \if#31
      \overline{\rel@kern{-0.6}\kern\dimen@\macc@nucleus\rel@kern{0.4}\kern\dimen@}%
      \advance\dimen@0.4\dimexpr\macc@kerna
      \let\final@kern#2%
      \ifdim\dimen@<\z@ \let\final@kern1\fi
      \if\final@kern1 \kern-\dimen@\fi
    \else
      \overline{\rel@kern{-0.6}\kern\dimen@#1}%
    \fi
  }%
  \macc@depth\@ne
  \let\math@bgroup\@empty \let\math@egroup\macc@set@skewchar
  \mathsurround\z@ \frozen@everymath{\mathgroup\macc@group\relax}%
  \macc@set@skewchar\relax
  \let\mathaccentV\macc@nested@a
  \if#31
    \macc@nested@a\relax111{#1}%
  \else
    \def\gobble@till@marker##1\endmarker{}%
    \futurelet\first@char\gobble@till@marker#1\endmarker
    \ifcat\noexpand\first@char A\else
      \def\first@char{}%
    \fi
    \macc@nested@a\relax111{\first@char}%
  \fi
  \endgroup
}
\makeatother

\newcommand{\wbar}[1]{\mathpalette\dowidebar{#1}}
\newcommand{\dowidebar}[2]{\widebar{#1#2}}

\theoremstyle{plain} 
\newtheorem{thm}{Theorem}[section]
\newtheorem{lemma}[thm]{Lemma}
\newtheorem{prop}[thm]{Proposition}
\newtheorem{corollary}[thm]{Corollary}

\theoremstyle{definition}

\newtheorem{setup}[thm]{Set-up}
\newtheorem{definition}[thm]{Definition}

\newtheorem{example}[thm]{Example}
\newtheorem{remark}[thm]{Remark}

\newtheorem{conjecture}[thm]{Conjecture}

\newtheorem{question}[thm]{Question}

\newcommand{\C}{\mathbb{C}}
\newcommand{\Q}{\mathbb{Q}}
\newcommand{\Z}{\mathbb{Z}}
\newcommand{\R}{\mathbb{R}}
\newcommand{\N}{\mathbb{N}}
\renewcommand{\P}{\mathbb{P}}

\newcommand{\cG}{\mathcal{G}}

\newcommand{\cI}{\mathcal{I}}
\newcommand{\cJ}{\mathcal{J}}

\newcommand{\cM}{\mathcal{M}}

\newcommand{\cO}{\mathcal{O}}

\newcommand{\cS}{\mathcal{S}}

\newcommand{\cV}{\mathcal{V}}

\newcommand{\bA}{\mathbf{A}}

\newcommand{\bZ}{\mathbf{Z}}

\newcommand{\fm}{\mathfrak{m}}

\newcommand{\sD}{\mathscr{D}}

\newcommand{\sH}{\mathscr{H}}

\newcommand{\sO}{\mathscr{O}}

\newcommand{\sT}{\mathscr{T}}

\newcommand{\Cur}{\mathfrak{C}}
\DeclareMathOperator{\dv}{div}
\newcommand{\Hsl}{\mathsf{H}}
\newcommand{\Xsl}{\mathsf{X}}
\newcommand{\Ysl}{\mathsf{Y}}
\newcommand{\wsl}{\mathsf{w}}

\newcommand{\Dbcoh}{D_{\mathit{coh}}^{\mathit{b}}}
\newcommand{\tensor}{\otimes}
\newcommand{\Mmodt}{\widetilde{\Mmod}}
\newcommand{\Mt}{\widetilde{M}}

\numberwithin{equation}{section}
\begin{document}

\title{Higher multiplier ideals}

\author{Christian Schnell and Ruijie Yang}
   % \date{\today}
\subjclass[2020]{Primary: 14J17, 14D07, 14F10, Secondary: 14F18}

\keywords{Higher multiplier ideals, nearby and vanishing cycles, V-filtrations, twisted differential operators, complex Hodge modules, the center of minimal exponent, vanishing theorems.}
\maketitle

{\centering \footnotesize\emph{Dedicated to Rob Lazarsfeld on the occasion of his 70th birthday}\par}

\begin{abstract} 
We associate a family of ideal sheaves to any $\Q$-effective divisor
on a complex manifold, called higher multiplier ideals, using the theory of
mixed Hodge modules and $V$-filtrations. This family is indexed by two parameters,
an integer indicating the Hodge level and a rational number, and these ideals
admit a weight filtration. When the Hodge level is zero, they recover the usual
multiplier ideals.
 
We study the local and global properties of higher multiplier ideals systematically. In particular, we prove vanishing theorems and restriction theorems, provide
criteria for the nontriviality, and introduce the center of minimal exponent (generalizing the notion of minimal log canonical center). The main idea is to exploit the
global structure of the $V$-filtration along an effective divisor using the notion of
twisted Hodge modules. 
As applications, we prove new cases of conjectures by Debarre, Casalaina-Martin and
Grushevsky on singularities of theta divisors on principally polarized abelian
varieties.
\end{abstract}

%\tableofcontents
\setcounter{tocdepth}{1}

\section{Introduction}

Let $X$ be a complex manifold of dimension $n$, and let $D$ be an effective divisor
on $X$. In this work, we construct a family of ideal sheaves $\cI_{k,\alpha}(D)$ from
the pair $(X,D)$, using the Hodge theory of the Kashiwara-Malgrange filtration along
$D$; they are indexed by $k\in \N$ and $\alpha \in \Q$. In the special case where
the divisor $D$ is reduced, these ideals were first considered by Saito
\cite{Saito16} under the name ``microlocal multiplier ideal sheaves''. We prefer to
call them \emph{higher multiplier ideals}, because $\cI_{0,<-\alpha}(D)=\cJ(X,\alpha
D)$ is the usual multiplier ideal of the effective $\Q$-divisor $\alpha D$. (Here the
subscript $<\beta$ is an abbreviation for $\beta - \varepsilon$, where $\varepsilon
> 0$ is a small positive rational number.) 

One purpose of this work is to show that these ideal sheaves can serve as a new
measure of singularities of $D$ and lead to a number of applications for
singularities of divisors. The theory is most interesting in those cases where there
is no useful information left in the usual multiplier ideals; this happens for
example for theta divisors on principally polarized abelian varieties (more generally
when $(X,D)$ is log canonical). The main improvement over Saito's work \cite{Saito16}
is that we look at the ideals $\cI_{k, \alpha}(D)$ from a \emph{global} point of view,
using the theory of twisted $\Dmod$-modules. This leads to many new results. In
the remainder of the introduction, we outline the most important of these: a version
of the Nadel vanishing theorem; restriction and semicontinuity theorems; and a numerical
jumping criterion. We postpone the (somewhat technical) definition until \S \ref{sec:
construction of HMI introduction}.

In \cite{DY25}, Davis and the second author prove several additional
results about higher multiplier ideals (such as a birational transformation formula)
and also settle the question on how higher multiplier ideals are related
to \emph{Hodge ideals} \cite{MPHodgeideal}.

\subsection{Local properties}

To begin with, the higher multiplier ideals $\cI_{k,\alpha}(D)$ have similar formal
properties as usual multiplier ideals. First, the sequence of ideal sheaves
$\{\cI_{k,\alpha}(D)
\}_{\alpha\in \Q}$ is discrete and right continuous. For any $\alpha \in \Q$, there
exist isomorphisms
\begin{align*}
 \cI_{k,\alpha-1}(D) &\xrightarrow{\sim}\cI_{k,\alpha}(D)\otimes \cO_X(-D),\quad \textrm{for $\alpha<0$},\\
\cI_{k+1,\alpha+1}(D)&\xrightarrow{\sim}  \cI_{k,\alpha}(D),  \quad \textrm{for $\alpha\geq -1$}, \end{align*}
and as a consequence, all $\cI_{k,\alpha}(D)$ are controlled by those with $\alpha\in [-1,0]$.
Moreover $\cI_{k+1,\alpha}(D) \subseteq \cI_{k,\alpha}(D)$ for any $\alpha,k$, see
Corollary \ref{corollary: Ik+1 contained in Ik}. 

Second, for a fixed value of $k \geq 0$, one has the containment
$\cI_{k,\alpha}(D)\subseteq \cI_{k,\beta}(D)$ whenever $\alpha\leq \beta$. As with
usual multiplier ideals, we call a rational number $\alpha\in \Q$ a \emph{jumping number} if
$\cI_{k,<\alpha}(D)\neq \cI_{k,\alpha}(D)$; one can show that all jumping numbers for
$\cI_{k, \alpha}(D)$ are strictly less than $k$. The graded pieces
\begin{equation}\label{eqn: graded pieces in the introduction}
    \cG_{k,\alpha}(D)\colonequals \cI_{k,\alpha}(D)/\cI_{k,<\alpha}(D)
\end{equation}
are supported inside the singular locus of $D$ if $\alpha\in (-1,0]$. They have an
additional ``weight filtration", indexed by $\Z$, whose subquotients we denote by the
symbol $\gr^W_{\ell}\cG_{k,\alpha}(D)$. This induces a weight filtration
$W_{\bullet}\cI_{k,\alpha}(D)$ via \eqref{eqn: graded pieces in the introduction}. 

\begin{remark}
We do not have a definition of higher multiplier ideals for $\Q$-divisors,
but it is possible to define the rank-one torsion free sheaf $\cI_{k,
\alpha}(D) \tensor \cO_X(kD)$ for arbitrary effective $\Q$-divisors; this is explained
in \S \ref{sec: Q divisors}.
\end{remark}

Let us summarize the most important local properties of higher multiplier ideals.
First, they behave well under restriction (compare \cite[Thm.~16.1]{MPHodgeideal}).
\begin{thm}\label{thm: restriction theorem}
Let $i:H\hookrightarrow X$ be the closed embedding of a smooth hypersurface that is not entirely contained in the support of $D$ so that the pullback $D_H=i^{\ast}D$ is defined. Then one has an inclusion $\cI_{k,\alpha}(D_H)\subseteq \cI_{k,\alpha}(D)\cdot \cO_H$, where the latter is defined as the image of $i^{\ast}\cI_{k,\alpha}(D)\to \cO_H$. If $H$ is sufficiently general, then
$\cI_{k,\alpha}(D_H)=\cI_{k,\alpha}(D)\cdot \cO_H$.
\end{thm}
By a standard argument, the restriction theorem implies the semicontinuity of higher
 multiplier ideals in families (Theorem \ref{thm: semicontinuity theorem}). We use these theorems to show that the presence of very singular points forces
 certain higher multiplier ideals to jump. Define
\[ \mathrm{Sing}_m(D)\colonequals \{x\in X \mid \mult_x D\geq m \}.\]
The following result is our version of \cite[Thm.~E]{MPHodgeideal} for higher
multiplier ideals.
\begin{thm}\label{thm: numerical criterion for nontriviality}
Suppose $Z\subseteq \Sing_m(D)$ is an irreducible component of dimension $d$. Write $n-d=km+r$ with $k\in \mathbb{N}$ and $0\leq r\leq m-1$. If $\alpha=-\frac{r+q-1}{m}$ for some $1\leq q\leq  \max(m-r,m-1)$, then $\cI_{k,<\alpha}(D)\subseteq  \cI_Z^{\langle q\rangle}$.
\end{thm}
Since the \emph{minimal exponent} $\tilde{\alpha}_D$ of the divisor $D$ (see \S \ref{sec: minimal exponent}) can be computed from the collection $\{\cI_{k,\alpha}(D)\}$ (Lemma~\ref{lemma: minimal exponent via Gkalpha}), this leads to a new upper bound of $\tilde{\alpha}_D$.

\begin{corollary}\label{corollary: upper bound of minimal exponent article}
For any irreducible component $Z$ of $\Sing_m(D)$ and $m\geq 2$, we have
\begin{equation}\label{eqn: upper bound by codimension of multiplicity}
\tilde{\alpha}_D \leq \frac{\mathrm{codim}_X(Z)}{m}.
\end{equation}
\end{corollary}
In particular, for $m=2$, we have $\textrm{codim}_X(D_{\textrm{sing}})\geq 2\tilde{\alpha}_D$. This strengthens \cite[Proposition 7.4]{MPminimalexponent}, which says that if $\tilde{\alpha}_D>k$, then $\textrm{codim}_X(D_{\textrm{sing}})\geq 2k+1$.

\subsection{Global properties}
We prove several vanishing theorems, similar in spirit to Nadel's vanishing theorem for usual
multiplier ideals, but involving the de Rham-type complex 
\begin{align*}
\gr^W_{\ell}K_{k,\alpha}(D)\colonequals [\Omega^{n-k}_X\otimes \gr^W_{\ell}\cG_{0,\alpha}(D) &\to \cdots \\
\cdots \to \Omega^{n-1}_X\otimes L^{k-1}&\otimes \gr^W_{\ell}\cG_{k-1,\alpha}(D) \to \Omega^{n}_X\otimes L^k\otimes \gr^W_{\ell}\cG_{k,\alpha}(D)][k].
\end{align*}
Once we go beyond the first step in the Hodge filtration, vanishing
theorems for mixed Hodge modules always involve complexes of sheaves (see
\cite{Schnell16} for example). 
\begin{thm}\label{thm: vanishing theorem for higher multiplier ideals}
Let $D$ be an effective divisor on a projective complex manifold $X$. Let $k\in \N,\ell\in \Z$ and $\alpha\in [-1,0]$. Furthermore, let $B$ be an effective divisor such that the $\Q$-divisor $B+\alpha D$ is ample. Then
\[ \mathbb{H}^i\left(X,\gr^W_{\ell}K_{k,\alpha}(D)\otimes \cO_X(B)\right)=0, \quad \textrm{for every } i>0.\]
\end{thm}
This is proved by relating higher multiplier ideals with the notion of \emph{twisted
polarizable Hodge modules}, which will be discussed in the next paragraph.  As usual,
the vanishing theorem improves dramatically when $X$ is an abelian variety (compare
\cite[\S28]{MPHodgeideal}).
\begin{thm}\label{thm: vanishing theorem on abelian varieties}
Let $D$ be an effective divisor on an abelian variety $A$ such that the line bundle $L=\cO_A(D)$ is ample. For any line bundle $\rho\in \mathrm{Pic}^0(A)$ and $i\geq 1$, we have
\begin{enumerate}
    \item $H^i\left(A,L^{k+1}\otimes W_{\ell} \, \cG_{k,\alpha}(D)\otimes \rho\right)=0$ for $k\in \N$, $\ell\in \Z$ and  $\alpha\in (-1,0]$.
    \item $H^i(A,L^k\otimes \cI_{k,0}(D)\otimes \rho)=0$ for $k\geq 1$.
    \item $H^i(A,L^{k+1}\otimes \cI_{k,\alpha}(D)\otimes \rho)=0$ for $k\in \N$ and $\alpha\in [-1,0)$.
\end{enumerate}
\end{thm}
\subsection{Twisted Hodge modules}

We get a handle on the global properties of higher multiplier ideals by relating them
to twisted $\sD$-modules. More precisely, we introduce a notion of \emph{twisted
polarizable Hodge modules}, which provides a convenient framework for discussing
the global structure of nearby and vanishing cycles along effective divisors.

First, let us recall the notion of twisted $\sD$-modules \cite[\S 2]{BB93}. Roughly speaking, these are objects that are locally
$\sD$-modules, but with a different transition rule from one coordinate chart to
another. The twisting depends on two parameters: a holomorphic line bundle $L$ on the
complex manifold $X$ and a complex number $\alpha \in \C$. Like the sheaf of
differential operators $\sD_X$ itself, the sheaf of \emph{$\alpha L$-twisted
differential operators} $\sD_{X,\alpha L}$ is a noncommutative $\cO_X$-algebra;
it still has an order filtration $F_{\bullet}$ such that 
\[ \gr^F_{k}\sD_{X,\alpha L} \cong \Sym^k\sT_X,\]
where $\sT_X$ is the tangent bundle of $X$. In particular, we have $\gr^F_{\bullet}\sD_{X,\alpha L}\cong \gr^F_{\bullet}\sD_X$. But unlike in the case of $\sD_X$ where $F_1\sD_X=\cO_X\oplus \sT_X$, the sequence 
\[ 0\to \cO_X \to F_1\sD_{X,\alpha L} \to \sT_X\to 0\]
does not split; instead, its extension class is equal to $\alpha\cdot c_1(L)$ in $\mathrm{Ext}^1_X(\sT_X,\cO_X)\cong H^1(X,\Omega^1_X)$; see \S \ref{par:tdo} for a construction of $\sD_{X,\alpha L}$ using differential operators on the total space of the line bundle $L$. 
An \emph{$\alpha L$-twisted $\sD$-module} is a \emph{right} module over $\sD_{X,\alpha L}$. One crucial difference with usual $\sD$-modules is that there is no de Rham complex for twisted $\sD$-modules, because there is no longer an action by $\sT_X$. But we do have $\gr^F_1\sD_{X,\alpha L}\cong \sT_X$, and so the ``graded pieces of the de Rham complex", by which we mean the complex
\begin{align*}
	\gr^F_p\DR(\cM)=\Bigl[\gr^F_{p-n}\cM&\otimes \bigwedge^n\sT_X\to \cdots \to \gr^F_{p-1}\cM\otimes \sT_X\to
\gr^F_p\cM\Bigr][n]
\end{align*}
still make sense for a twisted $\sD$-module $\cM$ with a good filtration $F_{\bullet}\cM$.

Claude Sabbah and the first author have been developing a theory of
\emph{complex mixed Hodge modules}, where one removes perverse sheaves from the
picture and describes polarizations as certain distribution-valued pairings on the
underlying $\sD$-modules. In \S \ref{sec:twisted MHM}, we extend this formalism to
\emph{$\alpha L$-twisted polarizable Hodge modules}: for $\alpha\in \mathbb{R}$,
these are filtered $\sD_{X,\alpha L}$-modules with a pairing valued in the sheaf of
$\alpha L$-twisted currents (see also \cite{SV11}). This poses no technical
difficulties, but leads to a very useful new point of view on certain global
questions. (For example, the results are also used in the paper \cite{BakkerSchnell} by
Bakker and the first author to give a new Hodge-theoretic proof of Hwang's theorem
on base manifolds of Lagrangian fibration on irreducible holomorphic symplectic
manifolds.)

We prove a general vanishing theorem of twisted polarizable Hodge modules.
\begin{thm}\label{thm: vanishing theorem for twisted Hodge module}
Let $D$ be an effective divisor on a projective complex manifold $X$ and set $L=\cO_X(D)$. For any $\alpha \in \mathbb{Q}$, let $M$ be an $\alpha L$-twisted Hodge module with strict support $Z$ and let $B$ be an effective divisor on $Z$ such that the $\Q$-divisor $B+\alpha D|_Z$ is ample. Then  
\[ \mathbb{H}^i\left(Z,\gr^F_k\DR(\cM)\otimes \cO_Z(B)\right)=0, \quad \textrm{when $i>0$ and $k\in \Z$}.\]
If $\Omega^1_X$ is trivial, then 
\[ \mathbb{H}^i\left(Z,\gr^F_k\cM\otimes \cO_Z(B)\right)=0, \quad \textrm{when $i>0$ and $k\in \Z$}.\]
\end{thm}
Some time after we posted the first version of this paper on \texttt{arXiv}, we learned that 
Davis and Vilonen have also proved a similar vanishing theorem \cite[Theorem
1.4]{Davisvilonen23} as part of their work on the Schmid-Vilonen conjecture.

\subsection{Higher multiplier ideals and twisted Hodge modules}
\label{sec: construction of HMI introduction}
We now give the construction of higher multiplier ideals and describe how they are related to twisted Hodge modules.
Let $D$ be an effective divisor on a complex manifold $X$.
 Let $L=\cO_X(D)$ be the corresponding holomorphic line bundle and let $s\in
 H^0(X,L)$ be a section with $\mathrm{div}(s)=D$. We view $s$ as a closed embedding
 $s: X\to L$ into the $(n+1)$-dimensional total space of the line bundle, which we
 denote by the same letter $L$. Let $M=s_{\ast}\mathbb{Q}_X^H[n]\in \MHM(L)$ be the direct image of the constant Hodge module on $X$, where $\MHM(L)$ is the category of graded-polarizable mixed Hodge modules on $L$. There are several interesting filtrations on the underlying $\sD$-module of $M$. First, the filtered \emph{right} $\sD$-module underlying $M$ is 
\begin{equation*}\label{eqn: underlying filtered D module for total graph embedding introduction}
    (\cM,F_{\bullet}\cM)=s_{+}(\omega_X,F_{\bullet}\omega_X),
\end{equation*}
where $\omega_X$ is viewed as a right $\sD$-module with $F_{-n}\omega_X=\omega_X$, $F_{-n-1}\omega_X=0$ and $s_{+}$ is the direct image functor for filtered $\sD$-modules. One can show that  $\gr^F_{-n+k}\cM \cong s_{\ast}(\omega_X\otimes L^k)$ for $k\in \N$. On the other hand, one has $V_{\bullet}\cM$, the $V$-filtration relative
to the zero section of $L$. For every $\alpha \in \Q$, the sheaf $\gr^F_{-n+k}V_{\alpha}\cM$ is a coherent sub-$\cO$-module of $\gr^F_{-n+k}\cM$. Then the higher multiplier ideal $\cI_{k,\alpha}(D)$ is defined as a unique coherent ideal sheaf satisfying
\begin{equation}\label{eqn: definition of Higher multiplier ideals in introduction}
    \gr^F_{-n+k}V_{\alpha}\cM = s_{\ast}\left(\omega_X\otimes L^k \otimes \cI_{k,\alpha}(D)\right).
\end{equation} 
Similarly, we define $\cI_{k,<\alpha}(D)$ using $V_{<\alpha}\cM$ (the discreteness of $V$-filtration implies that $\cI_{k,<\alpha}(D)=\cI_{k,\alpha-\epsilon}(D)$ for $0<\epsilon\ll 1$).
Denote by $\gr^V_{\alpha}\cM=V_{\alpha}\cM/V_{<\alpha}\cM$, which is equipped with a monodromy weight filtration $W_{\bullet}(N)\gr^V_{\alpha}\cM$. 
\begin{prop}\label{prop: vanishing cycles are twisted D-modules}
For $\ell\in \Z$ and $\alpha \in [-1,0]$, the pair
\[
\Bigl(\gr^{W(N)}_{\ell}\gr^V_{\alpha}\cM,F_{\bullet+\lfloor\alpha\rfloor}\gr^{W(N)}_{\ell}\gr^V_{\alpha}\cM
\Bigr)\]
is a filtered $\alpha L$-twisted $\sD$-module underlying an $\alpha L$-twisted polarizable Hodge module. \end{prop}
Recall that the weight filtration on $\gr^V_{\alpha}\cM$ as a mixed Hodge module is induced by
\[W_{\bullet}\gr^V_{\alpha}\cM=W_{\bullet-(n-1)}(N)\gr^V_{\alpha}\cM, \quad \alpha\in [-1,0),\quad W_{\bullet}\gr^V_{0}\cM=W_{\bullet-n}(N)\gr^V_{0}\cM,\]
see \eqref{eqn: nearby cycle D modules} and \eqref{eqn: unipotent vanishing cycle D modules} below.  By construction, the weight filtration $W_{\bullet}\cG_{k,\alpha}(D)$ is induced by 
\[ \omega_X\otimes L^k\otimes \gr^W_{\ell}\cG_{k,\alpha}(D)\cong \gr^F_{-n+k}\gr^W_{\ell}\gr^V_{\alpha}\cM.\]
This allows us to relate the higher multiplier ideals with twisted Hodge modules. We
need the theory of complex mixed Hodge modules \cite{SSMHMproject} for two reasons:
twisted $\Dmod$-modules make sense, but there are no ``twisted local systems'' (at
least not as objects on $X$); and, even locally, the individual $\gr_{\alpha}^V \cM$
do not have a rational structure.
\subsection{Application to minimal exponents}
We introduce a generalization of minimal log canonical centers. Let $D$ be an effective divisor on a complex manifold $X$. In birational geometry, the log canonical center is another basic invariant of the pair $(X,D)$. For example, it is a foundational result of Kawamata \cite{Kawamata} that any minimal log canonical center is normal and has at worst rational singularities. By the work of Lichtin \cite{Lichtin} and Koll\'ar \cite{Kollar97}, one has
$\lct(D)=\min\{1,\tilde{\alpha}_D\}$. Therefore it is natural to study the notion of \emph{the centers
of minimal exponent}, which is defined as the subscheme $Y\subseteq X$ such that $\gr^W_{\ell}\cG_{k,\alpha}(D)=\cO_Y$, where $\tilde{\alpha}_D=k-\alpha$ for $k\in \N$, $\alpha\in (-1,0]$ and $\ell$ is the
largest integer such that $\gr^W_{\ell}\cG_{k,\alpha}(D)\neq 0$ (see \S \ref{sec: the
center of minimal exponent}). We use local results about mixed Hodge modules to prove
the following generalization of Kawamata's theorem.

\begin{thm}\label{thm: rationality and normality of the center of minimal exponent article}
Every connected component of the center of minimal exponent of $(X,D)$ is irreducible, reduced, normal and has at worst rational singularities.
\end{thm}

\subsection{Application to geometric Riemann-Schottky problem}
Let $(A,\Theta)$ be a principally polarized abelian variety (p.p.a.v.) of dimension $g\geq 1$ which is indecomposable. By the work of Koll\'ar \cite{KollarShafarevich} and Ein-Lazarsfeld \cite{EL97}, it is known that $\Theta$ is normal with at worst rational singularities. Furthermore,
\begin{equation}\label{eqn: EinLazarsfeld bound}
    \dim \Sing_m(\Theta)\leq g-m-1, \quad \forall m\geq 2,
\end{equation}
where $\dim Z$ means the dimension of the largest component of $Z$ for a scheme $Z$. We are interested in the following conjecture by Casalaina-Martin \cite[Question 4.7]{Casalaina08}.
\begin{conjecture}\label{conjecture: Casalaina Martin}
If $m\geq 2$, then $\dim \Sing_m(\Theta)\leq g-2m+1$.
\end{conjecture}

This conjecture holds for theta divisors on the Jacobians of smooth projective curves (Marten's theorem \cite{ACGH}) and Prym theta divisors associated to etale double covers by Casalaina-Martin \cite{CasalainaMartin09}. In particular, it holds when $\dim A\leq 5$. The only known cases where equality holds in Conjecture \ref{conjecture: Casalaina Martin} are Jacobians of hyperelliptic curves, and intermediate Jacobians of cubic threefolds. This leads to the following stronger conjecture.
\begin{conjecture}\label{conjecture: Casalaina Martin stronger}
If $(A,\Theta)$ is not a hyperelliptic Jacobian or the intermediate Jacobian of a smooth cubic threefold, then $\dim \Sing_m(\Theta)\leq g-2m$, whenever $m\geq 2$.
\end{conjecture}
When $m=2$, it is due to Debarre \cite{Debarre88}. Conjecture \ref{conjecture: Casalaina Martin} implies a conjecture of Grushevsky \cite[Conjecture 5.12]{Grushevsky} saying that if $(A,\Theta)$ is an indecomposable p.p.a.v. and $x\in \Theta$ is any point, then $\mathrm{mult}_x\Theta \leq \frac{g+1}{2}$. This is proved in \cite[Theorem I]{MPHodgeideal} assuming $\Theta$ has only isolated singularities. We provide the first instance of Conjecture \ref{conjecture: Casalaina Martin stronger}.
\begin{thm}\label{thm: partial solution of Casalaina Martin conjecture}
    Assume the center of minimal exponent $Y$ of $(A,\Theta)$ is a one dimension scheme, then Conjecture \ref{conjecture: Casalaina Martin} holds and $Y$ must be a smooth hyperelliptic curve. Moreover, if there exists $m\geq 2$ such that  $\dim \Sing_m(\Theta)=g-2m+1$, then $(A,\Theta)=(\Jac(Y),\Theta_{\Jac(Y)})$. In particular, Conjecture \ref{conjecture: Casalaina Martin stronger} holds in this case. 
    \end{thm}
To have a better understanding of the picture, we also compute the center of minimal exponents for theta divisors in the boundary case of Conjecture \ref{conjecture: Casalaina Martin stronger} as follows. Let $C$ be a hyperelliptic curve of genus $g$ and write $A$ for the Jacobian of $C$. Consider an affine open subset of $A$ over which $\Theta$ is defined by a holomorphic function $f$. Denote by $W^r_{g-1}$ the locus of degree $g-1$ line bundles with at least $r+1$ sections.
\begin{thm}\label{thm: nearby cycle of hyperelliptic curve}
If $\lambda$ is a primitive $(r+1)$-th root of unity, then
$\psi_{f,\lambda}\Q_A[g]$ is pure with strict support equal to the intersection of
$W^r_{g-1}$ with the affine open subset above. Moreover, the center of minimal exponent of $(A,\Theta)$ is $\Theta_{\mathrm{sing}}$.
\end{thm}

 \subsection{Statement in terms of left $\sD$-modules}
 In this work, we work exclusively with \emph{right} $\sD$-modules (because this is more natural where spaces with singularities are involved), but one can also use \emph{left} $\sD$-modules to define higher multiplier ideals, where the notation is more aligned with the classical theory. 

Let us repeat the set-up in the beginning of introduction: assume $D$ is an effective divisor on $X$. Let $L=\cO_X(D)$ be the corresponding holomorphic line bundle, and let $s\in H^0(X,L)$ be a section with $\mathrm{div}(s)=D$, which is also viewed as a closed embedding $s: X\to L$, where $L$ is the total space of line bundle on $L$. Let $M=s_{\ast}\mathbb{Q}_X^H[n]\in \MHM(L)$ 
be the direct image of the constant Hodge module on $X$. Consider the  underlying \emph{left} filtered $\sD_X$-module $(\cM,F_{\bullet}\cM)=s_{+}(\cO_X,F_{\bullet}\cO_X)$, with $F_{0}\cO_X=\cO_X$ and $F_{-1}\cO_X=0$. Let $V^{\bullet}\cM$ be the $V$-filtration of $M$ relative to the zero section of $L$. For each $k\in \N$ and $\beta\in \Q$, we define $\cJ_k(\beta D)$ to be the unique ideal sheaf on $X$ satisfying
\[ s_{\ast}(\cJ_k(\beta D)\otimes \cO_X(kD))=\gr^F_{k}V^{>\beta}\cM. \]
Using the translation rule between left and right $\sD$-modules, we have
 \begin{equation*}\label{eqn: left and right higher multiplier ideals}
 \cJ_k(\beta D)=\cI_{k,<-\beta}(D), \quad \cI_{k,\beta}(D)=\cJ_k((-\beta-\epsilon)D).
\end{equation*}  
On the associated graded level we have $
\cJ_k((\beta-\epsilon)D)/\cJ_k(\beta D) =\cG_{k,-\beta}(D).$
Using this, all results discussed in the earlier part of Introduction can be translated. For example,  the Budur-Saito result translates into $\cJ_0(\alpha D)=\cJ(\alpha D)$.

\subsection{Acknowledgement} 
We thank the following people for helpful discussions: Bradley Dirks,
Lawrence Ein, Sam Grushevsky, J\'anos Koll\'ar, Rob Lazarsfeld, Mircea
Musta\c{t}\u{a}, Sung Gi Park, Stefan Schreieder, Jakub Witaszek, and Ziquan Zhuang. 

During the preparation of this paper, R.Y.~was partially supported by a research
fellowship at the Max-Planck-Institute for Mathematics. Ch.S.~was partially supported
by NSF grant DMS-1551677 and by a Simons Fellowship (award number 817464, Christian
Schnell). Both authors individually thank the Max-Planck-Institute for Mathematics
for providing them with excellent working conditions.

\section{Twisted $\sD$-modules and twisted Hodge modules}\label{sec: twisted D modules}

In this section, we review the theory of twisted $\sD$-modules and
introduce \emph{twisted Hodge modules}, generalizing the theory of polarizable
complex Hodge modules developed by Sabbah and the first author \cite{SSMHMproject} to the
setting of twisted $\sD$-modules. 
\subsection{Local trivializations}

We first explain a convenient way to deal with local trivializations. Let $X$ be a
complex manifold, and $L$ a holomorphic line bundle on $X$. A local trivialization of
$L$ is a pair $(U, \phi)$, where $U \subseteq X$ is an open subset and 
\[
	\phi \colon L \vert_U \to U \times \CC
\]
is an isomorphism of holomorphic line bundles. We can restrict a local trivialization
to any open subset of $U$ in the obvious way. In order to compare different local
trivializations, it is therefore enough to consider local trivializations over the
same open subset. Suppose now that we have two local trivializations $(U, \phi)$ and
$(U, \phi')$. The composition
\[
	\phi' \circ \phi^{-1} \colon U \times \CC \to L \vert_U \to U \times \CC
\]
then has the form $(x,t) \mapsto \bigl( x, g(x) t \bigr)$ for a unique
holomorphic function $g \in \Gamma(U, \shO_U^{\times})$ that is nonzero everywhere on $U$.
This way of thinking about local trivializations eliminates all the unnecessary subscripts.

We can also describe the change of trivialization in terms of sections. Let $s
\in \Gamma(U, L)$ be the nowhere vanishing section of $L$ determined by $(U, \phi)$:
the composition
\[
	\phi \circ s \colon U \to L \vert_U \to U \times \CC
\]
satisfies $(\phi \circ s)(x) = (x, 1)$. Similarly, define $s' \in \Gamma(U, L)$ using
the local trivialization $(U, \phi')$. Then $(\phi' \circ s)(x) = (x, g(x))$ and
$(\phi' \circ s')(x) = (x,1)$, and therefore $s = g \cdot s'$. 

\subsection{The sheaf of twisted differential operators}
\label{par:tdo}

Let $X$ be a complex manifold of dimension $n$, and let $L$ be a holomorphic line
bundle on $X$. We view $L$ as a complex manifold of dimension $n+1$, and denote the bundle
projection by $p \colon L \to X$. On $L$, we have the usual sheaf of differential
operators $\sD_L$. Let $\cI_X \subseteq \sO_L$ be the ideal sheaf of the zero
section. This gives us an increasing filtration
\[
	V_j \sD_L = \menge{P \in \sD_L}{\text{$P \cdot \cI_X^i \subseteq \cI_X^{i-j}$
	for all $i \geq 0$}}.
\]
We are only going to use $V_0 \sD_L$, which consists of those differential operators
that preserve the ideal of the zero section, and the quotient
\[
	\gr_0^V \sD_L = V_0 \sD_L / V_{-1} \sD_L = V_0 \sD_L / \cI_X V_0 \sD_L.
\]
This quotient is naturally a sheaf of $\sO_X$-bimodules. A local trivialization
$(U, \phi)$ for $L$ determines an isomorphism of sheaves of algebras (and
$\sO_U$-bimodules)
\[
	(\phi^{-1})_{\ast} \colon \sD_U[t \partial_t] \to \gr_0^V \sD_L \vert_U.
\]
We denote by $\theta \in \Gamma(U, \gr_0^V \sD_L)$ the image of $t
\partial_t = t \cdot \partial/\partial t$. This is actually a well-defined global
section of $\gr_0^V \sD_L$; the invariant description is as the vector field tangent
to the natural $\C^{\times}$-action on the line bundle $L$. This tells us what $\gr_0^V
\sD_L$ looks like locally.

Let us now try to understand what happens when we change the trivialization. Suppose
that $(U, \phi)$ and $(U, \phi')$ are two local trivializations of $L$, and that we
have local holomorphic coordinates $x_1, \dotsc, x_n$ on the open set $U$. Set
$\partial_j = \partial/\partial x_j$. A short computation using the chain rule shows
that the composition
\[
	\phi'_{\ast} \circ (\phi^{-1})_{\ast} \colon \sD_U[t \partial_t] \to 
	\gr_0^V \sD_L \vert_U \to \sD_U[t \partial_t]
\]
acts on the vector fields $\partial_1, \dotsc, \partial_n, t \partial_t$ in the
following way:
\[
	t \partial_t \mapsto t \partial_t, \quad
	\partial_j \mapsto \partial_j 
		+ g^{-1} \frac{\partial g}{\partial x_j} \cdot t \partial_t
\]

We use this to define the sheaf of twisted differential operators \cite{BB93}.

\begin{definition}
	Let $X$ be a complex manifold, $L$ a holomorphic line bundle on $X$, $\alpha \in \R$ a
	real number. The sheaf of $\alpha L$-twisted differential
	operators on $X$ is the quotient
	\[
		\sD_{X, \alpha L} = \gr_0^V \sD_L \big/ (\theta - \alpha) \gr_0^V \sD_L.
	\]
\end{definition}

In a local trivialization $(U, \phi)$, the sheaf of $\alpha L$-twisted
differential operators is just 
\[
	\sD_{X, \alpha L} \vert_U \cong 
	\sD_U[t \partial_t]/(t \partial_t - \alpha) \sD_U[t \partial_t] \cong \sD_U.
\]
We only see the difference with usual differential operators when we change the
trivialization: if $(U, \phi)$ and $(U, \phi')$ are two local trivializations, then
\[
	\phi'_{\ast} \circ (\phi^{-1})_{\ast} \colon \sD_U \to 
		\sD_{X, \alpha L} \vert_U \to \sD_U
\]
acts on the coordinate vector fields $\partial_1, \dotsc, \partial_n$ as
\begin{equation} \label{eq:vf-twisted}
	\partial_j \mapsto \partial_j 
		+ \alpha \cdot g^{-1} \frac{\partial g}{\partial x_j}.
\end{equation}
This formula is the reason for the name ``twisted'' differential operators. When
$\alpha = 0$, we get back the sheaf $\Dmod_X$ of usual differential operators.

\subsection{Twisted $\sD$-modules}

We keep the notation from above. An $\alpha L$-twisted $\Dmod$-module on $X$ is
simply a module over the sheaf $\sD_{X, \alpha L}$ of $\alpha L$-twisted differential
operators. We generally work with \emph{right} modules (because this is more
appropriate when dealing with singularities). Suppose that $\cM$ is a sheaf
of right $\sD_{X, \alpha L}$-modules. If we have a local trivialization $(U, \phi)$
for $L$, then the isomorphism
\begin{equation} \label{eq:phi-inv-ast}
	(\phi^{-1})_{\ast} \colon \sD_U \to \sD_{X, \alpha L} \vert_U
\end{equation}
gives $\cM \vert_U$ the structure of a usual $\sD_U$-module. A twisted $\sD$-module
therefore looks like a usual $\Dmod$-module in any local trivialization of $L$, but
the action by vector fields changes according to the formula in \eqref{eq:vf-twisted}
from one local trivialization to another. Since it is easy to get confused about the
signs, we give the local formulas. Let $\cM$ be a right module over $\sD_{X, \alpha
L}$. For every local trivialization $(U, \phi)$, we get a right $\sD_U$-module
$\cM_{(U, \phi)}$, whose $\sD$-module structure is defined by the rule
\[
	(m \cdot P)_{(U, \phi)} = m \cdot (\phi^{-1})_{\ast}(P), \quad \textrm{for every  $P\in \Gamma(U,\sD_U)$}.
\]
If $(U, \phi')$ is a second local trivialization, then we obtain
\[
	\left( m \cdot \biggl( \partial_j + \alpha g^{-1} \frac{\partial g}{\partial x_j}
		\biggr) \right)_{(U, \phi')} = m \cdot (\phi'^{-1})_{\ast} \bigl( \phi'_{\ast}
	\circ (\phi^{-1})_{\ast}(\partial_j) \bigr) 
	= \bigl( m \cdot \partial_j \bigr)_{(U, \phi)}.
\]
The $\Dmod$-module structure therefore changes from one local trivialization to
another in accordance with the identity in \eqref{eq:vf-twisted}.

Let us also convince ourselves that tensoring by the line bundle $L$ changes the
twisting parameter in the expected way. Suppose that $\cM$ is a right $\gr_0^V
\sD_L$-module. The tensor product $\cM \otimes L$ is naturally a sheaf of right
modules over $L^{-1} \otimes \gr_0^V \sD_L \otimes L$. If we view $\theta$ as a
morphism $\theta \colon \sO_X \to \gr_0^V \sD_L$, we see that $L^{-1} \otimes \gr_0^V
\sD_L \otimes L$ also has a global section that we denote by the same letter
$\theta$.

\begin{lemma}\label{lemma: canonical isomorphism of grV0DL and L tensor grV0DL}
	There is a canonical isomorphism 
	\[
		\gr_0^V \sD_L \cong L^{-1} \otimes \gr_0^V \sD_L \otimes L
	\]
	that takes the global section $\theta$ on the left-hand side to $\theta+1$. 
\end{lemma}

\begin{proof}
	We give a local proof to show how the formulas work. Let us first work out the
	local description of $L^{-1} \otimes \gr_0^V \sD_L \otimes L$. Let $(U, \phi)$ be
	a local trivialization of $L$, and let $s \in \Gamma(U, L)$ be the resulting
	nowhere vanishing section. Denote by $s^{-1} \in \Gamma(U, L^{-1})$ the induced
	section of the dual line bundle. We get an isomorphism
	\[
		\sD_U[t \partial_t] \to L^{-1} \otimes \gr_0^V \sD_L \otimes L \vert_U, \quad
		P \mapsto s^{-1} \otimes (\phi^{-1})_{\ast}(P) \otimes s.
	\]
	Let $s' \in \Gamma(U, L)$ be the section determined by a second local
	trivialization $(U, \phi')$. This gives us a second isomorphism
	\[
		\sD_U[t \partial_t] \to L^{-1} \otimes \gr_0^V \sD_L \otimes L \vert_U, \quad
		Q \mapsto s'^{-1} \otimes (\phi'^{-1})_{\ast}(Q) \otimes s'.
	\]
	If we compose the first isomorphism with the inverse of the second one, and
	remember that $s = g s'$, we find that the change of trivialization is
	\[
		\sD_U[t \partial_t] \to \sD_U[t \partial_t], \quad
		P \mapsto g^{-1} \cdot \phi'_{\ast} \circ (\phi^{-1})_{\ast}(P) \cdot g.
	\]
	In local coordinates $x_1, \dotsc, x_n$ as above, this acts on the vector fields
	$\partial_1, \dotsc, \partial_n, t \partial_t$ by
	\[
		t \partial_t \mapsto t \partial_t, \quad
		\partial_j \mapsto \partial_j 
		+ g^{-1} \frac{\partial g}{\partial x_j} \cdot (t \partial_t + 1).
	\]
	This shows that the collection of isomorphisms
	\[
		\sD_U[t \partial_t] \to \sD_U[t \partial_t], \quad
			t \partial_t \mapsto t \partial_t + 1, \quad \partial_j \mapsto \partial_j,
	\]
	give us the desired isomorphism between $\gr_0^V \sD_L$ and $L^{-1} \otimes
	\gr_0^V \sD_L \otimes L$.	
\end{proof}

\begin{remark}\label{remark: change the twisting of twisted D module}
    In particular, Lemma \ref{lemma: canonical isomorphism of grV0DL and L tensor grV0DL} says that if $\cM$ is an $\alpha L$-twisted right $\sD$-module,
then the tensor product $\cM \otimes L$ is an $(\alpha+1) L$-twisted right
$\sD$-module. More generally, if $\cM$ is a right $\gr_0^V \sD_L$-module on which the
operator $\theta - \alpha$ acts nilpotently, then $\cM \otimes L$ is again a right
$\gr_0^V \sD_L$-module on which $\theta - (\alpha+1)$
acts nilpotently. Note that this only works nicely for right modules: if $\cM$ is a
left module over $\Dmod_{X, \alpha L}$-module, then $L \otimes \cM$ becomes a left
module over $\Dmod_{X, (\alpha -1)L}$. This is one of many reasons why it is better
to work with right $\sD_{X,\alpha L}$-modules.  
\end{remark}

\subsection{Twisted currents}

We also need a notion of twisted currents, in order to define hermitian pairings on
twisted $\sD$-modules. We first introduce the space of twisted test functions. These
are compactly supported sections of a certain smooth line bundle that we now describe.
Fix a real number $\alpha \in \R$. The principal $\C^{\times}$-bundle corresponding to
the holomorphic line bundle $L$ is obtained by removing the zero section from $L$.
Let $L_{\alpha}$ be the smooth line bundle associated to the representation
\[
	\C^{\times} \to \GL_1(\C), \quad z \mapsto \abs{z}^{2 \alpha}.
\]
We can also describe $L_{\alpha}$ in terms of local trivializations. A local
trivialization $(U, \phi)$ of $L$ determines an isomorphism
\[
	\phi_{\alpha} \colon L_{\alpha} \vert_U \to U \times \C.
\]
Let $(U, \phi')$ be a second local trivialization, and let $g \in \Gamma(U,
\sO_U^{\times})$ be the unique nowhere vanishing holomorphic function such that
$(\phi' \circ \phi^{-1})(x) = (x, g(x)t)$. Then we have
\[
	\phi'_{\alpha} \circ \phi_{\alpha}^{-1} \colon U \times \C \to L_{\alpha} \vert_U
	\to U \times \C, \quad (x, t) \mapsto \Bigl( x, \abs{g(x)}^{2 \alpha} t \Bigr).
\]
For example, a global section of $L_{\alpha}$ is the same thing as a collection of
smooth functions $\varphi_{(U, \phi)} \colon U \to \CC$, one for each local trivialization
$(U, \phi)$, such that $\varphi_{(U, \phi')} = \abs{g}^{2 \alpha} \varphi_{(U,
\phi)}$. The usual action by differential operators on smooth functions makes
$L_{\alpha}$ into a left bimodule over the sheaf $\Dmod_{X, \alpha L}$ of twisted
differential operators and its conjugate.

An $\alpha L$-twisted test function is a compactly supported smooth section of the
smooth line bundle $L_{\alpha}$. We give this space the topology that agrees with
the usual topology on the space of compactly supported smooth functions in any local
trivialization of $L$. 

\begin{definition}
	An \emph{$\alpha L$-twisted current} is a continuous linear functional on the
	space of $(-\alpha L)$-twisted test functions. We denote by $\Cur_{X, \alpha L}$ the
	sheaf of $\alpha L$-twisted currents.
\end{definition}

Here is a more concrete description. A $(-\alpha L)$-twisted test function $\varphi \in
\Gamma_c(X, L_{\alpha})$ is the same thing as a collection of smooth functions
$\varphi_{(U, \phi)} \colon U \to \CC$, one for each local trivialization $(U, \phi)$
of the line bundle $L$, that are compatible with restriction to open subsets, and are
related to each other by the formula
\[
	\varphi_{(U, \phi')} = \abs{g}^{-2\alpha} \cdot \varphi_{(U, \phi)},
\]
where $(\phi' \circ \phi^{-1})(x,t) = (x, g(x) t)$ is the transition from one local
trivialization to the other. Of course, the union of the supports of all the
functions $\varphi_{(U, \phi)}$ must be a compact subset of $X$. 
Dually, an $\alpha L$-twisted current $C \in \Gamma(X, \Cur_{X, \alpha L})$ is the
same thing as a collection of currents $C_{(U, \phi)} \in \Gamma(U, \Cur_X)$ that are
compatible with restriction, and are related to each other by the formula
\[
	C_{(U, \phi')} = C_{(U, \phi)} \cdot \abs{g}^{2\alpha}.
\]

Let us return to the general properties of twisted currents.  We denote the natural
pairing between twisted currents and twisted test functions by the symbol
\[
	\langle C, \varphi \rangle \in \C,
\]
for $C \in \Gamma(U, \Cur_{X, \alpha L})$ a twisted current and $\varphi \in
\Gamma_c(U, L_{-\alpha})$ a twisted test function. As usual, operations on twisted
currents are defined in terms of the corresponding operations on twisted test
functions. For example, the complex conjugate of a twisted current is defined by the
formula
\[
	\bigl\langle \overline{C}, \varphi \bigr\rangle 
	= \overline{\langle C, \overline{\varphi} \rangle}.
\]
The sheaf of $\alpha L$-twisted currents has the structure of a right $\sD_{X, \alpha
L}$-module. This can be seen as follows. Each local trivialization $(U, \phi)$ for
$L$ determines an isomorphism
\[
	\Cur_U \cong \Cur_{X, \alpha L} \vert_U,
\]
where $\Cur_U$ is the sheaf of currents on $U$ and is of course a right $\sD_U$-module. Moreover, the
transition from one local trivialization to another works correctly. Indeed, if we
have a twisted current $C$, represented by a collection of currents $C_{(U,
\phi)} \in \Gamma(U, \Cur_X)$ such that
\[
	C_{(U, \phi')} = C_{(U, \phi)} \abs{g}^{2 \alpha},
\]
then a brief computation shows that
\[
	\bigl( C_{(U, \phi)} \partial_j \bigr) \abs{g}^{2 \alpha}
	= C_{(U, \phi)} \abs{g}^{2 \alpha} \left( \partial_j + \alpha g^{-1}
	\frac{\partial g}{\partial x_j} \right)
	= C_{(U, \phi')} \left( \partial_j + \alpha g^{-1} 
		\frac{\partial g}{\partial x_j} \right).
\]
This proves that $\Cur_{X, \alpha L}$ is an $\alpha L$-twisted right $\sD$-module.

\subsection{Flat hermitian pairings on twisted $\sD$-modules}

Fix a real number $\alpha \in \R$. A flat hermitian pairing on an $\alpha L$-twisted
right $\sD$-module $\cM$ is a morphism of sheaves
\[
	S \colon \cM \otimes_{\C} \wbar{\cM} \to \Cur_{X, \alpha L}
\]
with the following properties. First, $S$ is hermitian symmetric, in the sense
that for any two local sections $m',m'' \in \Gamma(U, \cM)$, one has
\[
	S(m'',m') = \overline{S(m',m'')}
\]
as twisted currents on $U$. Second, $S$ is $\sD_{X, \alpha L}$-linear in its
first argument, meaning that
\[
	S(m' P, m'') = S(m', m'') P
\]
for every twisted differential operator $P \in \Gamma(U, \sD_{X, \alpha L})$. It
follows that $S$ is conjugate $\sD_{X, \alpha L}$-linear in its second argument. In
any local trivialization of $L$, the twisted $\sD$-module $\cM$ becomes a usual
$\sD$-module, and the flat hermitian pairing $S$ becomes a sesquilinear pairing as in
\cite[Ch.~12]{SSMHMproject}.

\subsection{Good filtrations}\label{sec: good filtrations}

The sheaf $\gr_0^V \sD_L$ inherits an increasing filtration $F_{\bullet} \gr_0^V
\sD_L$ from the order filtration on $\sD_L$. Locally, this is just the usual order
filtration on differential operators. Indeed, if $(U, \phi)$ is a local trivialization
for $L$, then under the isomorphism
\[
	(\phi^{-1})_{\ast} \colon \sD_U[t \partial_t] \to \gr_0^V \sD_L \vert_U,
\]
the filtration $F_{\bullet} \gr_0^V \sD_L \vert_U$ is just the order
filtration on $\sD_U[t \partial_t]$, with $t \partial_t$ being considered as a
differential operator of order $1$. Globally, the first nonzero piece of our
filtration is $F_0 \gr_0^V \sD_L \cong \sO_X$; the next graded piece $\gr_1^F \gr_0^V
\sD_L$ sits in a short exact sequence
\[
	\begin{tikzcd}[column sep=small] 
		0 \rar & \sO_X \rar{\theta} & \gr_1^F \gr_0^V \sD_L \rar & \sT_X \rar & 0,
	\end{tikzcd}
\]
whose extension class in $\Ext_{\sO_X}^1(\sT_X, \sO_X) \cong H^1(X, \Omega_X^1)$ is
the first Chern class $c_1(L)$. The order filtration on $\gr_0^V \sD_L$ induces a
filtration on $\sD_{X, \alpha L} = \gr_0^V \sD_L / (\theta - \alpha) \gr_0^V \sD_L$,
and from the above, we get a short exact sequence
\begin{equation} \label{eq:extension-alpha}
	\begin{tikzcd}[column sep=small]
		0 \rar & \sO_X \rar & F_1 \sD_{X, \alpha L} \rar & \sT_X \rar & 0,
	\end{tikzcd}
\end{equation} 
whose extension class is now $\alpha \cdot c_1(L)$. Note that
\[
	\gr_{\bullet}^F \sD_{X, \alpha L} \cong \gr_{\bullet}^F \sD_X \cong
	\Sym^{\bullet}\sT_X
\]
is isomorphic to the symmetric algebra on the tangent sheaf $\sT_X$, just as in the
untwisted case. This is a consequence of \eqref{eq:vf-twisted}.

Good filtrations on twisted $\sD$-modules are defined just as in the untwisted case.
Let $\cM$ be a right $\sD_{X, \alpha L}$-module, and let $F_{\bullet}
\cM$ be an exhaustive increasing filtration by coherent $\sO_X$-submodules such that,
locally on $X$, one has $F_k \cM = 0$ for $k \ll 0$. We say that such a filtration
$F_{\bullet} \cM$ is a good filtration if 
\[
	F_k \cM \cdot F_{\ell} \sD_{X, \alpha L} \subseteq F_{k+\ell} \cM,
\]
with equality for $k \gg 0$. As usual, this is equivalent to the condition that 
\[
	\gr_{\bullet}^F \cM = \bigoplus_{k \in \ZZ} F_k \cM/F_{k-1} \cM
\]
is coherent over $\gr_{\bullet}^F \sD_{X, \alpha L} \cong \Sym^{\bullet} \sT_X$.

\subsection{Graded pieces of the de Rham complex}\label{sec: graded pieces of de Rham}

One important difference between twisted $\sD$-modules and usual $\sD$-modules is
that there is no de Rham complex for twisted $\sD$-modules (unless $\alpha=0$),
because there is no longer an action by $\sT_X$. This is due to the fact that the
short exact sequence in \eqref{eq:extension-alpha} does not split (unless $\alpha =
0$). But we do have $\gr^F_1\sD_{X,\alpha L}\cong \sT_X$, so the notion of ``graded
pieces of the de Rham complex" still makes sense for a twisted $\sD$-module $\cM$
with a good filtration $F_{\bullet} \cM$.

\begin{definition}\label{definition: associated graded de Rham}
    Let $(\cM,F_{\bullet}\cM)$ be a twisted $\sD$-module on $X$ with a good filtration. For every $k\in \Z$, the \emph{graded piece of the de Rham complex} is defined by
\begin{equation}\label{eqn: graded pieces of the de Rham complex}
    \gr^F_k\DR(\cM)\colonequals \left[\gr^F_{k-n}\cM\otimes \bigwedge^n \sT_X\to \cdots\to \gr^F_{k-1}\cM\otimes \sT_X \to \gr^F_k\cM\right][n],
\end{equation} 
where $n=\dim X$ and where the differential is induced by the multiplication morphism $\gr_k^F \cM \otimes
\gr_1^F \sD_{X, \alpha L} \to \gr_{k+1}^F \cM$ and the isomorphism $\gr_1^F \sD_{X,
\alpha L} \cong \sT_X$.
\end{definition}

\subsection{Twisted Hodge modules}\label{sec:twisted MHM}

Before we can define twisted Hodge modules, which are the main objects in this
chapter, we briefly review the theory of complex Hodge modules from
\cite[\S14]{SSMHMproject}. We will give a simplified version of the definition (without
weights) that is sufficient for our purposes. 

Let $X$ be a complex manifold of dimension $n$. A complex Hodge module on $X$
consists of the following three pieces of data:
\begin{enumerate}
	\item A regular holonomic right $\sD_X$-module $\cM$.
	\item An increasing filtration $F_{\bullet} \cM$ by coherent $\sO_X$-submodules.
		This filtration needs to be good, which means that it is exhaustive; that $F_k
		\cM = 0$ for $k \ll 0$ locally on $X$; and that one has $F_k \cM \cdot F_{\ell}
		\sD_X \subseteq F_{k+\ell} \cM$, with equality for $k \gg 0$.
	\item A flat hermitian pairing $S \colon \cM \otimes_{\C} \overline{\cM} \to
		\Cur_X$, valued in the sheaf of currents of bidegree $(n,n)$ on $X$. Again, $S$
		needs to be hermitian symmetric and $\sD_X$-linear in its first argument (and
		therefore conjugate linear in its second argument).
\end{enumerate}
The object $M = (\cM, F_{\bullet} \cM, S)$ is a polarized Hodge module if it satisfies
several additional conditions that are imposed on the nearby and vanishing cycle
functors with respect to holomorphic functions on open subsets of $X$. An
important point is that the definition is local: if the restriction of $M$ to every
subset in an open covering of $X$ is a polarized Hodge module on that open subset,
then $M$ is a polarized Hodge module on $X$.

\begin{remark}
Note that there are no weights in the simplified definition above. In order to have
an intrinsic notion of weights, one needs to work with triples of the form
\[
	\bigl( (\cM', F_{\bullet} \cM'), (\cM'', F_{\bullet} \cM''), S \bigr),
\]
where $\cM'$ and $\cM''$ are regular holonomic right $\sD_X$-modules with good
filtrations $F_{\bullet} \cM'$ and $F_{\bullet} \cM''$, and where $S \colon \cM'
\otimes_{\C} \cM'' \to \Cur_X$ is a flat sesquilinear pairing. In this formulation, a
polarization is then a certain kind of isomorphism between $\cM'$ and $\cM''$ that is
compatible with the filtrations and the pairing \cite[Ch.~14]{SSMHMproject}.
\end{remark}

We now come to the main definition of this chapter. It is modelled on the definition
of polarized complex Hodge modules \cite[\S14.2]{SSMHMproject}, but with twisted
$\sD$-modules in place of usual ones. This works because being a polarized complex
Hodge module is a local condition.

Let $X$ be a complex manifold, $L$ a holomorphic line bundle on $X$, and $\alpha \in
\R$ a real number. We again consider objects of the type $(\cM, F_{\bullet} \cM, S)$,
where $\cM$ is a right $\sD_{X, \alpha L}$-module with a good filtration $F_{\bullet}
\cM$, and where 
\[
	S \colon \cM \otimes_{\C} \wbar{\cM} \to \Cur_{X, \alpha L}
\]
is a flat hermitian pairing on $\cM$. If $(U, \phi)$ is a local trivialization of the
line bundle $L$, the restriction $(\cM, F_{\bullet} \cM, S) \vert_U$ to the open
subset $U$ becomes, via the isomorphism in \eqref{eq:phi-inv-ast}, a usual filtered
$\sD_U$-module with a flat hermitian pairing. 

\begin{definition}
	We say that an object $(\cM, F_{\bullet} \cM, S)$ is an \emph{$\alpha L$-twisted
	polarized Hodge module} if, for any local trivialization $(U, \phi)$,
	the restriction $(\cM, F_{\bullet} \cM, S) \vert_U$ is a polarized complex Hodge
	module in the usual sense.
\end{definition}

Because of the local nature of the definition, all local properties of polarized
complex Hodge modules (such as existence of a decomposition by strict support or the
strictness of morphisms) immediately carry over to the twisted setting
\cite[\S14.2]{SSMHMproject}. 

\subsection{Direct images and the decomposition theorem}\label{sec: direct images of twisted D modules}

Another important difference between twisted $\sD$-modules and usual $\sD$-modules is
that one cannot take the direct image of a twisted $\sD$-module (or twisted Hodge module)
along a proper morphism $f \colon X \to Y$ unless the line bundle $L$ is pulled back
from $Y$.

Let us start with a few remarks about the direct image functor for twisted
$\sD$-modules. Let $f \colon X \to Y$ be a proper holomorphic mapping between complex
manifolds, and let $L$ be a holomorphic line bundle on $Y$, viewed as a complex
manifold of dimension $\dim Y + 1$ via the bundle projection $p \colon L \to Y$. Let
$L_X = X \times_Y L$ be the pullback of the line bundle to $X$, as in the commutative
diagram below.
\[
	\begin{tikzcd}
		L_X \rar \dar & L \dar{p} \\
		X \rar{f} & Y
	\end{tikzcd}
\]
Let $\alpha \in \R$. As in the untwisted case \cite[8.6.4]{SSMHMproject}, we
introduce the transfer module 
\[
	\sD_{X \to Y, \alpha L}
	= \sO_X \otimes_{f^{-1} \sO_Y} f^{-1} \sD_{Y, \alpha L}.
\]
This is an $(\sD_{X, \alpha L_X}, f^{-1} \sD_{Y, \alpha L})$-bimodule: the right
$f^{-1} \sD_{Y, \alpha L}$-module structure is the obvious one, and the left
$\sD_{X, \alpha L_X}$-module structure is induced by the morphism
\[
	F_1 \sD_{X, \alpha L_X} \to f^{\ast} F_1 \sD_{Y, \alpha L}
	= \sO_X \otimes_{f^{-1} \sO_Y} f^{-1} F_1 \sD_{Y, \alpha L},
\]
which is part of the following commutative diagram:
\[
	\begin{tikzcd}[column sep=small] 
		0 \rar& \sO_X \dar[equal] \rar& F_1 \sD_{X, \alpha L_X} \dar \rar& \sT_X \dar \rar& 0 \\
		0 \rar& \sO_X \rar& f^{\ast} F_1 \sD_{Y, \alpha L} \rar& f^{\ast} \sT_Y \rar& 0
	\end{tikzcd}
\]
We can then define the direct image functor
\[
	f_+ \colon \Dbcoh(\sD_{X, \alpha L_X}) \to \Dbcoh(\sD_{Y, \alpha L})
\]
from the derived category of right $\sD_{X, \alpha L_X}$-modules to that of right
$\sD_{Y, \alpha L}$-modules as
\[
	f_+(\argbl) = \derR \fl \bigl( \argbl \otimes_{\sD_{X, \alpha L_X}}^{\mathbf{L}} \sD_{X \to Y,
	\alpha L} \bigr),
\]
which is essentially the same formula as in the untwisted case
\cite[\S8.7]{SSMHMproject}. If $\phi \colon L
\vert_U \to U \times \CC$ is a local trivialization of $L$, then we get an induced
trivialization of $L_X$ over the open subset $f^{-1}(U)$, and so twisted
$\sD$-modules on $U$ and $f^{-1}(U)$ are the same thing as usual $\sD$-modules. It is
then easy to see that the diagram
\[
	\begin{tikzcd}
		\Dbcoh(\sD_{X, \alpha L_X}) \dar{f_+} \rar & \Dbcoh(\sD_{f^{-1}(U)}) \dar{f_+} \\
		\Dbcoh(\sD_{Y, \alpha L}) \rar & \Dbcoh(\sD_U)
	\end{tikzcd}
\]
is commutative, where the horizontal arrows are restriction to the open subsets $U$
and $f^{-1}(U)$ and $f_+ \colon \Dbcoh(\sD_{f^{-1}(U)}) \to \Dbcoh(\sD_U)$ is the
usual direct image functor for right $\sD$-modules. By the same method as in
\cite[\S8.7]{SSMHMproject}, the definition of the direct image functor can be extended to
filtered $\sD$-modules (using the natural filtration on the transfer module), and as
in \cite[\S12]{SSMHMproject}, a flat hermitian pairing $S \colon \cM \otimes_{\C}
\overline{\cM} \to \Cur_{X, \alpha L_X}$ on an $\alpha L_X$-twisted $\sD$-module
induces flat sesquilinear pairings
\[
	S_i \colon \sH^i f_{+} \cM \otimes_{\C} \overline{\sH^{-i} f_{+} \cM} \to
	\Cur_{Y, \alpha L}.
\]
We can now state the decomposition theorem for twisted Hodge modules.

\begin{thm}\label{thm: decomposition for twisted Hodge modules}
	Let $f \colon X \to Y$ be a projective morphism between complex manifolds,
	let $L$ be a holomorphic line bundle on $Y$, and set $L_X = f^{\ast} L$. If $(\cM,
	F_{\bullet} \cM, S)$ is an $\alpha L_X$-twisted polarized Hodge module on $X$, then each
	\[
		\sH^i f_{+}(\cM, F_{\bullet} \cM)
	\]
	with the induced polarization, is an $\alpha L$-twisted polarized Hodge module on
	$Y$. Moreover, the decomposition theorem
	\[
		f_{+}(\cM, F_{\bullet} \cM) \cong \bigoplus_{i \in \ZZ} \sH^i f_{+}(\cM,
		F_{\bullet} \cM)
	\]
	holds in the derived category of filtered twisted $\sD_Y$-modules.
\end{thm}

\begin{proof}
	Locally on $Y$, the direct image functor for twisted $\sD$-modules agrees with the
	usual direct image functor for $\sD$-modules. All the local assertions in the
	theorem therefore follow from \cite[\S14.3]{SSMHMproject}, and in particular, each
	$\sH^i f_{+}(\cM, F_{\bullet} \cM)$ is strict. Let $\omega$ be the first Chern
	class of a relatively ample line bundle. From the relative Hard Lefschetz theorem
	for complex Hodge modules, applied locally, it follows that
	\[
		\omega^i \colon \sH^{-i} f_{+}(\cM, F_{\bullet} \cM) \to \sH^i f_{+}(\cM,
		F_{\bullet-i} \cM)
	\]
	is an isomorphism for every $i \geq 0$. Just as in the untwisted case, this
	implies the decomposition theorem. The relative Hard Lefschetz theorem also gives
	us a representation of the Lie algebra $\sltwo$ on the direct sum of all the
	$\sH^i f_{+} \cM$, and we again let $\wsl$ denote the corresponding Weil element.
	Then \cite[\S14.3]{SSMHMproject}, applied locally, shows that the flat hermitian
	pairing
	\[
		(-1)^{i(i-1)/2} S_i \circ (\id \otimes \wsl) \colon
			\sH^i f_{+} \cM \otimes_{\C} \overline{\sH^i f_{+} \cM} 
			\to \Cur_{Y, \alpha L}
	\]
	polarizes $\sH^i f_{+}(\cM, F_{\bullet} \cM)$, which is therefore a twisted
	polarized Hodge module.
\end{proof}

\subsection{Kashiwara's equivalence}

Let $i \colon Y \into X$ be the inclusion of a closed submanifold. Let $L$ be a holomorphic line bundle on $X$ and set $L_Y=i^{\ast}L$. For any $\alpha\in
\R$, there is a direct image functor (see \S \ref{sec: direct images of twisted D modules}) 
\[ 
	i_{+}: \{ \text{$\alpha L_Y$-twisted Hodge modules on $Y$}  \} \to 
		\{ \text{$\alpha L$-twisted Hodge modules on $X$}  \}.
\]

\begin{thm}\label{thm: twisted Kashiwara equivalence for Hodge modules}
The functor $i_{+}$ induces an equivalence between the category of $\alpha
L_Y$-twisted polarized Hodge modules on $Y$ and the category of $\alpha L$-twisted polarized
Hodge modules on $X$ whose support is contained in the submanifold $Y$.
\end{thm}

\begin{proof}
	 This follows from the twisted Kashiwara's equivalence  \cite[\S 2.5.5(iv)]{BB93} and the strict Kashiwara's equivalence for complex polarized Hodge
	 modules \cite[Proposition 9.6.2]{SSMHMproject}.  
\end{proof}

\subsection{Twisted variations of Hodge structure}
\label{sec: twisted structure}

Twisted Hodge modules are closely related to twisted variations of Hodge structure.
Our discussion here will be very brief, because everything follows in the expected
way from the local theory for usual polarized Hodge modules. Let $X$ be a complex
manifold, $L$ a holomorphic line bundle, and $\alpha \in \RR$ a real number. Recall
that the extension class of
\[
	\begin{tikzcd}[column sep=small] 
		0 \rar & \OX \rar& F_1 \Dmod_{X, \alpha L} \rar{p} & \shT_X \rar & 0
	\end{tikzcd}
\]
is $\alpha \cdot c_1(L) \in H^1(X, \OmX^1)$. In the sheaf of $\OX$-algebras
$\Dmod_{X, \alpha L}$, the commutator between a local section $f \in \OX$ and a local
section $\xi \in F_1 \Dmod_{X, \alpha L}$ is 
\[
	[\xi, f] = p(\xi) f,
\]
where the right-hand side means the action of the vector field $p(\xi) \in \shT_X$
on the function $f \in \OX$. Similarly, the commutator of two local sections $\xi,
\eta \in F_1 \Dmod_{X, \alpha L}$ is another local section $[\xi, \eta] \in F_1
\Dmod_{X, \alpha L}$, and this defines a bracket operation
\[
	\lbrack \argbl, \argbl \rbrack \colon F_1 \Dmod_{X, \alpha L} \tensor F_1
	\Dmod_{X, \alpha L} \to F_1 \Dmod_{X, \alpha L}
\]
that restricts to the usual Lie bracket on $\shT_U$ in any local trivialization $(U,
\phi)$ of $L$. We use this to define twisted flat connections. If $\shE$ is an
$\OX$-module, we use the right $\OX$-module structure on $F_1 \Dmod_{X, \alpha L}$ to
define the tensor product
\[
	F_1 \Dmod_{X, \alpha L} \tensor_{\OX} \shE;
\]
it becomes an $\OX$-module using the left $\OX$-module structure on $F_1 \Dmod_{X,
\alpha L}$. 

\begin{definition}
Let $\shE$ be a locally free sheaf of $\OX$-modules. An $\alpha L$-twisted flat
connection on $\shE$ is an $\OX$-linear morphism
\[
	\nabla \colon F_1 \Dmod_{X, -\alpha L} \tensor_{\OX} \shE \to \shE,
	\quad \xi \tensor s \mapsto \nabla_{\xi} s,
\]
that extends the natural morphism $\OX \tensor_{\OX} \shE \to \shE$ and satisfies the
identity
\[
	\nabla_{\xi} \nabla_{\eta} s - \nabla_{\eta} \nabla_{\xi} s = \nabla_{[\xi,\eta]} s
\]
for arbitrary local sections $s \in \shE$, and $\xi, \eta \in F_1 \Dmod_{X, -\alpha L}$.
\end{definition}

Linearity over $\OX$ means that we have
\[
	\nabla_{\xi \cdot f}(s) = \nabla_{\xi}(fs) \quad \text{and}
	\quad \nabla_{f \xi}(s) = f \cdot \nabla_{\xi} s
\]
for local sections $f \in \OX$, $s \in \shE$, and $\xi \in F_1 \Dmod_{X, -\alpha L}$.
This implies the usual Leibniz rule
\[
	\nabla_{\xi}(fs) - f \cdot \nabla_{\xi} s = [\xi, f] \cdot s.
\]
The reason for the minus sign in the definition is explained by the following 
lemma.

\begin{lemma}
	Let $\nabla$ be an $\alpha L$-twisted flat connection on a locally free
	$\OX$-module $\shE$. Then the tensor product $\omega_X \tensor_{\OX} \shE$ is
	naturally a right module over $\Dmod_{X, \alpha L}$. 
\end{lemma}

\begin{proof}
	Let $(U, \phi)$ be a local trivialization, and let $x_1, \dotsc, x_n$ be local
	coordinates. Just as in the untwisted case, we define
	\[
		\bigl( \dx_1 \wedge \dotsm \wedge \dx_n \tensor s \bigr) \cdot \partial_j
		= -\dx_1 \wedge \dotsm \wedge \dx_n \tensor \nabla_{\partial_j} s
	\]
	for any local section $s \in \shE$. It is then an easy matter to check that this
	puts a well-defined right $\Dmod_{X, \alpha L}$-module structure on $\omega_X
	\tensor \shE$. It is the minus sign in the conversion from left to right
	$\Dmod$-modules that is responsible for the change from $-\alpha$ to $\alpha$.
\end{proof}

A right $\Dmod_X$-module comes from a vector bundle with flat connection if and only if
it is coherent as an $\OX$-module. By looking at local trivializations, this
immediately implies the analogous result for twisted $\Dmod$-modules.

\begin{lemma}
	Let $\Mmod$ be a right $\Dmod_{X, \alpha L}$-module. Then $\Mmod$ is coherent over
	$\OX$ if and only if $\Mmod \cong \omega_X \tensor \shE$ for a locally free
	$\OX$-module $\shE$ with an $\alpha L$-twisted flat connection.
\end{lemma}

An $\alpha L$-twisted polarized variation of Hodge structure consists of a locally
free sheaf $\shE$ with an $\alpha L$-twisted flat connection $\nabla$; a finite
decreasing Hodge filtration $F^{\bullet} \shE$ such that
\[
	\nabla(F_1 \Dmod_{X, -\alpha L} \tensor F^p \shE) \subseteq F^{p-1} \shE,
\]
and an $\alpha L$-twisted flat pairing
\[
	S \colon \shE \tensor \wbar{\shE} \to L_{\alpha}
\]
that is hermitian symmetric and satisfies $S \bigl( \nabla_{\xi} s', s'' \bigr) = \xi
\cdot S(s', s'')$ for arbitrary local sections $\xi \in F_1 \Dmod_{X, -\alpha L}$ and
$s',s'' \in \shE$. At each point $x \in X$, this data needs to give a polarized Hodge
structure on the complex vector space $\shE_x$. This definition only makes sense for complex
variations of Hodge structure, because there are no ``twisted local systems'' (at
least not as objects on $X$). We define the associated twisted
Hodge module $(\Mmod, F_{\bullet} \Mmod, S)$ as in the untwisted case: the twisted
$\Dmod$-module is $\Mmod = \omega_X \tensor \shE$; the Hodge filtration is $F_k \Mmod
= \omega_X \tensor F^{-k-n} \shE$; and the polarization is defined by the rule
\[
	\bigl\langle S(\omega' \tensor s', \omega'' \tensor s''), \varphi \bigr\rangle
	= \frac{(-1)^{n(n+1)/2}}{(2 \pi i)^n} \int_X \varphi S(s',s'') \cdot \omega' \wedge
	\overline{\omega''}
\]
on any $(-\alpha L)$-twisted test function $\varphi$. Note that the twisting in
$\varphi$ and $S(s',s'')$ cancels out and so the integrand is a well-defined
compactly supported $(n,n)$-form.

Let $M = (\Mmod, F_{\bullet} \Mmod, S)$ be a twisted Hodge
module on a complex manifold $X$. We say that it is \emph{smooth} if the twisted
$\Dmod$-module $\Mmod$ is coherent over $\OX$. This implies that $\Mmod \cong
\omega_X \tensor \shE$, where $\shE$ is a locally free $\OX$-module with an $\alpha
L$-twisted flat connection $\nabla$. It then follows from the usual local theory for
polarized Hodge modules that $M$ is the twisted Hodge module associated to a twisted
polarized variation of Hodge structure on $\shE$, by the construction in the previous
paragraph.

\subsection{Structure theorem and generic pullback}
\label{sec: generic pullback}

We say that a twisted Hodge module has strict support equal to an irreducible
analytic subset $Z \subseteq X$ if the support of any nontrivial subquotient
is equal to $Z$. We have the following analogue of Saito's structure theorem
\cite[Ch.~16]{SSMHMproject} for objects with strict support.

\begin{thm} \label{thm: twisted structure}
	Any $\alpha L$-twisted polarized variation of Hodge structure defined on a dense
	Zariski-open subset of $Z \subseteq X$ extends uniquely to an $\alpha L$-twisted
	polarized Hodge module on $X$ with strict support equal to $Z$. Moreover, every
	$\alpha L$-twisted Hodge module with strict support equal to $Z$ arises in this
	way.
\end{thm}

\begin{proof}
	This is a local problem, and so the result follows from the structure theorem for
	polarized Hodge modules \cite[Thm.~16.2.1]{SSMHMproject}.
\end{proof}

There is also a generic pullback for twisted Hodge modules, similar to the untwisted
case. Let $f \colon X \to Y$ be a holomorphic mapping between complex manifolds, and
let $L$ be a holomorphic line bundle on $Y$. We denote by $L_X = \fu L$ its pullback
to $X$. Let $(\Mmod, F_{\bullet} \Mmod, S)$ be an $\alpha L$-twisted Hodge module on
$Y$ with strict support equal to $Y$. Let $Y_0 \subseteq Y$ denote the maximal open
subset over which $\Mmod$ is smooth, meaning coherent as an $\OY$-module; according
to the discussion in the previous section, the
restriction to $Y_0$ is then a twisted polarized variation of Hodge structure $(\shE,
F^{\bullet} \shE, S)$. Set $X_0 = f^{-1}(Y_0)$, and denote by $f_0 \colon X_0 \to
Y_0$ the restriction of $f$.  The pullback of $\shE$ along $f_0$ is naturally a twisted
variation of Hodge structure on $X_0$; by Theorem~\ref{thm: twisted structure}, it
extends uniquely to a twisted Hodge module on $X$ with strict support equal to $X$.
We call this object the generic pullback of the original twisted Hodge module.

\subsection{Kashiwara-Malgrange $V$-filtration}\label{sec: V filtrations}

Before the discussion of nearby and vanishing cycles in the next section, let us  fix the convention for the $V$-filtration. Let $f$ be a holomorphic function on a complex manifold $X$ and let $\cM$ be a regular holonomic $\sD$-module on $X$. Consider the graph embedding $i_f: X \to X\times \C$ with $i_f(x)=(x,f(x))$ and the right $\sD$-module $\cM_f\colonequals(i_f)_{+}\cM$. Let $t$ be the coordinate on $\C$, and $\d_t=\d/\d_t$ the corresponding vector field. As above, inside $\sD_{X\times \C}$, we have the subsheaf $V_{0}\sD_{X\times \C}$ of those differential operators that preserve the ideal sheaf of $X\times \{0\}$; it is generated by $\sD_X$ and the additional two operators $t$ and $t\d_t$.

\begin{definition}\label{definition: V-filtration}
The \emph{Kashiwara-Malgrange $V$-filtration} on $\cM_f$ is an increasing filtration $V_{\bullet}\cM_f$, indexed by $\Q$, with the following properties: (1) each $V_{\alpha}\cM_f$ is coherent over $V_{0}\sD_{X\times \C}$, (2) $V_{\alpha}\cM_f$ is indexed right-continuously and discretely, (3) one has $V_{\alpha}\cM_f \cdot t \subseteq V_{\alpha-1}\cM_f$, with equality for $\alpha<0$, (4) one has $V_{\alpha}\cM_f \cdot \partial_t \subseteq V_{\alpha+1}\cM_f$, (5) the operator $t\d_t-\alpha$ acts nilpotently on $\gr^V_\alpha\cM_f$.
\end{definition}
It is proved by Kashiwara \cite{Kashiwara83} and Malgrange \cite{Malgrange} that the $V$-filtration exists and is unique. Suppose $(\cM,F_{\bullet}\cM,K)$ underlies a $\Q$-mixed Hodge module $M$ where $K\otimes_{\Q}\C=\DR(\cM)$. Set $(\cM_f,F_{\bullet}\cM_f)=(i_f)_{+}(\cM,F_{\bullet}\cM)$. For $-1\leq \alpha<0$, the nearby cycles for the eigenvalue $\lambda=e^{2\pi i \alpha}$ are described by
\begin{equation}\label{eqn: nearby cycle D modules}
\psi_{f,\lambda}M=\left(\gr^V_{\alpha}\cM_f, F_{\bullet-1}\gr^V_{\alpha}\cM_f, \psi_{f,\lambda}K\right), \quad W_{\bullet}\gr^V_{\alpha}\cM_f=W_{\bullet-(n-1)}(N)\gr^V_{\alpha}\cM_f,
\end{equation}
where $n=\dim X$ and $W_{\bullet}(N)$ is the monodromy weight filtration of $N$ and 
\[ F_{\bullet}\gr^V_{\alpha}\cM_f\colonequals \frac{F_{\bullet}\cM_f\cap V_{\alpha}\cM_f}{F_{\bullet}\cM_f\cap V_{<\alpha}\cM_f}.\]
For the unipotent vanishing cycles, this changes to
\begin{equation}\label{eqn: unipotent vanishing cycle D modules}
\phi_{f,1}M=\left(\gr^V_{0}\cM_f, F_{\bullet}\gr^V_{0}\cM_f,\phi_{f,1}K\right), \quad W_{\bullet}\gr^V_{0}\cM_f=W_{\bullet-n}(N)\gr^V_{0}\cM_f.
\end{equation}
These shifts are needed to make the definition work out properly in all cases.

\subsection{Nearby and vanishing cycle for divisors}
\label{sec: vanishing cycle for divisors}
Now let us explain why the nearby and vanishing cycles of a complex Hodge module with respect to a divisor are naturally twisted Hodge modules.  Let $X$ be a complex manifold and let $\cM$ be a regular holonomic right
$\sD_X$-module.  Let $D$ be an effective divisor on $X$, set $L=\cO_X(D)$, and
let $s \in H^0(X,L)$ be a global section such that $\dv(s)=D$. We view $L$
as a complex manifold of dimension $\dim X + 1$, and the section $s$
as a closed embedding $s \colon X \into L$. Let $\cM_L=s_{+}\cM$ be the direct image $\sD$-module on $L$. 

The zero section of $L$ induces a filtration $V_{\bullet}\cM_L$; this is defined
locally as the $V$-filtration in Definition \ref{definition: V-filtration}, but the
resulting filtration on $\cM_L$ is actually globally well-defined \cite[Proposition
1.5]{Budur}. Each $V_{\alpha}\cM_L$ is a sheaf of $V_0 \sD_L$-modules and each
$\gr^V_{\alpha}\cM_L$ is therefore a well-defined $\gr_0^V \sD_L$-module, but unlike
in the local setting, it is not a $\sD_X$-module. We get closer to $\sD$-modules
once we take the associate graded of the weight filtration. Recall from \S \ref{par:tdo}
that the operator $\theta \in F_1\gr^V_0\sD_L$ is globally well-defined and
central, then the properties of $V$-filtration implies that the operator 
\begin{equation}\label{eqn: the global nilponent operator}
	N\colonequals \theta -\alpha
\end{equation} 
acts nilpotently on $\gr^V_{\alpha}\cM_L$. Let $W_{\bullet}(N)\gr^V_{\alpha}\cM$ denote the weight
filtration of this nilpotent operator. Since $\sD_{X, \alpha L} =
\gr_0^V \sD_L / (\theta - \alpha) \gr_0^V \sD_L$, it follows that each subquotient
\[
	\gr_{\ell}^{W(N)} \gr_{\alpha}^V \cM_L
\]
has the structure of an $\alpha L$-twisted right $\sD$-module. 

Now let us suppose that $\cM$ comes with a flat hermitian pairing
\[
	S \colon \cM \otimes_{\C} \overline{\cM} \to \Cur_X
\]
into the sheaf of currents of bidegree $(n,n)$. In that case, we get an
induced flat hermitian pairing on $\gr_{\alpha}^V \cM$, but now valued in the sheaf
of twisted currents.

\begin{lemma} \label{lem:induced-pairing}
	For each $\alpha \in [-1, 0]$, we have an induced flat hermitian pairing
	\[
		S_{\alpha} \colon \gr_{\alpha}^V \cM_L \otimes_{\C} \overline{\gr_{\alpha}^V
		\cM_{L}} \to \Cur_{X, \alpha L}.
	\]
	The operator $\theta \in \End(\gr_{\alpha}^V \cM_{L})$ is self-adjoint with respect to
	$S_{\alpha}$.
\end{lemma}

\begin{proof}
	The point is that the construction of the induced pairing on nearby and vanishing
	cycles in \cite{SSMHMproject} transforms correctly from one local trivialization
	of $L$ to another. We will give the proof for $-1 \leq \alpha < 0$, the case
	$\alpha = 0$ being similar.

	Let $\phi \colon L \vert_U \to U \times \CC$ be a local trivialization of the line
	bundle $L$, and denote by $t$ the coordinate function on $\CC$. The restriction of
	$V_{\alpha}\cM_L$ to the open subset $U \times \CC$ is a module over $V_0 \sD_L$, and the
	restriction of $\gr_{\alpha}^V \cM_{L}$ to the open set $U$ is a right $\sD_U$-module.
	Let us quickly review the construction of the induced pairing $S_{\alpha}$, using
	the Mellin transform. Let $\varphi \in A_c(U)$ be a test function with compact
	support in $U$, and let $\eta \colon \CC \to [0,1]$ be a compactly support smooth
	function that is identically equal to $1$ near the origin. Let $m_1, m_2 \in
	\Gamma(U \times \CC, V_{\alpha} \cM_L)$ The expression
	\[
		\Bigl\langle S(m_1, m_2), \abs{t}^{2s} \eta(t) \varphi \Bigr\rangle
	\]
	is a holomorphic function of $s \in \CC$ as long as $\Re s \gg 0$, and extends to
	a meromorphic function on all of $\CC$ (using the properties of the
	$V$-filtration). One can show that the residue at $s = \alpha$ depends
	continuously on the test function $\varphi$, and that the formula
	\[
		\Bigl\langle S_{\alpha}([m_1], [m_2]), \varphi \Bigr\rangle
		= \Res_{s = \alpha} \Bigl\langle S(m_1, m_2), \abs{t}^{2s} \eta(t) \varphi \Bigr\rangle
	\]
	defines a flat hermitian pairing on the right $\sD_U$-module $\gr_{\alpha}^V \cM_{L}
	\vert_U$. Moreover, the endomorphism $\theta = t \partial_t$ is self-adjoint with
	respect to this pairing.

	To prove the lemma, it is enough to check that these pairings transform correctly
	from one local trivialization to another. Let us denote by $C_{(U, \phi)} =
	S_{\alpha}([m_1], [m_2])$ the current constructed above. Let $(U, \phi')$ be a second
	trivialization such that
	\[
		(\phi' \circ \phi^{-1})(x,t) = \bigl( x, g(x) t \bigr) \quad \text{and $g \in
		\Gamma(U, \sO_X^{\times})$.}
	\]
	If we let $t'$ be the resulting holomorphic coordinate on $\CC$, we have $t' = g
	t$. The same formula as above then defines a second current
	\[
		\bigl\langle C_{(U, \phi')}, \varphi \bigr\rangle
		= \Res_{s = \alpha} \Bigl\langle S(m_1, m_2), \abs{t'}^{2s} \eta(t') \varphi \Bigr\rangle
	\]
	Since $g$ is holomorphic, the right-hand side evaluates to
	\begin{align*}
		\Res_{s = \alpha} \Bigl\langle S(m_1, m_2), \abs{g}^{2s} \abs{t}^{2s} \eta(g t)
			\varphi \Bigr\rangle &= \Res_{s = \alpha} \Bigl\langle
		S(m_1, m_2), \abs{g}^{2 \alpha} \abs{t}^{2s} \eta(g t) \varphi \Bigr\rangle \\
		&= \bigl\langle C_{(U, \phi)}, \abs{g}^{2 \alpha} \varphi \bigr\rangle,
	\end{align*}
	and so we arrive at the identity $C_{(U, \phi')} = C_{(U, \phi)} \cdot \abs{g}^{2
	\alpha}$. This is enough to conclude that the currents on each local
	trivialization glue together into an $\alpha L$-twisted current on $X$.
\end{proof}

As long as $\alpha \in \RR$, the nilpotent endomorphism $N = \theta - \alpha$ is
self-adjoint with respect to the pairing $S_{\alpha}$, and so the pairing descends to
the graded quotients of the weight filtration. This gives us flat sesquilinear
pairings
\[
	S_{\alpha} \colon \gr_{\ell}^{W(N)} \gr_{\alpha}^V \cM_L \otimes_{\C}
	\overline{\gr_{-\ell}^{W(N)} \gr_{\alpha}^V \cM_L} \to \Cur_{X, \alpha L}.
\]
We need one extra piece of data to describe the induced polarization. The direct sum
\[
	\bigoplus_{\ell \in \ZZ} \gr_{\ell}^{W(N)} \gr_{\alpha}^V \cM_L
\]
carries a representation of the Lie algebra $\sltwo$. If we denote the three
generators by 
\[
	\Hsl = \begin{pmatrix} 1 & 0 \\ 0 & -1 \end{pmatrix}, \quad
	\Xsl = \begin{pmatrix} 0 & 1 \\ 0 & 0 \end{pmatrix}, \quad
	\Ysl = \begin{pmatrix} 0 & 0 \\ 1 & 0 \end{pmatrix},
\]
then $\Hsl$ acts as multiplication by $\ell$ on $\gr_{\ell}^{W(N)}$, and $\Ysl$ acts
as $N = \theta - \alpha$. The Weil element
\[
	\wsl = \begin{pmatrix} 0 & 1 \\ -1 & 0 \end{pmatrix} \in \operatorname{SL}_2(\CC)
\]
has the property that $\wsl \Hsl \wsl^{-1} = -\Hsl$, $\wsl \Xsl \wsl^{-1} = -\Ysl$,
and $\wsl \Ysl \wsl^{-1} = -\Xsl$, and therefore determines an isomorphism between
$\gr_{-\ell}^{W(N)}$ and $\gr_{\ell}^{W(N)}$. It follows that
\[
	S_{\alpha} \circ (\id \otimes \wsl) \colon
	\gr_{\ell}^{W(N)} \gr_{\alpha}^V \cM_L \otimes_{\C}
	\overline{\gr_{\ell}^{W(N)} \gr_{\alpha}^V \cM_L} \to \Cur_{X, \alpha L}
\]
is a flat hermitian pairing on the indicated $\alpha L$-twisted $\sD$-module.
We can use this to show that the nearby and vanishing cycles of a complex Hodge
module are naturally twisted Hodge modules.

\begin{prop}\label{prop: vanishing cycles of Hodge modules}
	Let $(\cM, F_{\bullet} \cM, S)$ be a  polarized complex Hodge module on $X$. For any
	real number $\alpha \in [-1,0)$ and any integer $\ell \in \ZZ$, the object
	\[
		\left( \gr_{\ell}^{W(N)} \gr_{\alpha}^V \cM_L, F_{\bullet-1} \gr_{\ell}^{W(N)}
		\gr_{\alpha}^V \cM_L, S_{\alpha} \circ (\id \otimes \wsl) \right)
	\]
	is a polarized $\alpha L$-twisted Hodge module on $X$. The object
	\[
		\left( \gr_{\ell}^{W(N)} \gr_0^V \cM_L, F_{\bullet} \gr_{\ell}^{W(N)}
		\gr_0^V \cM_L, S_0 \circ (\id \otimes \wsl) \right)
	\]
	is a polarized complex Hodge module (without any twisting).
\end{prop}

\begin{proof}
	The claim is that both objects are polarized complex Hodge modules in any local
	trivialization of $L$. This follows from the definition of polarized complex Hodge
	modules, because in any local trivialization, the two objects are exactly the
	nearby cycles and unipotent vanishing cycles of $(\cM, F_{\bullet} \cM, S)$. 
\end{proof}

\subsection{Vanishing theorems for twisted Hodge modules}\label{sec: vanishing theorems for twisted Hodge modules}
We prove the Kodaira-Saito vanishing theorem for twisted Hodge modules. Let $L$ be a line bundle on $X$, $\alpha \in \Q$ and let $(\cM,F_{\bullet}\cM)$ be a filtered $\alpha L$-twisted $\sD$-module with a good filtration (see \S \ref{sec: good filtrations}).
\begin{definition}\label{definition: non-characteristic for twisted D modules}
We say that $f$ is \emph{non-characteristic} for the filtered twisted $\sD$-module $(\cM,F_{\bullet}\cM)$ if for every open subset $U\subseteq X$ trivializing $L$, the induced morphism $f^{-1}(U)\to U$ is non-characteristic for the filtered $\sD$-module $(\cM,F_{\bullet}\cM)|_U$ in the sense of \cite[\S 3.5.1]{Saito88}.
\end{definition}

\begin{lemma}\label{lemma: noncharacteristic pullback for complex Hodge modules}
    Let $M$ be a $\alpha L$-twisted polarized Hodge module on $X$ with strict support $Z$. If $f:Y\to X$ is non-characteristic for the underlying twisted $\sD$-module $(\cM,F_{\bullet}\cM)$. Then we have isomorphisms
    \[ \omega_{Y/X}\otimes f^{\ast}\cM\xrightarrow{\sim} \cM_Y, \quad \omega_{Y/X}\otimes f^{\ast}F_{\bullet}\cM\xrightarrow{\sim} F_{\bullet}\cM_Y,\]
    where $(\cM_Y,F_{\bullet}\cM_Y)$ underlies an $\alpha L_Y$-twisted polarized Hodge module with strict support $f^{-1}(Z)$ and $L_Y=f^{\ast}L$.
\end{lemma}

\begin{proof}
    The proof proceeds in the same way as in  \cite[Theorem 9.3]{Schnell16}. \end{proof}
 
\begin{proof}[Proof of Theorem \ref{thm: vanishing theorem for twisted Hodge module}]
Write $\alpha=p/m$ with $\mathrm{gcd}(p,m)=1$. By a result of Bloch and Gieseker (see \cite{BG71} or \cite[Proposition 2.67]{KM98}), there is 
\begin{equation}\label{eqn: the condition of Bloch-Gieseker covering}\tag{$\ast$}
    \textrm{a finite flat morphism $f:Y\to X$, with $Y$ smooth projective,  $f^{\ast}L=mL_Y$}.
\end{equation} 
Following the proof \cite[Theorem 4.1.10]{Laz04I}, it is not hard to see that one can choose $f$ to be non-characteristic for $(\cM,F_{\bullet}\cM)$.  With such a choice of $f$, by Lemma \ref{lemma: noncharacteristic pullback for complex Hodge modules} and $\alpha f^{\ast}L=pL_Y$, one knows that the pullback 
\[(\cM_Y,F_{\bullet}\cM_Y)\colonequals \omega_{Y/X}\otimes f^{\ast}(\cM,F_{\bullet}\cM)\]
underlies a $pL_Y$-twisted polarized Hodge module with strict support $f^{-1}(Z)$  and 
\[ \gr^F_{k}\DR(\cM_Y)\otimes f^{\ast}\cO_Z(B)=\omega_{Y/X}\otimes f^{\ast}(\gr^F_{k}\DR(\cM)\otimes \cO_Z(B)).\]
Since pulling back by finite map induces injective map on cohomology groups  \cite[Lemma 4.1.14]{Laz04I}, it suffices to prove the vanishing of 
\[ \mathbb{H}^i\left(Y,\gr^F_{k}\DR(\cM_Y)\otimes f^{\ast}\cO_Z(B)\right)=\mathbb{H}^i\left(Y,\gr^F_{k}\DR(\cM_Y\otimes L_Y^{-p}))\otimes f^{\ast}\cO_Z(B)\otimes L_Y^{p}\right).\]
Since $p$ is an integer, Lemma \ref{lemma: canonical isomorphism of grV0DL and L tensor grV0DL} implies that $(\cM_Y,F_{\bullet}\cM_Y)\otimes L_Y^{-p}$ underlies a polarized complex Hodge module with strict support $f^{-1}(Z)$. Note that the line bundle 
\[f^{\ast}\cO_Z(B)\otimes L_Y^{p}|_{f^{-1}(Z)}=f^{\ast}\cO_Z(B+\alpha D|_Z)\]
is ample by assumption and finiteness of $f$, therefore the desired vanishing follows from Kodaira-Saito's vanishing theorem for complex Hodge modules \cite[Theorem 16.3.10]{SSMHMproject}.

If $\Omega^1_X$ is trivial, one can argue as in the proof of \cite[Lemma 2.5]{Schnellgenericvanishing}: since the Hodge filtration on $\cM$ is bounded from below, one has $\gr^F_k\cM=0$ for $k\ll 0$. Inductively for each $k$, there is a distinguished triangle
\[ E_k \to \gr^F_k\DR(\cM) \to \gr^F_{k}\cM \to E_k[1],\]
where $E_k\in D^b_{\textrm{coh}}(\cO_Z)$ is an object satisfying $\mathbb{H}^i(Z,E_k\otimes \cO_Z(B))=0$ for $i>0$.
\end{proof}

\subsection{Arakelov inequalities for twisted Hodge modules}
We need the following results for the proof of Theorem \ref{thm: partial solution of Casalaina Martin conjecture}.  Let $C$ be a smooth projective curve of genus $g(C)$.
\begin{lemma}\label{lemma: degree bound of lowest Hodge module}
Let $\cM$ be a right $\sD$-module underlying a polarized complex Hodge module with strict support $C$. Let $p$ be the smallest integer such that $F_p\cM\neq 0$. Then 
\[\deg F_p\cM \geq 2g(C)-2.\]
\end{lemma}

\begin{proof}
Let $\cV$ be the polarized complex variation of Hodge structures on some open subset $C_0\subseteq C$, associated to $\cM$. Then $F_p\cM$ is locally free and $F_p\cM=\omega_C\otimes \cV'$, where $\cV'$ is the extension of the smallest piece $F^{-p}\cV$ of the Hodge filtration on $\cV$ from $C_0$ to $C$ in Saito's theory. By \cite{Peters00}, one has $\deg \cV'\geq 0$. As $\mathrm{rank}(\cV')\geq 1$, we have
\[\deg F_p\cM=\mathrm{rank}(\cV')\cdot(2g(C)-2)+\deg \cV'\geq 2g(C)-2.\qedhere\]
\end{proof}

\begin{corollary}\label{cor: degree bound of lowest twisted Hodge module}
Let $D$ be an effective divisor on $C$ and set $L=\cO_C(D)$. Let $\cM$ be an $\alpha L$-twisted $\sD$-module underlying a twisted Hodge module with strict support $C$. Let $p$ be the smallest integer such that $F_p\cM\neq 0$, then \[ \deg F_p\cM \geq 2g(C)-2+\alpha \cdot \deg D.\]
\end{corollary}

\begin{proof}
If $\alpha\in \Z$, then  $\cM\otimes L^{-\alpha}$ underlies a polarized complex Hodge module and thus Lemma \ref{lemma: degree bound of lowest Hodge module} gives 
\[ \deg F_p\cM-\alpha \cdot \deg D=\deg F_p(\cM\otimes L^{-\alpha})\geq 2g(C)-2,\]
which is what we want. If $\alpha\in \Q$, write $\alpha=q/m$ with $\mathrm{gcd}(q,m)=1$. As in the proof of Theorem \ref{thm: vanishing theorem for twisted Hodge module}, there is a finite flat morphism $f:C'\to C$, with $C'$ smooth projective,  $f^{\ast}L=mL'$, where $L'$ is a line bundle on $C'$. Moreover, $f$ is non-characteristic with respect to the underlying twisted $\sD$-module of $M$ so that 
\[ (\cM',F_{\bullet}\cM')\colonequals \omega_{C'/C}\otimes f^{\ast}(\cM,F_p\cM)\]
underlies an $\alpha f^{\ast}L=qL'$-twisted Hodge module on $C'$. Therefore Case I shows that
\[ \deg \omega_{C'/C}+\deg f^{\ast}F_{p}\cM=\deg F_p\cM'\geq 2g(C')-2+q\cdot \deg L'.\]
Since $\deg \omega_{C'/C}=\deg \omega_{C'}-\deg f\cdot \deg \omega_C$ and $C',C$ are both smooth projective (so that $\deg \omega_C=2g(C)-2$, same for $C'$), the above inequality gives
\[ \deg F_{p}\cM\geq 2g(C)-2+\frac{q \deg L'}{\deg f}=2g(C)-2+\alpha \cdot \deg L.\]
The last equality follows from $\deg L'=\frac{1}{m}\deg f^{\ast}L=\frac{\deg f}{m}\deg L$. Therefore we obtain the desired result.
\end{proof}

\section{Higher multiplier ideals: definition and basic properties}\label{sec: higher multiplier ideals: definition and first properties}

\begin{setup}\label{setup for total graph embedding}
Let $X$ be a complex manifold of dimension $n$. Let $D$ be an effective divisor on $X$ and $L=\cO_X(D)$ be the associated holomorphic line bundle and also denote by $L$ the total space of the line bundle. Let $s\in H^0(X,L)$ be a section with $\mathrm{div}(s)=D$, which is viewed as a closed embedding $s: X\to L$. Consider the  direct image Hodge module $M\colonequals s_{\ast}\mathbb{Q}_X^H[n]$, with the underlying filtered right $\sD$-module $(\cM,F_{\bullet}\cM)=s_{+}(\omega_X,F_{\bullet}\omega_X)$.
\end{setup}

\subsection{Definitions and basic properties}
Let $V_{\bullet}\cM$ be the $V$-filtration relative to the zero section of $L$. It is direct to check that \begin{equation}\label{eqn: associated graded of grF}
\gr^F_{-n+k}\cM \cong s_{\ast}(\omega_X\otimes L^k),  \quad \textrm{if $k\geq 0$},
\end{equation} 
and is $0$ otherwise. This leads to an inclusion of coherent $\cO_X$-modules:
\[\gr^F_{-n+k}V_{\alpha}\cM\colonequals \frac{F_{-n+k}\cM\cap V_{\alpha}\cM}{F_{-n+k-1}\cM\cap V_{\alpha}\cM}\hookrightarrow \gr^F_{-n+k}\cM=s_{\ast}(\omega_X\otimes L^k).\] 
\begin{definition}\label{definition: higher multiplier ideals}
For any $k\in \N$ and $\alpha\in \Q$, the \emph{higher multiplier ideal} $\cI_{k,\alpha}(D)$ is defined to be the unique coherent ideal sheaf on $X$ such that
\begin{equation}\label{eqn: definition of higher multiplier ideals}
\gr^F_{-n+k}V_{\alpha}\cM \cong s_{\ast}(\omega_X\otimes L^k \otimes \cI_{k,\alpha}(D)).
\end{equation}
Similarly, we define $\cI_{k,<\alpha}(D)$ using $\gr^F_{-n+k}V_{<\alpha}\cM$. 
\end{definition}
\begin{remark}
If $k=0$, \cite[Theorem 0.1]{BS05} gives 
\begin{equation}\label{eqn: I0 gives the multiplier ideals}
     \cI_{0,\alpha}(D)= \cJ(X,(-\alpha-\epsilon)D), \quad \textrm{for some $0<\epsilon \ll 1$},
\end{equation}
the right hand side is the usual multiplier ideal. Equivalently, $\cI_{0,<\alpha}(D)=\cJ(X,-\alpha D)$.\end{remark}

\begin{definition}\label{definition: jumping part of higher multiplier ideal}
Define $\cG_{k,\alpha}(D)\colonequals \cI_{k,\alpha}(D)/\cI_{k,<\alpha}(D)$. We call $\alpha\in \mathbb{Q}$ a \emph{jumping number} if $\cG_{k,\alpha}(D)\neq 0$. There is an induced isomorphism:
\begin{equation}\label{eqn: Gkalpha as grFgrV}
    \gr^F_{-n+k}\gr^V_{\alpha}\cM \cong \omega_X\otimes L^k\otimes \cG_{k,\alpha}(D).
\end{equation}
\end{definition}

\begin{definition}\label{definition: weight filtration on graded pieces}
For $\alpha\in [-1,0]$, we define the \emph{weight filtration} $W_{\bullet}\cG_{k,\alpha}(D)$ using \eqref{eqn: Gkalpha as grFgrV}:
\begin{equation}\label{eqn: Wgkalpha as grFWgrV}
    W_{\ell}\gr^F_{-n+k}\gr^V_{\alpha}\cM \cong \omega_X\otimes L^k\otimes W_{\ell}\cG_{k,\alpha}(D), \quad \ell\in \Z.
\end{equation}
The graded piece is denoted by $\gr^W_{\ell}\cG_{k,\alpha}(D)$
with
\begin{equation}\label{eqn:gr^Wgkalpha as grFgrWgrV}
    \gr^F_{-n+k}\gr^{W}_{\ell}\gr^V_{\alpha}\cM \cong \omega_X\otimes L^k\otimes \gr^W_{\ell}\cG_{k,\alpha}(D).
\end{equation}
\end{definition}

\begin{definition}\label{definition: weight filtration on Ikalpha}
For $\alpha\in [-1,0]$, the weight filtration $W_{\bullet}\cI_{k,\alpha}(D)$ is defined as the preimage of $W_{\bullet}\cG_{k,\alpha}(D)$ under the quotient map $\cI_{k,\alpha}(D)\to \cG_{k,\alpha}(D)$, which satisfies\begin{equation}\label{eqn: ses of weighted higher multiplier ideals}
0\to \cI_{k,<\alpha}(D) \to W_{\ell}\cI_{k,\alpha}(D) \to W_{\ell}\cG_{k,\alpha}(D)\to 0, \quad \ell\in \Z.
\end{equation}
\end{definition}

\begin{prop}\label{prop: first properties of higher multiplier ideals for divisors}
For $k\in \N$, the following properties hold.
\begin{enumerate}[label=(\Roman*)]
    \item If $\alpha\leq \beta$, then $\cI_{k,\alpha}(D)\subseteq \cI_{k,\beta}(D)$. The sequence of ideal sheaves $\{\cI_{k,\alpha}(D) \}_{\alpha\in \Q}$ is discrete and right continuous, the set of jumping numbers is discrete.
    \item One has $\cI_{k,<k}(D)=\cO_X$.
    \item For any $\alpha$, there exist morphisms
    \begin{equation}\label{eqn: periodicity morphism}
    \cI_{k,\alpha}(D) \to\cI_{k,\alpha-1}(D)\otimes \cO_X(D),
    \end{equation}
    \begin{equation}\label{eqn: griffiths transversality morphism}
     \cI_{k,\alpha}(D) \to \cI_{k+1,\alpha+1}(D),   
    \end{equation}
    \begin{equation}\label{eqn: gkalpha to gk+1alpha+1}
     \cG_{k,\alpha}(D) \to \cG_{k+1,\alpha+1}(D),   
    \end{equation}
    which are isomorphisms for $\alpha<0$ (respectively $\alpha\geq -1$ and $\alpha >-1$). The morphism \eqref{eqn: gkalpha to gk+1alpha+1} is surjective for $\alpha=-1$. 
    \item For $k\geq 1$, there are two short exact sequences:    \begin{align}\label{eqn: ses for Ik0 Ik-1}
        0 \to L^{k}\otimes \cI_{k,0}(D)\otimes \omega_X\to L^{k+1}\otimes \cI_{k,-1}(D)\otimes \omega_X\to \gr^F_{-n+k}\omega_X(\ast D) \to 0,
    \end{align}
    \begin{align}\label{eqn: ses from Gk-1-1 to Gk0}
        0 \to \gr^F_{-n+k}\omega_X(!D)\otimes L \to L^{k}\otimes \cG_{k-1,-1}(D)\otimes \omega_X \to L^{k}\otimes \cG_{k,0}(D)\otimes \omega_X \to 0.
    \end{align}
 Here $(\omega_X(\ast D),F)$ and $(\omega_X(!D),F)$ are filtered $\sD$-modules underlying the mixed Hodge modules $j_{\ast}\Q^H_{X\setminus D}[n]$ and $j_{!}\Q^{H}_{X\setminus D}[n]$ with $j:X\setminus D\hookrightarrow X$.
    
  \end{enumerate}    \end{prop}

\begin{proof}

The statement (I) follows from properties of $V$-filtration in Definition \ref{definition: V-filtration}.

For (II), if $\alpha\geq k$, the surjectivity of $\d_t$ for $\alpha\geq -1$ induces a surjection
\[ F_{-n-1}\gr^V_{\alpha-k-1}\cM \twoheadrightarrow F_{-n+k}\gr^V_{\alpha}\cM \otimes L^{-k-1}.\]
Since $F_{-n-1}\cM=0$ by \eqref{eqn: associated graded of grF}, we have $F_{-n+k}\gr^V_{\alpha}\cM=0$ for $\alpha \geq k$. This means that $F_{-n+k}\cM \subseteq V_{<k}\cM$ and therefore  $\gr^F_{-n+k}V_{<k}\cM=\gr^F_{-n+k}\cM$.

For (III), the existence and properties of \eqref{eqn: periodicity morphism}, \eqref{eqn: griffiths transversality morphism} and \eqref{eqn: gkalpha to gk+1alpha+1} follow from \cite[(3.2.1)]{Saito88} and $\Q^H_X[n]$ has strict support $X$, except for the property of \eqref{eqn: griffiths transversality morphism}. To prove this, let us consider the following commutative diagram
\[ \begin{tikzcd}
0 \arrow[r] &\cI_{k,\alpha}(D) \arrow[r]\arrow[d] &\cI_{k,k}(D) \arrow[r]\arrow[d] &\cI_{k,k}(D)/\cI_{k,\alpha}(D) \arrow[r]\arrow[d] &0\\
0 \arrow[r] &\cI_{k+1,\alpha+1}(D) \arrow[r] &\cI_{k+1,k+1}(D) \arrow[r] &\cI_{k+1,k+1}(D)/\cI_{k+1,\alpha+1}(D) \arrow[r]&0,
\end{tikzcd}\]
where vertical arrows come from \eqref{eqn: griffiths transversality morphism}. It is clear by the proof of (II) that the second vertical map is an isomorphism. The third vertical map is also an isomorphism by \eqref{eqn: gkalpha to gk+1alpha+1}, since $\cI_{k,k}(D)/\cI_{k,\alpha}(D)$ is a finite extension of $\cG_{k,\beta}(D)$ for $\beta\in (\alpha,k]$ and $\beta>\alpha\geq -1$. We conclude by snake lemma that the first vertical map is an isomorphism as well.

To prove (IV), let $j:X\setminus D \hookrightarrow X$ and $i: D\hookrightarrow X$ be the open and closed embeddings. We use the functorial triangles from \cite[(4.4.1)]{Saito90}. Since $i_{\ast}(H^0(i^{!}\Q^H_X[n]))=0$, we have a short exact sequence of mixed Hodge modules
\begin{equation}\label{eqn: Hodge module ses for i!}
0 \to \Q^H_X[n] \to j_{\ast}\Q^H_{X\setminus D}[n] \to i_{\ast}(H^1i^{!}\Q^H_X[n])\to 0,
\end{equation}
with the underlying filtered $\sD$-modules 
\begin{equation}\label{eqn: filtere D module ses for i!}
0\to (\omega_X,F) \to (\omega_X(\ast D), F) \to i_{\ast}(H^1i^{!}(\omega_X,F)) \to 0.
\end{equation}
Since $\Q^H_X[n]$ has strict support $X$, the underlying filtered $\sD$-module of $i_{\ast}(H^1i^{!}\Q^H_X[n])$ can be computed as the cokernel of the injective morphism
\[ \mathrm{var}: \left(\gr^V_0\cM,F_{\bullet}\gr^V_0\cM\right) \to \left(\gr^V_{-1}\cM\otimes L, F_{\bullet}\gr^V_{-1}\cM\otimes L\right), \]
where locally $\mathrm{var}=t$. Since $\gr^F_{-n}\omega_X=\omega_X$ and $\gr^F_{-n+k}\omega_X=0$ for $k\geq 1$, combined with \eqref{eqn: Gkalpha as grFgrV} and \eqref{eqn: filtere D module ses for i!}, one has a short exact sequence for $k\geq 1$:
\begin{equation}\label{eqn: ses from Gk0 to Gk-1}
0 \to \omega_X\otimes L^{k}\otimes \cG_{k,0}(D) \to \omega_X \otimes L^{k+1}\otimes \cG_{k,-1}(D) \to \gr^F_{-n+k}\omega_X(\ast D) \to 0.
\end{equation}
It follows from (III) that the morphism induced by \eqref{eqn: periodicity morphism}
\[ \omega_X\otimes L^k\otimes \cI_{k,<0}(D)\xrightarrow{\sim}\omega_X\otimes L^{k+1}\otimes \cI_{k,<-1}(D)\]
is an isomorphism. Therefore \eqref{eqn: ses for Ik0 Ik-1} holds. 

Now, let us consider the exact sequence dual to \eqref{eqn: Hodge module ses for i!}:
\[ 0\to i_{\ast}(H^{-1}i^{\ast}\Q^H_X[n]) \to j_{!}\Q^H_{X\setminus D}[n] \to \Q^H_X[n] \to 0.\]
Dually, because $\Q^H_X[n]$ has strict support $X$, the underlying filtered $\sD$-module of $i_{\ast}(H^{-1}i^{\ast}\Q^H_X[g])$ is computed by the kernel of the surjective morphism
\[ \mathrm{can}: \left( \gr^V_{-1}\cM,F_{\bullet-1}\gr^V_{-1}\cM\right) \to \left(\gr^V_{0}\cM\otimes L^{-1}, F_{\bullet}\gr^V_{0}\cM\otimes L^{-1}\right),\]
where locally $\mathrm{can}=\d_t$. This leads to \eqref{eqn: ses from Gk-1-1 to Gk0}.

\end{proof}

\subsection{Higher multiplier ideals of $\Q$-divisors}\label{sec: Q divisors}

The definition of higher multiplier ideals can be extended to $\Q$-divisors, although
the twisting causes some complications. Suppose that $D$ is an effective divisor on
$X$, defined by a global section $s \in H^0(X, L)$. For any integer $m \geq 1$, we
denote by $M_m$ the Hodge module on the total space of the line bundle $L^m$,
obtained by the graph embedding along the section $s^m$ defining the divisor $mD$.
Let $(\Mmod_m, F_{\bullet} \Mmod_m)$ be the underlying filtered $\sD$-module.
It follows from Proposition \ref{prop: V filtration for powers of
functions} (applied locally) that we have an isomorphism of $\sO_X$-modules
\[
	F_{-n+k} V_{m \alpha} \Mmod_1 \cong F_{-n+k} V_{\alpha} \Mmod_m,
\]
for $\alpha \leq 0$ and $k \in \Z$. In light of Definition \ref{definition: higher
multiplier ideals}, this is saying that 
\begin{equation}\label{eqn: global version of compatibility of powers of functions}
	 \cI_{k,m\alpha}(D)\otimes \sO_{X}(kD) \cong
	 \cI_{k,\alpha}(mD) \otimes \sO_X(kmD).
\end{equation}
Both sides are torsion-free coherent $\sO_X$-modules of rank $1$. We can use this
formula in order to extend the definition of higher multiplier ideals to
effective $\Q$-divisors.

Let $E$ be an effective $\Q$-divisor on a complex manifold $X$. Let $m \geq 1$ be a
positive integer with the property that $mE$ has integer coefficients. For $\alpha
\leq 0$ and $k \in \N$, we then define the torsion-free coherent $\sO_X$-module
\[
	\cS_{k, \alpha}(E) \defeq
	 \cI_{k,\alpha/m}(mE) \otimes \sO_X(kmE).
\]
%Of course, the notation on the left-hand side is purely symbolic, since $\sO_X(kE)$
%does not make sense as a line bundle.
As a consequence of \eqref{eqn: global version of compatibility of powers of functions}, the resulting $\sO_X$-module is (up to
isomorphism) independent of the choice of $m$. In order to get
a sheaf of ideals, we observe that the reflexive hull of $\cS_{k,
\alpha}(E)$ is a line bundle. Consequently, we have
\[
	\cS_{k, \alpha}(E) = \cI_{k,\alpha}'(E) \otimes \cS_{k, \alpha}(E)^{\ast\ast}
\]
for a unique coherent sheaf of ideals in $\sO_X$ and we define $\cI_{k,\alpha}'(E)$ to be the \emph{higher multiplier ideal} of $E$. By construction, the cosupport of
this ideal sheaf has codimension $\geq 2$ in $X$. In the case of a $\Z$-divisor $D$,
this ideal $\cI_{k, \alpha}'(D)$ is the result of removing from $\cI_{k, \alpha}(D)$
its divisorial part. All the local properties of higher multiplier ideals therefore
carry over to the setting of $\Q$-divisors.

\subsection{Comparison with the microlocal $V$-filtration}
We relate the higher multiplier ideals with Saito's microlocal $V$-filtration \cite{Saito94,Saito16}. Denote by $\tilde{V}^{\bullet}\cO_X$ the microlocal $V$-filtration. Using \cite[Lemma 2.13]{MY23}, it is not hard to show the following.

\begin{corollary}\label{corollary: higher multiplier ideals and microlocal ideals}
For any $k\in \N$, one has
 \[\cI_{k,\alpha}(D) =\begin{cases}\tilde{V}^{k-\alpha}\cO_X, &\textrm{if $\alpha\geq -1$},\\
 \tilde{V}^{k-(\alpha+t)}\cO_X\otimes \cO_X(-tD), &\textrm{if $\alpha< -1$ and $t\in \N$ so that $-1\leq \alpha+t<0$}.
 \end{cases}\]
\end{corollary}

\begin{corollary}\label{corollary: Ik+1 contained in Ik}
One has 
\[ \cI_{k+1,\alpha}(D)\subseteq \cI_{k,\alpha}(D), \quad \textrm{for all $\alpha\in \Q,k\in \N$}.\]
\end{corollary}

\begin{proof}
    It follows from Corollary \ref{corollary: higher multiplier ideals and microlocal ideals} and that $\tilde{V}^{\bullet}\cO_X$ is an decreasing filtration \cite{Saito16}.
\end{proof}

\subsection{Minimal exponents as  jumping numbers}\label{sec: minimal exponent}
Let $f$ be a holomorphic function on $X$. The \emph{minimal exponent} of $f$, denoted by $\tilde{\alpha}_f$ in \cite{Saito94}, is defined to be the negative of the largest root of $b_f(s)/(s+1)$, where $b_f(s)$ is the Bernstein-Sato polynomial of $f$. If $b_f(s)=s+1$, we set $\tilde{\alpha}_f=+\infty$. Since the set $\{ \textrm{exp}(2\pi is_0) \mid b_f(s_0)=0\}$ is exactly the set of monodromy eigenvalues of the nearby cycle $\psi_{f}\Q_X[\dim X]$ by Malgrange \cite{Malgrange}, it follows that $\{f=0\}$ is smooth if and only if $\tilde{\alpha}_f=+\infty$. If $D$ is an effective divisor, the minimal exponent can be defined as $\tilde{\alpha}_D=\min_{x\in D} \tilde{\alpha}_{f_x}$, where $f_x$ is the local function of $D$ such that $f_x(x)=0$. \begin{lemma}\label{lemma: minimal exponent via Gkalpha}
Let $D$ be an effective divisor on  $X$, then 
    \begin{align*}
    \tilde{\alpha}_{D}&= \min\{ k-\alpha, k\in \N, \alpha\in (-1,0] \mid \cG_{k,\alpha}(D)\neq 0 \}\\
    &=\min  \{k-\alpha, k\in \N, \alpha\in (-1,0] \mid \cI_{k,<\alpha}(D)\subsetneq \cO_X\}.
    \end{align*}
\end{lemma}

\begin{proof}
It follows from \cite[Proposition 2.14]{MY23} and \cite[(1.3.8)]{Saito16}.
\end{proof}

\section{Local and global properties}
We give a more detailed study of higher multiplier ideals and use Set-up \ref{setup for total graph embedding} throughout. 
\subsection{Local properties}

\begin{proof}[Proof of Theorem \ref{thm: restriction theorem}]
The inclusion $H\hookrightarrow X$ induces a commutative diagram
\[ \begin{tikzcd}
L_H \arrow[d,"p"]\arrow[r,"i"] & L \arrow[d,"p"]\\
H \arrow[r,"i"] & X
\end{tikzcd}\]
where $L_H$ is the total space of the pullback line bundle $i^{\ast}L$ and $p$ are projection maps. In Set-up \ref{setup for total graph embedding} the effective divisor $D$ induces a closed embedding $s:X\to L$ and a mixed Hodge module $M=s_{\ast}\Q^H_X[n]\in \MHM(L)$; analogously, we define $M_H\in \MHM(L_H)$ using the divisor $D_H$.  Let $j:L\setminus L_H\hookrightarrow L$ be the open embedding.

Consider the distinguished triangle from \cite[(4.4.1)]{Saito90}:
\[ i_{\ast}i^{!}M=i_{!}i^{!}M \to M \to j_{\ast}j^{\ast}M\to i_{!}i^{!}M[1].\]
Since $L_H$ is a smooth hypersurface in $L$, one has $i^{!}M[1]=M_H(-1)$ with $F_{\bullet}\cM_H(-1)=F_{\bullet+1}\cM_H$, where the Tate twist comes from $M_H$ being the pushforward of the constant Hodge module on $H$. This gives a short exact sequence of mixed Hodge modules 
\begin{equation}\label{eqn: ss for open embedding}
     0 \to M \to j_{\ast}j^{\ast}M \to i_{\ast}M_H(-1) \to 0.
\end{equation}
Fix $k\in \N$ and $\alpha\in \Q$. Since both Hodge and $V$-filtration preserve exactness and commute with the functor $i_{\ast}$, \eqref{eqn: ss for open embedding} induces a short exact sequence of coherent $\cO_X$-modules
\begin{equation}\label{eqn: ss of grF and V}
0 \to \gr^F_{-n+k}V_{\alpha}\cM \to \gr^F_{-n+k}V_{\alpha}(j_{\ast}j^{\ast}\cM) \to i_{\ast}\left(\gr^F_{-n+k+1}V_{\alpha}\cM_H\right) \to 0.
\end{equation}
This defines an element of
\[ \mathrm{Ext}^1_{X}\left(i_{\ast}(\gr^F_{-n+k+1}V_{\alpha}\cM_H),\gr^F_{-n+k}V_{\alpha}\cM\right).\]
The right adjoint of the functor $i_{\ast}$ for coherent sheaves is the functor $i^{!}=\omega_{H/X}\otimes \mathbf{L}i^{\ast}[-1]$, hence \eqref{eqn: ss of grF and V} also determines in a natural way a morphism
\begin{equation}\label{eqn: restriction morphism}
    \gr^F_{-(n-1)+k}V_{\alpha}\cM_H \to i^{!}\gr^F_{-n+k}V_{\alpha}\cM[1]\cong\omega_{H/X}\otimes \mathbf{L}i^{\ast}\gr^F_{-n+k}V_{\alpha}\cM.
\end{equation}
Substituting in the definition of higher multiplier ideals \eqref{eqn: definition of higher multiplier ideals}, we arrive at
\[ \omega_H\otimes L^k_H \otimes \cI_{k,\alpha}(D_H) \to \omega_{H/X} \otimes \mathbf{L} i^{\ast}(\omega_X\otimes L^k \otimes \cI_{k,\alpha}(D)),\]
which easily gives the desired morphism
\[   \cI_{k,\alpha}(D_H) \to \mathbf{L} i^{\ast} \cI_{k,\alpha}(D).\]
Since $\gr^F_{-n+k}\cM \cong s_{\ast}(\omega_X\otimes L^k)$ and $\gr^F_{-n+k+1}\cM_H\cong (s\circ i)_{\ast}(\omega_H\otimes L^k_H)$ by \eqref{eqn: associated graded of grF}, it is easy to see that there is a commutative diagram
\[ \begin{tikzcd}
\cI_{k,\alpha}(D_H) \arrow[r] \arrow[d] &i^{\ast}\cI_{k,\alpha}(D)\arrow[d]\\
\cO_H \arrow[r,equal] & i^{\ast}\cO_X 
\end{tikzcd}
\]
Since the image of $i^{\ast}\cI_{k,\alpha}(D)$ in $i^{\ast}\cO_X=\cO_H$ is defined to be $\cI_{k,\alpha}(D)\cdot \cO_H$, we conclude that there is an injection $\cI_{k,\alpha}(D_H)\hookrightarrow \cI_{k,\alpha}(D)\cdot \cO_H$.

Now suppose $H$ is sufficiently transverse to $D$. To prove that
$\cI_{k,\alpha}(D_H)=\cI_{k,\alpha}(D)\cdot \cO_H$, we need to show that the morphism
\eqref{eqn: restriction morphism} is an isomorphism. This is a local problem, and so
we may assume without loss of generality that $D$ is defined, in the graph embedding,
by a holomorphic function $t$ whose divisor is a smooth hypersurface.
By passing to subquotients, the morphism in \eqref{eqn: restriction
morphism} determines, for each $\alpha \in \R$, a morphism
\begin{equation}\label{eqn: gr of restriction morphism}
    \gr^F_{-(n-1)+k} \gr^V_{\alpha}\cM_H \to 
	 \omega_{H/X}\otimes i^{\ast}\gr^F_{-n+k} \gr^V_{\alpha}\cM.
\end{equation}
Here we are allowed to replace $\mathbf{L} i^*$ by $i^*$ because $H$, being sufficiently
transverse to $D$, is noncharacteristic with respect to the mixed Hodge
module $\gr^V_{\alpha} \cM$ for every $\alpha$.

We claim that it is enough to prove that \eqref{eqn: gr of restriction morphism} is an
isomorphism for every $\alpha$. Indeed, this implies that the kernel and
cokernel of \eqref{eqn: restriction morphism} are trivial modulo $V_{\alpha} \cM$
(respectively $V_{\alpha} \cM_H$) for all $\alpha \ll 0$. But by the definition of
mixed Hodge modules, multiplication by $t$ induces isomorphisms
\[
	t \colon \gr^F_p V_{\alpha} \cM \to \gr^F_p V_{\alpha-1} \cM \quad \text{and} \quad
	t \colon \gr^F_p V_{\alpha} \cM_H \to \gr^F_p V_{\alpha-1} \cM_H
\]
for every $\alpha < 0$. Since $\gr^F_p V_{\alpha} \cM$ is a coherent $\sO_X$-module, the
result that we want now follows from Krull's intersection theorem (with respect to
the ideals generated by the powers of $t$) and Nakayama's lemma.

Now the right-hand side of \eqref{eqn: gr of restriction morphism} is the
noncharacteristic restriction of the mixed Hodge module $\gr^V_{\alpha} \cM$ to the
hypersurface $H$. Since morphisms of mixed Hodge modules are strict with respect to
the Hodge filtration, it is therefore enough to show that
\begin{equation}\label{eqn: gr of restriction morphism as MHM}
    \gr^V_{\alpha}\cM_H \to 
	 \omega_{H/X}\otimes i^{\ast} \gr^V_{\alpha}\cM
\end{equation}
is an isomorphism of $\sD$-modules. But this follows from \cite[Thm.~1.1]{DMST06}, again due to the fact that $H$ is
sufficiently transverse to $D$.
\end{proof}

Let $p:X\to T$ be a smooth morphism of relative dimension $n$ between arbitrary varieties $X$ and $T$, and let $s:T\to X$ be a morphism such that $p\circ s=\mathrm{id}_T$. Suppose that $D$ is an effective divisor on $X$, relative over $T$. For every $t\in T$, denote by $D_t$ the restriction of $D$ on $p^{-1}(t)$.
\begin{thm}\label{thm: semicontinuity theorem}
For every $q\geq 1$, the set $ \left\{ t\in T \mid \cI_{k,\alpha}(D_t)\not\subseteq \fm^q_{s(t)}\right\}$ is open.
\end{thm}
\begin{proof}
Since we already have Theorem \ref{thm: restriction theorem}, the same proof of \cite[Theorem E]{MPRestriction} applies.\end{proof}

\begin{proof}[Proof of Theorem \ref{thm: numerical criterion for nontriviality}]
The problem is local, and so we may assume that $X$ is an open subset of $\C^n$, and that $D=\mathrm{div}(f)$ for a holomorphic function $f$ with $f(0)=0$. By cutting with $d$ generic chosen hyperplanes and using  Theorem \ref{thm: restriction theorem}, we can assume $\dim X=n-d$, where $d=\dim Z$ and $0\in D$ is an isolated singular point of multiplicity $m$.

Let $\mathbf{A}^N$ be the affine space parametrizing the coefficients of homogeneous polynomials of degree $m$, with coordinates $c_v$, for $v=(v_1,\ldots,v_{n-d})\in \bZ^{n-d}_{\geq 0}$, with $|v|\colonequals \sum_i v_i=m$. Let us consider the effective divisor $F$ on $X\times \bA^N$ defined by $f+\sum_{|v|=m}c_vx^v$. It is direct to see that there is an open neighborhood  $U\subseteq \bA^N$ of $0$, consisting of those $t\in \bA^N$ where
\[D_t\colonequals F\cap (X\times \{t\})\] 
is a reduced divisor on $X\times\{t\}\cong X$ and $D_t$ has an ordinary singularity at $0$ for $t\neq 0$ and $D_0=D$. Let $\alpha=-\frac{r+q-1}{m}$ for some $1\leq q\leq \max(m-r,m-1)$,  by \cite[Corollary 1.5]{DY25} one has \begin{equation*} \cI_{k,<\alpha}(D_t)_0=\cI_{k,-(r+q)/m}(D_t)_0 =\fm^{q}_0.
\end{equation*}
The first equality uses that all jumping numbers of $\cI_{k,<\alpha}(D_t)_0$ satisfy $m\alpha\in \Z$. By Theorem \ref{thm: semicontinuity theorem} we must have
\begin{equation*} 
\cI_{k,<\alpha}(D)_0=\cI_{k,<\alpha}(D_0)_0\subseteq  \fm^{q}_0.
\end{equation*}
The desired results follow.\end{proof}

\begin{proof}[Proof of Corollary \ref{corollary: upper bound of minimal exponent article}]
Write $\codim_X(Z)=mk+r$ for $k\in \N$ and $0\leq r\leq m-1$. Set $\alpha=-r/m\in (-1,0]$. Since $\cI_{k,<\alpha}(D)\neq \cO_X$ by Theorem \ref{thm: numerical criterion for nontriviality}, Lemma \ref{lemma: minimal exponent via Gkalpha} gives the desired bound.
\end{proof}

\subsection{The center of minimal exponent}\label{sec: the center of minimal exponent}
We construct a subscheme of $D$ associated to the minimal exponent $\tilde{\alpha}_D$, generalizing minimal log canonical centers \cite[\S 4]{KollarsingularitiesMMP}. Write $\tilde{\alpha}_D=k-\alpha$ for unique $k\in \N$ and $\alpha\in (-1,0]$. By Lemma \ref{lemma: minimal exponent via Gkalpha}, we have
\begin{equation}\label{eqn: vanishing of gkalpha by the minimality of D}
    \cI_{k,<\alpha}(D)\subsetneq \cO_X, \quad \cG_{k,\alpha}(D)\neq 0,\textrm{ and }\cG_{\ell,\alpha}(D)=0 \quad \textrm{whenever $\ell<k$}.
\end{equation}
Let $W_{\bullet}\cG_{k,\alpha}(D)$ be the weight filtration in Definition \ref{definition: weight filtration on graded pieces} and let $\ell$ be the largest integer such that $\gr^W_{\ell}\cG_{k,\alpha}(D)\neq 0$. Then there is a surjection
\begin{align}\label{eqn: definition of center of minimal exponent}
\cO_Z=\cG_{k,\alpha}(D)\twoheadrightarrow \gr^W_{\ell}\cG_{k,\alpha}(D).
\end{align}
In particular, $\gr^W_{\ell}\cG_{k,\alpha}(D)=\cO_Y$ for some subscheme $Y$ of $Z$, the zero locus of $\cI_{k,<\alpha}(D)$. \begin{definition}\label{definition: center of minimal exponent}
The \emph{center of minimal exponent} of $(X,D)$ is defined to be the subscheme $Y\subseteq X$ satisfying $\gr^W_{\ell}\cG_{k,\alpha}(D)=\cO_Y$, where $\ell,k,\alpha$ are defined above. The ideal sheaf of $Y$ fits into the following short exact sequence
\begin{equation}\label{eqn: ideal sheaf of center of minimal exponent}
0 \to \cI_{k,<\alpha}(D) \to \cI_Y \to W_{\ell-1}\cG_{k,\alpha}(D) \to 0.
\end{equation}
\end{definition}
\begin{proof}[Proof of Theorem \ref{thm: rationality and normality of the center of minimal exponent article}]
The isomorphism \eqref{eqn:gr^Wgkalpha as grFgrWgrV} gives
\[ \omega_X\otimes L^k\otimes \cO_Y=\omega_X\otimes L^k\otimes \gr^W_{\ell}\cG_{k,\alpha}(D)\cong \gr^F_{-n+k}\gr^W_{\ell}\gr^V_{\alpha}\cM.\]
Using \eqref{eqn: vanishing of gkalpha by the minimality of D} and \eqref{eqn:gr^Wgkalpha as grFgrWgrV}, we have
\[ \gr^F_{-n+k'}\gr^W_{\ell}\gr^V_{\alpha}\cM \cong \gr^W_{\ell}\cG_{k',\alpha}(D)=0, \quad \forall k'<k.\]
So locally $\cO_Y\cong F_{-n+k}\gr^W_{\ell}\gr^V_{\alpha}\cM$ is the first step in the Hodge filtration of a complex Hodge module. The desired property follows from Proposition \ref{prop: support of lowest Hodge filtration}.
\end{proof}
\begin{prop}\label{prop: support of lowest Hodge filtration}
Let $M$ be a complex Hodge module on a complex manifold $X$. Let $p$ be the integer such that $F_{p-1}\cM=0$ and $F_p\cM\neq 0$.  If $F_{p}\cM\cong \cO_Y$ for a connected closed subscheme $Y$, then $Y$ is irreducible, reduced, normal and has at worst rational singularities.
\end{prop}

\begin{proof}
Being a Hodge module, $M$ admits a decomposition by strict support. But since $Y$ is connected, $\cO_Y$ is indecomposable and so  $Y$ must be irreducible and there is a unique summand of $M$ that has strict support equal to $Y$ such that its lowest piece is $\cO_Y$. Without loss of generality, we can assume $M$ has strict support equal to $Y$.  We also find that $Y$ must be reduced, because it is the first step in the Hodge filtration of a complex mixed Hodge module. By the structure theorem for complex Hodge modules \cite[Part III, Section 16]{SSMHMproject}, $M$ is generically a variation of complex polarized Hodge structure $\cV$ on $Y$. Let $\mu:\tilde{Y}\to Y$ be a resolution of singularities that makes the singular locus of the CVHS $\cV$ into a normal crossing divisor $D$. On $\tilde{Y}$, using the structure theorem again, we then get a complex Hodge module $\tilde{M}$ by uniquely extending $\cV$ from $\tilde{Y}\setminus D$ to $\tilde{Y}$ such that $F_{p}\tilde{\cM}$ is a line bundle. The direct image theorem for complex Hodge modules gives
\[\cO_Y=F_p\cM\xrightarrow{\sim} \mathbf{R}\mu_{\ast}F_{p}\tilde{\cM}.\]
By adjunction we get a morphism $\cO_{\tilde{Y}}\to F_p\tilde{\cM}$. Taking direct image gives
\[\mathbf{R}\mu_{\ast}\cO_{\tilde{Y}}\to \mathbf{R}\mu_{\ast}F_{p}\tilde{\cM}\xrightarrow{\sim}F_p\cM=\cO_Y,\]
which is a splitting of the natural morphism $\cO_Y\to \mathbf{R}\mu_{\ast}\cO_{\tilde{Y}}$. By a result of Kov\'{a}cs \cite[Theorem 1]{Kovacs00}, this implies that $Y$ is normal with at worst rational singularities.
\end{proof}

Assume $D=\textrm{div}(f)$ for a holomorphic function $f$. We can describe the two sets above set-theoretically. By \cite[Corollary 5.7]{DLY}, one has
\begin{equation} Y=\{ x\in D \mid \underset{s=-\tilde{\alpha}_f}{\textrm{mult}}b_{f,x}(s)=\underset{s=-\tilde{\alpha}_f}{\textrm{mult}}b_{f}(s)\},\end{equation}
where $b_{f,x}(s)$ is the local $b$-function of $f$ near $x$. Lemma \ref{lemma: minimal exponent via Gkalpha} implies
\begin{equation} Z=\{ x\in D \mid \tilde{\alpha}_{f,x}=\tilde{\alpha}_f=\min_{x\in D}\tilde{\alpha}_{f,x}\},\end{equation}
where $\tilde{\alpha}_{f,x}$ is the minimal exponent of the local equation of $f$ near $x$. In particular, if $-\tilde{\alpha}_f$ is a simple root of $b_f(s)/(s+1)$, then $Y=Z$. Finally, one has
\begin{equation} D_{\textrm{sing}}=\{x\in D\mid \tilde{\alpha}_{f,x}<+\infty\},\end{equation}
because $D$ is smooth at $x$ if and only if $\tilde{\alpha}_{f,x}=+\infty$. Note that, as varieties, $Y\subseteq Z\subseteq D_{\textrm{sing}}$. 
\begin{example}
Since $\tilde{\alpha}_{f,x}$ is lower semicontinuous in $x$ \cite[Theorem E]{MP18Vfiltration}, it can certainly happen that $Z\subsetneq D$. For example, consider $f=x(x+y)\left((x-1)^2+y^3\right)$ on $\C^2$. The divisor $D$ has two singular points $p_1=(1,0)$ and $p_2=(0,0)$, where $p_1$ is a cusp and $p_2$ is a node. It follows that $\tilde{\alpha}_f=\tilde{\alpha}_{f,p_1}=5/6$, but $\tilde{\alpha}_{f,p_2}=1$. Thus $Z=\{p_1\}\subsetneq D_{\textrm{sing}}$. 
\end{example}
\begin{example}\label{example: center of minimal exponent for theta}
Suppose $f$ is the determinant function on $X$, the space of $n$ by $n$ matrices. Here $b_f(s)=\prod_{i=1}^n (s+i)$. So $\tilde{\alpha}_f=2$ and $\textrm{mult}_{s=-2}b_{f}(s)=1$, whenever $n\geq 2$. Suppose $x\in D$ is a matrix of rank $k$; then near $x$, up to a unit, $f$ is the determinant function on the space of $(n-k)$ by $(n-k)$ matrices. It follows from above that $Y=Z=D_{\Sing}$.\end{example}

The following two examples show that the strict inclusion $Y\subsetneq Z$ can occur.
\begin{example}\label{example:snc}
Let $X=\C^3$ and $f=xyz$. Here $b_f(s)=(s+1)^3$ and $\tilde{\alpha}_f=1$. So $Y$ is computed by $\gr^{V}_0$ (see \eqref{eqn: unipotent vanishing cycle D modules}). Let $p=(x_0,y_0,z_0)\in D_{\textrm{sing}}$, the union of $3$ axes. Assume one of the coordinates is nonzero, say $x_0\neq 0$. One can find local coordinates $(u,v,w)$ such that $f=vw$ near $p$. So $b_{f,p}(s)=(s+1)^2$. Otherwise, $b_{f,p}(s)=(s+1)^3$. Hence
\[ Y=\{(0,0,0)\}\subsetneq Z=D_{\textrm{sing}}.\]
A similar argument shows that if $D=\{x_1x_2\cdots x_n=0\}\subseteq \C^n$, then $Y$ is equal to the origin and $Z=D_{\textrm{sing}}$ (see also \cite[Example 13.3]{Park}).
\end{example}

\begin{example}
Here is a more sophisticated example communicated to us by Sung Gi Park. Let $X=\C^6$ and $f=xyz+uvw$. Here $b_f(s)=(s+1)(s+2)^3$ and $\tilde{\alpha}_f=2$. First, as above, one can show that $Z=D_{\textrm{sing}}$, which is the union of $9$ coordinate planes $(\C_{x,u},\C_{x,v},\C_{x,w},\ldots,\C_{z,w})$. As $\tilde{\alpha}_{xyz}=1$, by above the relevant $Y$ is computed by its vanishing cycle. Applying the Thom–Sebastiani Theorem for vanishing cycles \cite[Theorem 2]{MSS20} to $f_1=xyz$ and $f_2=uvw$, one can show (see, e.g., \cite[Theorem 13.1]{Park}) that $Y$ is the product of the correspondence locus of $f_1$ and $f_2$. Thus, $Y$ is the origin.
\end{example}

\subsection{Vanishing theorems}

\begin{proof}[Proof of Theorem \ref{thm: vanishing theorem for higher multiplier ideals}]
By Proposition \ref{prop: vanishing cycles are twisted D-modules}, $\gr^{W(N)}_{\ell}\gr^V_{\alpha}\cM$ underlies a twisted Hodge module, so the desired vanishing follows from Theorem \ref{thm: vanishing theorem for twisted Hodge module}.
\end{proof}

\begin{proof}[Proof of Theorem \ref{thm: vanishing theorem on abelian varieties}]
Since $\omega_A=\cO_A$, by \eqref{eqn:gr^Wgkalpha as grFgrWgrV}, we have
\[ L^{k+1} \otimes \gr^W_{\ell}\cG_{k,\alpha}(D)\otimes \rho \cong \gr^F_{-n+k}\gr^W_{\ell}\gr^V_{\alpha}\cM\otimes (L\otimes \rho).\]
Since $(L\otimes \rho)\otimes \cO_A(\alpha D)$ is ample for any $\alpha\in (-1,0], \rho\in \Pic^0(A)$, and in addition $\Omega^1_A\cong \cO_A^{\oplus \dim A}$, one can apply the second part of Theorem \ref{thm: vanishing theorem for twisted Hodge module} to $\gr^{W(N)}_{\ell}\gr^V_{\alpha}\cM$ to obtain
\[ H^i(A,L^{k+1}\otimes \gr^W_{\ell}\cG_{k,\alpha}(D)\otimes \rho)=0,\quad \textrm{for all $i\geq 1$}.\]
Then the first statement follows because the weight filtration is finite. Therefore\begin{equation}\label{eqn: vanishing of gkalpha on abelian varieties}
   H^i(A,L^{k+1}\otimes \cG_{k,\alpha}(D)\otimes \rho)=0, \quad \textrm{whenever $\alpha\in (-1,0]$, $i\geq 0 $}.
\end{equation}

For the vanishing of $\cI_{k,\alpha}(D)$, let us treat the cases $\alpha=-1$ and $\alpha=0$ together, by induction on $k\geq 0$. The base case is $k=0$, we have $\cI_{0,-1}(D)=\cJ(A,(1-\epsilon)D)$ by \eqref{eqn: I0 gives the multiplier ideals} for some $0<\epsilon\ll 1$. So the desired vanishing is a consequence of Nadel vanishing theorem (see \cite[Theorem 9.4.8]{Laz04II}). There is nothing to prove for $\alpha=0$. Assume the vanishing holds for some $k-1$ with $k\geq 1$. By \eqref{eqn: griffiths transversality morphism}, one has $\cI_{k,0}(D) \cong \cI_{k-1,-1}(D)$. Consequently, the vanishing for $L^k\otimes \cI_{k,0}(D)\otimes \rho$ follows from the induction hypothesis. Now consider the short exact sequence from Proposition \ref{prop: first properties of higher multiplier ideals for divisors}
\[ 0 \to L^k\otimes \cI_{k,0}(D) \to L^{k+1}\otimes \cI_{k,-1}(D)\to \gr^F_{-n+k}\omega_A(\ast D)\to 0.\]
Since $A\setminus D$ is affine and $\Omega^1_A$ is trivial, one can show that $H^i(A,\gr^F_{-n+k}\omega_A(\ast D)\otimes \rho)=0$ for any $i\geq 1$ and any $\rho\in \Pic^0(A)$ (see \cite[Theorem 28.2]{MPHodgeideal}). Therefore we conclude that $L^{k+1}\otimes \cI_{k,-1}(D)\otimes \rho$ has no higher cohomology as well and this finish the inductive proof.

It suffices to deal with the remaining case $\alpha\in (-1,0)$. Indeed, \eqref{eqn: vanishing of gkalpha on abelian varieties} implies that $L^{k+1}\otimes \cG_{k,\beta}(D)\otimes \rho$ has no higher cohomology for $\beta\in (-1,\alpha)$; this suffices, because we already know the result in the case $\alpha=-1$ and $\cI_{k,\alpha}(D)/\cI_{k,-1}(D)$ is a finite extension of $\cG_{k,\beta}(D)$ for $\beta\in (-1,\alpha)$.
\end{proof}

\section{Application: singularities of theta divisors}\label{sec: application to theta divisors}
Throughout, let $(A,\Theta)$ be an indecomposable principally polarized abelian variety of dimension $g\geq 1$, and let $\Theta$ be a \emph{symmetric} theta divisor.

\subsection{Minimal exponents of theta divisors}\label{sec: minimal exponent of theta divisors}
We first analyze the minimal exponent for the boundary cases in Conjecture \ref{conjecture: Casalaina Martin stronger}.
\begin{thm}\label{thm: minimal exponent of boundary theta divisors} 
Let $C$ be a smooth projective curve and $A=\Jac(C)$, then $1<\tilde{\alpha}_{\Theta}\leq 2$. If $C$ is Brill-Noether general, then $\tilde{\alpha}_{\Theta}=2$; if $C$ is hyperelliptic, then $\tilde{\alpha}_{\Theta}=\frac{3}{2}$. If $A$ is the intermediate Jacobian of a smooth cubic threefold, then $\tilde{\alpha}_{\Theta}=\frac{5}{3}$.
\end{thm}

\begin{remark}\label{remark: numerical evidence of the strong conjecture}
This provides some numerical evidences for Conjecture \ref{conjecture: Casalaina Martin stronger}: assume there exists a $m\geq 2$ such that $\dim \Sing_m(\Theta)\geq g-2m+1$, then Corollary \ref{corollary: upper bound of minimal exponent article} gives $\tilde{\alpha}_{\Theta}\leq \frac{2m-1}{m}<2$,
which are satisfied by the boundary examples. This is in fact the starting point of our proof.
\end{remark}
 
\begin{proof}
Let $C$ be a smooth projective curve. By \cite[Chapter IV, Lemma (3.3)]{ACGH} we know that every component of $\Sing_m(\Theta)$ has dimension greater or equal to Brill-Noether number     \[\rho=g-(m-1+1)(g-(g-1)+m-1)=g-m^2.\] 
    Therefore Corollary \ref{corollary: upper bound of minimal exponent article} gives $ \tilde{\alpha}_{\Theta} \leq \min_{m\geq 2} \frac{m^2}{m}=2$. 
    
    If $C$ is Brill-Noether general, then $\tilde{\alpha}_{\Theta}=2$ by \cite[Theorem 1.6(ii)]{budur2023local}. If $C$ is hyperelliptic, then by \cite[Chapter IV, Theorem 5.1]{ACGH} one has $\Sing_m(\Theta)$ is an irreducible variety of dimension $g-2m+1$ for each $m\geq 2$.  \cite[Theorem A]{SY22} provides a log resolution of $(A,\Theta)$ by iteratively blowing up of (the proper transform of) $\Sing_{m}(\Theta)$. Then by Corollary \ref{corollary: upper bound of minimal exponent article} and \cite[Corollary D]{MP18Vfiltration}, one has $\frac{3}{2}\leq \tilde{\alpha}_{\Theta}\leq\min_{m\geq 2}\frac{2m-1}{m}=\frac{3}{2}$.

    If $A$ is the intermediate Jacobian of a smooth cubic threefold,  \cite[p. 348]{Mumford} proved that $\Theta$ has a unique isolated ordinary singularity of multiplicity $3$, so $\tilde{\alpha}_{\Theta}=5/3$.
\end{proof}

\subsection{Properties of the center of minimal exponent}
By \cite[Theorem 1]{EL97} and  \cite[Theorem 0.4]{Saitorational}, $\Theta$ is reduced and has rational singularities, which is equivalent to \begin{equation}\label{eqn: minimal exponent >1}
\tilde{\alpha}_{\Theta} >1.
\end{equation}
Hence Lemma \ref{lemma: minimal exponent via Gkalpha} implies that
\begin{equation}\label{eqn: Kollar multiplier ideal ppav}
\cI_{0,-1}(\Theta)=\cO_A, \quad \cG_{0,\alpha}(\Theta)=0, \quad \forall -1<\alpha<0, \quad \cG_{0,-1}(\Theta)=\cO_{\Theta}, \quad \cG_{1,0}(\Theta)=0.
\end{equation}
Furthermore, one can also show that $N$ acts trivially on $\cG_{0,-1}(\Theta)$; we omit the details here. Consequently, there is no interesting information left in the usual multiplier ideals of $\Theta$ and $\cG_{0,-1}(\Theta)$. The idea is to use the first higher multiplier ideal $\cI_{1,\alpha}(\Theta)$ to get more information on the singularities of the theta divisor. 
\begin{lemma}\label{lemma: I1<0=OA}
 We have $\cI_{1,<0}(\Theta)=\cO_A$ and $\cI_{1,<-1}(\Theta)=\cO_A(-\Theta)$.
\end{lemma}
\begin{proof}
Using \eqref{eqn: gkalpha to gk+1alpha+1} and \eqref{eqn: Kollar multiplier ideal ppav}, we have $\cG_{1,\alpha}(\Theta)=0$ for $\alpha\in [0,1)$.  Since $\cI_{1,<1}(\Theta)=\cO_A$ by Proposition \ref{prop: first properties of higher multiplier ideals for divisors}, this gives $\cI_{1,<0}(\Theta)=\cO_A$. The second identity follows from \eqref{eqn: periodicity morphism}.
\end{proof}

\begin{lemma}\label{lemma: vanishing for the first ideal on abelian variety}
For $\alpha \in [-1,0)$, we have 
\[H^i\left(A,\cI_{1,\alpha}(\Theta)\otimes \cO_A(2\Theta)\otimes \rho\right)=0, \quad \forall i>0, \rho\in \Pic^0(A). \]
\end{lemma}
\begin{proof}It follows from  Theorem \ref{thm: vanishing theorem on abelian varieties}.\end{proof}

Let $Y$ be the center of minimal exponent for $(A,\Theta)$ in the sense of Definition \ref{definition: center of minimal exponent}. Recall in Remark \ref{remark: numerical evidence of the strong conjecture} that, if we want to prove Conjecture \ref{conjecture: Casalaina Martin stronger} by contradiction, we can assume there exists a $m\geq 2$ such that $\dim \Sing_m(\Theta)\geq g-2m+1$, then we must have $\tilde{\alpha}_{\Theta}<2$. Hence we can assume $\tilde{\alpha}_{\Theta}<2$ in the remaining of this section. In the proofs of the next two lemmas, we write $\tilde{\alpha}_{\Theta}=1-\alpha$ for some $\alpha\in (-1,0)$ and let $\ell$ be the maximal integer such that $\gr^W_{\ell}\cG_{1,\alpha}(\Theta)\neq 0$; note that $\cO_Y=\gr^W_{\ell}\cG_{1,\alpha}(\Theta)$.
\begin{lemma}\label{lemma: the center of minimal exponent for ppav}
Assume $1<\tilde{\alpha}_{\Theta}<2$. Then  $Y$ is connected and generates $A$ as an abelian variety. Moreover, $2\Theta|_Y$ is not very ample.
\end{lemma}
\begin{proof}
By \eqref{eqn: ideal sheaf of center of minimal exponent}, there is a short exact sequence
\[ 0\to \cI_{1,<\alpha}(\Theta) \to \cI_Y \to W_{\ell-1}\cG_{1,\alpha}(\Theta) \to 0.\]
Since $\alpha\in (-1,0)$, Theorem \ref{thm: vanishing theorem on abelian varieties} also gives
    \[ H^i(A,\cO_A(2\Theta)\otimes W_{\ell-1}\cG_{1,\alpha}(\Theta)\otimes L)=0, \quad \forall i>0, L\in \Pic^0(A).\]
Combined with Lemma \ref{lemma: vanishing for the first ideal on abelian variety}, we have 
\[H^i(A,\cI_Y\otimes \cO_A(2\Theta)\otimes L)=0, \quad \forall i>0, L\in \Pic^0(A).\]
As a consequence, we have
\begin{equation}\label{eqn: surjectivity of 2 theta}H^0\left(A,\cO_A(2\Theta)\otimes L\right) \twoheadrightarrow H^0\left(Y,(\cO_A(2\Theta)\otimes L)|_{Y}\right)\end{equation}
is surjective for any $L\in \Pic^0(A)$. 

We now prove the connectivity of $Y$. Recall that we assume $\Theta$ is symmetric,  i.e. $[-1]^{\ast}\cO_A(\Theta)\cong \cO_A(\Theta)$, where $[-1]:A\to A$ sends $x$ to $-x$. This induces an isomorphism on $\gr^V_{\alpha}s_{+}\omega_A$ as mixed Hodge modules, where $s$ is the total graph embedding of $\cO_A(\Theta)$. So it induces an isomorphism on $\cO_Y=\gr^F_{-g+1}\gr^W_{\ell}\gr^V_{\alpha}s_{+}\omega_A\otimes \cO_A(-\Theta)$. It follows that $Y$ must be symmetric. It is known that $|2\Theta|$ is base-point-free and the resulting morphism $A \to \P^{2^g-1}=\P H^0(A,2\Theta)$
factors through the quotient of $A$ by the involution $x\mapsto -x$.  This implies that $Y$ must be connected: otherwise we could translate $Y$ so that one connected component contains a point $x_0$ and another contains the point $-x_0$, which would contradict the fact that $|2\Theta|$ does not separate these two points.  

Finally let us prove the generation statement. Let $B$ be the subtorus generated by $Y$. The argument above shows that $2\Theta|_{Y}$ cannot separate $x_0$ and $-x_0$, thus not very ample. By \eqref{eqn: surjectivity of 2 theta}, $2\Theta|_B$ is also not very ample. It follows from \cite[Theorem A]{Ohbuchi88} that $B$ is isomorphic to a product of polarized abelian varieties with at least one positive dimensional principally polarized factor. We conclude $B=A$ because $(A,\Theta)$ is indecomposable.

\end{proof}

\begin{lemma}\label{lemma: vanishing theorem twist by 2Theta on Y}
    With the same assumption as in Lemma \ref{lemma: the center of minimal exponent for ppav}, then
    \begin{equation}\label{eqn: vanishing of Hi with twists on A} H^i(Y,\cO_A(2\Theta)|_Y\otimes L)=0, \quad \forall i>0, L\in \Pic^0(A).\end{equation}
Moreover, if $Y$ is smooth, then $2\Theta|_Y$ is the smallest piece of the Hodge filtration in a $(1+\alpha)\Theta|_Y$-twisted Hodge module on $Y$ and
        \begin{equation}\label{eqn: vanishing of twisted 2Theta on Y}H^i(Y,\cO_A(2\Theta)|_Y\otimes L_Y)=0,\end{equation}
for all $i>0$ and line bundle $L_Y$ on $Y$ such that $\cO_A((1+\alpha)\Theta)|_Y\otimes L_Y$ is ample on $Y$.

\end{lemma}

\begin{proof}
Denote by $\cM$ the direct image of $\omega_A$ along the total embedding of $\cO_A(\Theta)$. Since $\cO_Y=\gr^W_{\ell}\cG_{1,\alpha}(\Theta)$, it follows from \eqref{eqn:gr^Wgkalpha as grFgrWgrV} that
\[ \cO_A(\Theta)\otimes \cO_Y =\gr^F_{-g+1}\gr^W_{\ell}\gr^V_{\alpha}\cM.\]
By construction, $\gr^F_{k}\gr^W_{\ell}\gr^V_{\alpha}\cM=0$ for all $k\leq -g-1$. On the other hand, since $\alpha\in (-1,0)$, using \eqref{eqn:gr^Wgkalpha as grFgrWgrV} and \eqref{eqn: Kollar multiplier ideal ppav} one has
\[ \gr^F_{-g}\gr^W_{\ell}\gr^V_{\alpha}\cM=\cO_A(\Theta)\otimes \cG_{0,\alpha}(\Theta)=0.\]
Therefore $\cO_A(\Theta)\otimes \cO_Y$ is the smallest piece in the Hodge filtration of the $\alpha \Theta$-twisted Hodge module $\gr^W_{\ell}\gr^V_{\alpha}\cM$ (Proposition \ref{prop: vanishing cycles of Hodge modules}). Remark \ref{remark: change the twisting of twisted D module} implies that
\[ \cO_A(2\Theta)\otimes \cO_Y=\gr^F_{-g+1}((\gr^W_k\gr^V_{\alpha}\cM)\otimes \cO_A(\Theta))\]
 is the smallest piece in the Hodge filtration of a $(1+\alpha)\Theta$-twisted Hodge module. The assumption $\alpha>-1$ implies that the $\Q$-divisor $(1+\alpha)\Theta+L$ is ample on $A$ for all $L\in \Pic^0(A)$, hence we conclude \eqref{eqn: vanishing of Hi with twists on A} from Theorem \ref{thm: vanishing theorem for twisted Hodge module}.

Assume $Y$ is smooth, Theorem \ref{thm: twisted Kashiwara equivalence for Hodge modules} shows that $\cO_A(2\Theta)|_Y$ is the direct image of the first step in the Hodge filtration of an $(1+\alpha)\Theta|_Y$-twisted Hodge module on $Y$. Therefore \eqref{eqn: vanishing of twisted 2Theta on Y} follows from Theorem \ref{thm: vanishing theorem for twisted Hodge module}.
\end{proof}

\subsection{Proof of Theorem \ref{thm: partial solution of Casalaina Martin conjecture}}\label{sec: proof of the partial solution}
Suppose that for some $m\geq 2$, we have
\begin{equation}\label{eqn: violating dimension inequality}
    d\colonequals \dim \Sing_m(\Theta)\geq g-2m+1.
\end{equation} 
The Ein-Lazarsfeld bound \eqref{eqn: EinLazarsfeld bound} gives $d\leq g-m-1$, thus we get $g-d=m+r$ for
\begin{equation}\label{eqn: 1 < r< m-1}
    1\leq r\leq m-1.
\end{equation}
According to Theorem \ref{thm: numerical criterion for nontriviality}, we can conclude that 
\[ \cI_{1,<\beta}(\Theta)\subsetneq \cO_A, \quad \textrm{for some $\beta\geq -r/m>-1$}.\] Let $\alpha\in (-1,0]$ be the  maximal value for which $\cI_{1,<\alpha}(\Theta)\subsetneq \cO_A$. Lemma \ref{lemma: I1<0=OA} gives $\cI_{1,<0}(\Theta)=\cO_A$, so\begin{equation}\label{eqn: bound on the first jumping number alpha}
    -1<-r/m\leq \alpha <0.
\end{equation}
Let $\tilde{\alpha}_{\Theta}$ be the minimal exponent of $\Theta$. Corollary \ref{corollary: upper bound of minimal exponent article} implies that
\[ \tilde{\alpha}_{\Theta} \leq \frac{\codim_A(\Sing_m(\Theta))}{m}\leq \frac{2m-1}{m}<2.\]
Lemma \ref{lemma: minimal exponent via Gkalpha} implies that $\tilde{\alpha}_{\Theta}=1-\alpha$, and together with \eqref{eqn: minimal exponent >1} one has
\begin{equation}\label{eqn: bound on the minimal exponent}
    1<\tilde{\alpha}_{\Theta}\leq \frac{g-d}{m}<2.
\end{equation}
Let $Y$ be the center of minimal exponent of $(A,\Theta)$. As $1<\tilde{\alpha}_{\Theta}<2$, one can apply Lemma \ref{lemma: the center of minimal exponent for ppav} and Theorem \ref{thm: rationality and normality of the center of minimal exponent article} to conclude that $Y$ is irreducible and normal. Since $\dim Y=1$ by assumption, it must be a smooth projective curve. For the rest of the proof, set \[ C\colonequals Y, \quad e\colonequals \Theta\cdot C=\deg_{C}(\Theta|_C), \quad g(C)\colonequals \textrm{genus of $C$}.\]
By Lemma \ref{prop: jumping loci is a hyperelliptic curve} below, the curve $C$ must be hyperelliptic and $e=g(C)$. 

Now let us finish the proof of Theorem \ref{thm: partial solution of Casalaina Martin conjecture}. Since $C$ is smooth, Lemma \ref{lemma: vanishing theorem twist by 2Theta on Y} implies that $\cO_A(2\Theta)|_C$ is the smallest piece in the Hodge filtration of a $(1+\alpha)\Theta|_C$-twisted Hodge module on $C$. Then Corollary \ref{cor: degree bound of lowest twisted Hodge module} and $e=g(C)$ give
\[ 2e=\deg_C(\cO_A(2\Theta)|_C)\geq \deg \omega_C+(1+\alpha)\deg_C(\cO_A(\Theta)|_C)=2e-2+(1+\alpha)e.\]
This is equivalent to
\begin{equation}\label{eqn: 1+alpha <2}
(1+\alpha)e\leq 2.
\end{equation}
On the other hand, Lemma \ref{lemma: the center of minimal exponent for ppav} says that $C$ generates $A$ and so the Matsusaka-Ran theorem \cite{Ran} gives $e=\Theta\cdot C\geq \dim A=g$, with equality if and only if $A=\Jac(C)$. Recall that $d=\dim \Sing_m(\Theta)\geq 0$ and $r=g-d-m$. Then \eqref{eqn: 1 < r< m-1} and \eqref{eqn: bound on the first jumping number alpha} give
\[\alpha \geq -r/m, \quad 1\leq r \leq m-1.\] 
Combining with the inequality \eqref{eqn: 1+alpha <2}, we see that
\begin{equation}\label{eqn: m+r+d<2m/m-r}
 m+r\leq m+r+d=g\leq e \leq \frac{2}{1+\alpha}\leq \frac{2m}{m-r}.
\end{equation}
This implies that $m^2-2m\leq r^2$. Suppose $r\leq m-2$, then $m^2-2m\leq m^2-4m+4$. Therefore, we have $m\leq 2$ and $r\leq 0$, which contradicts with $1\leq r\leq m-1$. In particular, we conclude that $r=m-1$.  Plugging $r=m-1$ back to \eqref{eqn: m+r+d<2m/m-r}, we see that
\[ 2m+d-1\leq 2m.\]
Therefore $d=0$ or $d=1$, and $e=g$ or $e=g+1$.

If $e=g$, then we see that the abelian variety $A$ contains a curve $C$ generating $A$ with
\[ C\cdot \Theta=e=g=\dim A.\]
By the Matsusaka-Ran criterion for Jacobians \cite{Ran}, we conclude that  $A=\mathrm{Jac}(C)$.

If $e=g+1$, we claim that this is impossible (we thank Nelson Alvarado for pointing this out). In this case, $C$ is an irreducible curve generating $A$ and
\[ C\cdot \Theta=e=g+1=\dim A+1.\]
Then a result of Debarre (cf. \cite[Proposition 6.7]{Debarre94}) implies that $A$ must be isomorphic to the Jacobian of $C$. Hence $g(C)=\dim A$, which contradicts with the fact $g(C)=e=\dim A+1$. Therefore the case $e=g+1$ cannot happen.

% $\alpha=-\frac{m-1}{m}$ and $\tilde{\alpha}_{\Theta}=\frac{2m-1}{m}$. In this case, $\Theta$ has a unique singular point with multiplicity $m$. If $m=3$, then \cite[Proposition 3.5]{Casalaina08} says that either $A$ is the intermediate Jacobian of a cubic threefold or $A$ is the Jacobian of a hyperelliptic curve. The first case is impossible because then $\dim \Theta_{\Sing}=0$ but $C\subseteq \Theta_{\Sing}$. The second case is ruled out by Theorem \ref{thm: minimal exponent of boundary theta divisors} where $\tilde{\alpha}_{\Theta}=3/2\neq 5/3$.

\begin{lemma}\label{prop: jumping loci is a hyperelliptic curve}
The curve $C$ must be hyperelliptic and $e=g(C)$.
\end{lemma}

\begin{proof}
Since $1<\tilde{\alpha}_{\Theta}<2$, Lemma \ref{lemma: vanishing theorem twist by 2Theta on Y} gives $H^1(C,\cO_A(2\Theta)|_C)=0$. Therefore by Riemann-Roch we have
\[ \dim H^0(C,\cO_A(2\Theta)|_{C})=2e-g(C)+1.\]
It follows from \eqref{eqn: surjectivity of 2 theta} that the morphism defined by $|2\Theta|_C|$ maps $C$ two-to-one to a non-degenerate curve of degree $e$ in $\P^{2e-g(C)}$. Castelnuovo's theorem on degrees of non-degenerate curves gives us the bound 
\begin{equation}\label{eqn: Castelnuovo bound}
    e\geq 2e-g(C), \quad \textrm{or equivalently }  g(C)\geq e.
\end{equation}

On the other hand, for any $L_C\in \Pic^0(C)$, as $\alpha>-1$ by \eqref{eqn: bound on the first jumping number alpha}, we know $(1+\alpha)\Theta|_{C}\otimes L_C$ is ample on $C$. Hence Lemma \ref{lemma: vanishing theorem twist by 2Theta on Y} gives 
\[ H^1(C,\cO_A(2\Theta)|_C\otimes L_C)=0, \quad \forall L_C\in \Pic^0(C).\]
By Serre duality, this is equivalent to
\[ H^0(C,\omega_C\otimes \cO_A(-2\Theta)|_C\otimes L_C)=0, \quad \forall L_C\in \Pic^0(C).\]
Note that if $P$ is a line bundle on $C$ such that $H^0(C,P\otimes L)=0$ for all $L\in \Pic^0(C)$, then $\deg P<0$. Hence
\[2e=\deg_C(2\Theta|_C)>\deg_C(\omega_C)= 2g(C)-2,\]
and hence that $e\geq g(C)$. We therefore have equality in the Castelnuovo bound \eqref{eqn: Castelnuovo bound}, and so the image of $C$ in $\P^{2e-g(C)}=\P^{e}$ must be a rational normal curve. We conclude that $C$ must be hyperelliptic, as a 2:1 cover of a rational normal curve.
\end{proof}
\subsection{Proof of Theorem \ref{thm: nearby cycle of hyperelliptic curve}}

Let us write $\mu \colon \tilde{A} \to A$ for the log resolution in \cite[Theorem A]{SY22}. Set $n = \lfloor \frac{g-1}{2} \rfloor$ and abbreviate $W_{g-1}^r$ as $W^r$, so that $\Theta = W^0$. Write
\[
	\mu^{\ast} \Theta = \sum_{r=0}^n (r+1) D_r,
\]
where $D_r$ sits over $W^r$ and $D_0 = \tilde{\Theta}$ the proper transform of $\Theta$. The direct image $\derR \mu_{\ast} \QQ_{\tilde{A}}[g]$
contains $\QQ_A[g]$ as a direct summand, and all the other summands are supported
inside the singular locus of $\Theta$. Consider an affine open subset of $A$ over which $\Theta$ is defined by a holomorphic function $f$. To simplify the notation, we still denote by $W^r$ the intersection of $W^r$ with this affine open. Because nearby cycles commute with direct
images by proper morphisms, and because $\psi_f$ is trivial on perverse sheaves with
support inside of $\Theta$, so
\begin{equation}\label{eqn: vanishing cycle proper base change}
	\psi_{f, \lambda} \QQ_A[g] \cong \derR \mu_{\ast} \bigl( \psi_{\mu \circ f,
		\lambda} \QQ_{\tilde{A}}[g] \bigr).
\end{equation}
Let $x \in W^r \setminus W^{r+1}$, let $X = \mu^{-1}(x)$ be the fiber in
the log resolution, and write $i$ and $i_x$ for resulting closed embeddings, as in
the following commutative diagram:
\[
	\begin{tikzcd}
		X \rar{i} \dar{p} & \tilde{A} \dar{\mu} \\
		\pt \rar{i_x} & A
	\end{tikzcd}
\]
By proper base change and \eqref{eqn: vanishing cycle proper base change}, we have
\[
	\iu_x \psi_{f, \lambda} \QQ_A[g] \cong \iu_x \derR \mu_{\ast} \bigl( \psi_{\mu
		\circ f, \lambda} \QQ_{\tilde{A}}[g] \bigr)
		\cong \derR \pl \bigl( \iu \psi_{\mu \circ f, \lambda} \QQ_{\tilde{A}}[g] \bigr).
\]

We can compute the restriction on the right-hand side in two steps. First, we
restrict to the divisor $D_r$, the unique component of the normal crossing divisor
that contains $X$; according to Lemma \ref{lemma: restriction to a divisor}, this
restriction is trivial unless $\lambda^{r+1} = 1$, in which case we get the
$\ast$-extension of the rank-one local system whose monodromy around $D_r \cap D_j$
is $\lambda^{j+1}$, with a shift by $g-1$. We then have to restrict further to $X$. Now according to \cite[Corollary 6.6]{SY22}, $X$ is isomorphic to
Bertram's log resolution of $\PP^{2r}$ along the top secant variety of the rational
normal curve of degree $2r$. Moreover, \cite[Claim 6.4]{SY22} tells us that the morphism
\[
	\mu^{-1} \bigl( W^r \setminus W^{r+1} \bigr) \to W^r \setminus W^{r+1}
\]
is smooth on every stratum of the normal crossing divisor $D = D_0 + \dotsb + D_n$,
and therefore a stratified fiber bundle. Therefore restriction to $X$ is
non-characteristic for the $\ast$-extension of our local system. The conclusion is
that the perverse sheaf $\iu \psi_{\mu \circ f, \lambda} \QQ_{\tilde{A}}[g]$ is isomorphic to the $\ast$-extension of the rank-one local system on $X \setminus
(D_0 \cup \dotsb \cup D_{r-1})$ whose monodromy is $\lambda^{j+1}$ along the divisor
$X \cap D_j$, still with a shift by $g-1$. 

Let $S \subseteq \PP^{2r}$ denote the top secant variety; this is a
hypersurface of degree $r+1$. Then $X \setminus (D_0 \cup \dotsb \cup D_{r-1})$ is
isomorphic to $\PP^{2r} \setminus S$ by \cite[Corollary B]{SY22}, and so the perverse sheaf on $X$ is just the
$\ast$-extension of $L_p[g-1]$, where $L_p$ is the local system on $\PP^{2r} \setminus
S$ that we get from the $(r+1)$-fold cyclic covering $\tilde{S} \to \PP^{2r}$. Here
$0 \leq p \leq r$ is again the unique integer such that $\lambda = e^{2 \pi i
p/(r+1)}$. The cohomology of the $\ast$-extension is therefore isomorphic to the
appropriate $\mu_{r+1}$-eigenspace in the cohomology group
\[
	H^{g-1+\ast}(\tilde{S} \setminus S, \CC) \cong H^{g-1+\ast}(F_r, \CC),
\]
where $F_r$ is the affine Milnor fiber of the defining function
\[
	\det \begin{pmatrix}
		x_0 & x_1 & \dots & x_r \\
		x_1 & x_2 & \dots & x_{r+1} \\
		\vdots & \vdots & \ddots & \vdots \\
		x_r & x_{r+1} & \dots & x_{2r}
	\end{pmatrix}
\]
of $S$. These cohomology groups, together with the $\mu_{r+1}$-action, were computed
by Brogan \cite{Dan}. The relevant result is that each eigenspace has dimension $1$, and occurs
either in cohomological degree $2r$ or in a strictly smaller degree, depending on
whether or not $\lambda$ is a primitive $(r+1)$-th root of unity.

To summarize, for a point $x \in W^r \setminus W^{r+1}$, the stalk $\iu_x \psi_{f, \lambda} \QQ_A[g]$ is trivial unless $\lambda^{r+1} = 1$. If $\lambda$ is an $(r+1)$-th root of unity,
then the stalk consists of a single one-dimensional vector space in a certain degree;
the shifts work out in such a way that this degree is $2r+1-g = -\dim W^r$ if $\lambda$ is a
primitive $(r+1)$-th root of unity, and is strictly smaller otherwise. Just as in Brogan's work, we will show that this fact is enough to show that $K_{\lambda} \colonequals
\psi_{f,\lambda} \QQ_A[g]$ must be the intermediate extension of a rank-one local
system on $W^r \setminus W^{r+1}$. The following lemma formalizes the main point.
\begin{lemma}\label{lem: Klambda support on Wr}
Let $\lambda$ be a
primitive $(r+1)$-th root of unity. Then $K_{\lambda}$ is supported on $W^r$ and does not have quotient supported in a proper subvariety of $W^r$.
\end{lemma}
\begin{proof}[Proof of the lemma]According to the calculation above, the stalk of $K_{\lambda}$ is zero at every
	point of $A \setminus W^r$, whereas the stalk at points in $W^r \setminus W^{r+1}$
	is a one-dimensional vector space in degree $-\dim W^r$. This shows that $\Supp
	K_{\lambda} = W^r$. Now suppose that $K_{\lambda}$ has a nontrivial quotient with
	support equal to a proper subvariety $Z \subset W^r$. Let $x \in Z$ be a generic
	point; then $x \in W^k \setminus W^{k+1}$ for some $k > r$. Generically on $Z$,
	the quotient is a local system of some rank; this gives us a nontrivial morphism
	\[
		K_{\lambda} \to (i_x)_{\ast} V[\dim Z],
	\]
	where $V$ is a nonzero vector space. This implies that $H^{-\dim Z} \iu_x
	K_{\lambda} \neq 0$; but as $\dim Z < \dim W^r$, this contradicts the result
	about the stalk of $K_{\lambda}$ from above.
\end{proof}
Now we use Lemma \ref{lem: Klambda support on Wr} to prove that $K_{\lambda}$ must be the intermediate extension of a
rank-one local system on $W^r \setminus W^{r+1}$. The reason is the existence of the
weight filtration. Indeed, $K_{\lambda}$ is a direct summand in $\psi_f \QQ_A[g]$, and
therefore has a weight filtration with pure subquotients; by construction, the weight
filtration is symmetric around $g-1$. Since $K_{\lambda} / W_{g-1} K_{\lambda}$ would
be supported in a proper subset of $W^r$, the argument above shows that $K_{\lambda} = W_{g-1}
K_{\lambda}$.  Due to the symmetry of the weight filtration, this means that
$K_{\lambda}$ is actually pure of weight $g-1$. Because of the decomposition by strict
support, we conclude that $K_{\lambda}$ has strict support $W^r$.

Finally, because $\tilde{\alpha}_{\Theta}=3/2$ by Theorem \ref{thm: minimal exponent of boundary theta divisors} and $\psi_{f,e^{2\pi i\cdot (-1/2)}}\QQ_A[g]$ is pure, we see that the center of minimal exponent of $(A,\Theta)$ is the vanishing locus of $\cI_{1,<-1/2}(\Theta)$, which is
\[ \{x \in \Theta \mid \tilde{\alpha}_{\Theta,x}=3/2\}.\]
Using the log resolution in \cite{SY22}, \cite[Corollary D]{MP18Vfiltration} and Corollary \ref{corollary: upper bound of minimal exponent article}, one has \[ 3/2=\tilde{\alpha}_{\Theta}=\min_{x\in \Theta}\tilde{\alpha}_{\Theta,x}\leq \tilde{\alpha}_{\Theta,x}\leq \frac{\mathrm{codim}_{A}(\Sing_2(\Theta))}{2}=3/2, \quad \forall x\in \Sing_2(\Theta)\setminus\Sing_3(\Theta).\]
By the closedness, the vanishing locus of $\cI_{1,<-1/2}(\Theta)$ is exactly the singular locus of $\Theta$.

\subsection{A question}
For any indecomposable p.p.a.v $(A,\Theta)$, one has $\tilde{\alpha}_{\Theta}>1$ by \eqref{eqn: minimal exponent >1}. Inspired by Conjecture \ref{conjecture: Casalaina Martin stronger} and Theorem \ref{thm: minimal exponent of boundary theta divisors}, we ask
\begin{question}
    Let $(A,\Theta)$ be an indecomposable principally polarized abelian variety. Does one always have $\tilde{\alpha}_{\Theta}\geq \frac{3}{2}?$ If it is true, do hyperelliptic Jacobians characterize the equality case?
\end{question}

\appendix

\section{$V$-filtrations for powers of functions}\label{sec:  technical results of Vfiltrations}
We prove a formula for the nearby cycles of a $\sD$-module with
respect to a power of a function, which is used for the definition of higher multiplier ideals of $\Q$-divisors. This is closely related to \cite[Proposition 9.9.1]{SSMHMproject}, but we choose to include it here for reader's convenience. Let $f:X \to \mathbb{C}$ be a nonconstant
holomorphic function on a complex manifold $X$. We want to relate the nearby cycles
and the $V$-filtration with respect to the two functions $f$ and $f^m$. Let $M\in
\MHM(X)$, and denote by $(\cM,F_{\bullet}\cM)$ the underlying filtered right
$\sD$-module. For $m\geq 1$, consider the graph embedding 
\[
	i_m:X \to X\times \C, \quad i_m(x)= \bigl( x,f(x)^m \bigr).
\]
In order to keep the notation consistent, we set
\[
	M_m\colonequals (i_m)_{\ast}M\in \MHM(X\times \C),
\]
and denote the underlying filtered $\sD$-module by
$(\cM_m,F_{\bullet}\cM_m)=(i_m)_{+}(\cM,F_{\bullet}\cM)$. Let $V_{\bullet}\cM_m$
denote the $V$-filtration with respect to the function $t:X\times \C\to \C$. 

\begin{prop}[$V$-filtrations for $f^{m}$]\label{prop: V filtration for powers of functions}
	For any real number $\alpha < 0$, there is a natural isomorphism of filtered
	$\sO_X$-modules
	\[ 
		\phi_{\alpha} \colon \bigl( V_{m \alpha} \cM_1, F_{\bullet} V_{m \alpha} \cM_1 \bigr)
		\to \bigl( V_{\alpha}\cM_m, F_{\bullet} V_{\alpha} \cM_m \bigr)
	\]
	such that $\phi_{\alpha}(v \cdot t^m) = \phi_{\alpha}(v) \cdot t$ and
	$\phi_{\alpha}(v \cdot \frac{1}{m} t \partial_t) = \phi_{\alpha}(v) \cdot t \partial_t$ for
	all local sections $v \in V_{m\alpha} \cM_1$. If multiplication by $f$ is
	injective on $\Mmod$, then the same is true for $\alpha = 0$.
\end{prop}

\begin{corollary} \label{cor: V filtration for powers of functions}
On the level of mixed Hodge modules, this gives 
\[
	\psi_{f^m,\lambda}M\cong \psi_{f,\lambda^m}M 
	\quad \text{and} \quad
	\phi_{f^m,1}M\cong \phi_{f,1}M,
\]
under the assumption that multiplication by $f$ is injective on $\Mmod$.
\end{corollary}

The proof takes up the remainder of this section. We observe that the graph
embeddings fit into a commutative diagram
\[ \begin{tikzcd}   
X \arrow[r,"i_1"] \arrow[dr,bend right, "i_m"] &  X\times \C_t \arrow[r,"j"] \arrow[d,"p"]& X\times \C_t\times \C_s \arrow[dl,bend left,"q"]\\
 & X\times \C_s & 
\end{tikzcd}\]
in which $j(x,t)=(x,t,t^m), p(x,t)=(x,t^m)$, and $q(x,t,s)=(x,s)$. To distinguish the
two copies of $X\times \C$, let us denote the two coordinate functions by $t$ and
$s$, as indicated in the diagram above; then $s \circ p = t^m$. Since the diagram is
commutative, we have
\begin{equation}\label{eqn: computation of Mm via M1}
	M_m \cong p_{\ast} M_1 \cong q_{\ast} \bigl( j_{\ast} M_1 \bigr).
\end{equation}
Our strategy is to compute these direct images, and then use the bistrictness of
direct images with respect to the Hodge and $V$-filtration. 

We begin by describing the $\sD$-modules involved in the computation. Since $i_1$ is
a closed embedding, the $\sD$-module underlying the direct image $M_1 = (i_1)_{\ast}
M$ is 
\[
	\Mmod_1 = \Mmod \tensor_{\C} \C[\partial_t],
\]
with the right $\sD_{X \times \C_t}$-module structure determined by the following
formulas:
\begin{align*}
	(u \tensor \partial_t^k) \cdot t &= uf \tensor \partial_t^k + k u \tensor
	\partial_t^{k-1} \\
	(u \tensor \partial_t^k) \cdot \partial_t &= u \tensor \partial_t^{k+1} \\
	(u \tensor \partial_t^k) \cdot x_j &= u x_j \tensor \partial_t^k \\
	(u \tensor \partial_t^k) \cdot \partial_j &= u \partial_j \tensor \partial_t^k 
	- u \frac{\partial f}{\partial x_j} \tensor \partial_t^{k+1}
\end{align*}
Here $x_1, \dotsc, x_n$ are local coordinates on $X$, and $\partial_j =
\partial/\partial x_j$ are the corresponding vector fields. The Hodge filtration
$F_{\bullet} \Mmod_1$ is given by
\[
	F_p \Mmod_1 = \sum_{k \in \ZZ} F_{p-k} \Mmod \tensor \partial_t^k.
\]
Similarly, the $\sD$-module underlying the direct image $\Mt_1 = j_* M_1$ is
\[
	\Mmodt_1 = \Mmod_1 \tensor_{\C} \C[\partial_s], 
\]
with the right $\sD_{X \times \C_t \times \C_s}$-module structure determined by the
formulas
\begin{align*}
	(v \tensor \partial_s^{\ell}) \cdot s &= v t^m \tensor \partial_s^{\ell} 
	+ \ell v \tensor \partial_s^{\ell-1} \\
	(v \tensor \partial_s^{\ell}) \cdot \partial_s &= v \tensor \partial_s^{\ell+1} \\
	(v \tensor \partial_s^{\ell}) \cdot P &= v P \tensor \partial_s^{\ell} 
	\quad \text{for any $P \in \sD_X$} \\
	(v \tensor \partial_s^{\ell}) \cdot \partial_t &= v \partial_t \tensor
	\partial_s^{\ell} - m v t^{m-1} \tensor \partial_s^{\ell+1}
\end{align*}
and the Hodge filtration
\[
	F_p \Mmodt_1 = \sum_{\ell \in \ZZ} F_{p-\ell} \Mmod_1 \tensor \partial_s^{\ell}.
\]
Lastly, the filtered $\sD$-module underlying $M_m \cong q_* \Mt_1$ can be computed by the
relative de Rham complex, and this gives us a short exact sequence
\begin{equation} \label{eq:Mm-complex}
	\begin{tikzcd}
		0 \rar & (\Mmodt_1, F_{\bullet-1} \Mmodt_1) \rar{\partial_t} & 
		(\Mmodt_1, F_{\bullet} \Mmodt_1) \rar{\pi} & (\Mmod_m, F_{\bullet} \Mmod_m) \rar& 0.
	\end{tikzcd}
\end{equation}

The next step is to compute the $V$-filtrations. Since $M$ is a mixed Hodge module,
the $V$-filtration $V_{\bullet} \Mmod_1$ with respect to the function $t$ exists, and
with the convenient shorthand $F_p V_{\alpha} \Mmod_1 = F_p \Mmod_1 \cap V_{\alpha}
\Mmod_1$, the two morphisms
\[
	t \colon F_p V_{\alpha} \Mmod_1 \to F_p V_{\alpha-1} \quad
	\text{and} \quad
	\partial_t \colon F_p \gr_{\alpha}^V \Mmod_1 \to F_p \gr_{\alpha+1}^V \Mmod_1
\]
are isomorphisms for $\alpha < 0$ respectively $\alpha > 0$ \cite[(3.2.1)]{Saito88}. The $V$-filtration
$V_{\bullet} \Mmod_m$ with respect to the function $s$ also exists and has the same
properties. The following lemma describes the $V$-filtration of $\Mmodt_1$ with
respect to the function $s \colon X \times \C_t \times \C_s \to \C_s$, which is inspired by \cite[Theorem 3.4]{Saito90}.

\begin{lemma}\label{lemma: V filtration of M1Cs}
	The $V$-filtration on $\Mmodt_1 = \Mmod_1 \tensor_{\C} \C[\partial_s]$ is given by
	the formula
	\[
		V_{\alpha} \Mmodt_1 = \sum_{k,\ell \in \N} \Bigl( V_{m(\alpha-\ell)} \Mmod_1
		\tensor \partial_s^{\ell} \Bigr) \cdot \partial_t^k.
	\]
\end{lemma}

\begin{proof}
	It is easy to see that $V_{\alpha} \Mmodt_1$ is preserved by the action of $\sD_{X
	\times \C_t}$, and that one has $V_{\alpha} \Mmodt_1 \cdot \partial_s \subseteq
	V_{\alpha+1} \Mmodt_1$ and $V_{\alpha} \Mmodt_1 \cdot s \subseteq V_{\alpha-1}
	\Mmodt_1$ for every $\alpha \in \R$, with the second inclusion being an equality
	for $\alpha < 0$. It remains to show that each $V_{\alpha} \Mmodt_1$ is coherent
	over $V_0 \sD_{X \times \C_t \times \C_s}$, and that the operator $s \partial_s -
	\alpha$ acts nilpotently on the quotient $\gr_{\alpha}^V \Mmodt_1$. 
	Both of these facts rely on the following important identity, which is readily
	proved using the formulas for the $\sD$-module structure on $\Mmodt_1$:
	\begin{equation} \label{eq:key-relation}
		(v \tensor \partial_s^{\ell}) \cdot (s \partial_s - \alpha)
		= \frac{1}{m} \Bigl( v \bigl( t \partial_t - m(\alpha - \ell) \bigr) \tensor
		\partial_s^{\ell} - (v \tensor \partial_s^{\ell}) \cdot t \partial_t \Bigr)
	\end{equation}

	Let us first prove the coherence. Set $a = \max(0, \lfloor \alpha \rfloor)$. If
	$\alpha \not\in \N$, then $\alpha - a < 0$, and so for every integer $\ell
	\geq a + 1$, we have
	\[
		V_{m(\alpha - \ell)} \Mmod_1 \otimes \partial_s^{\ell}
		= V_{m(\alpha - a)} \Mmod_1 \cdot t^{m(\ell-a)} \otimes \partial_s^{\ell}
		= \Bigl( V_{m(\alpha -a)} \Mmod_1 \otimes 1 \Bigr) \cdot s^{\ell-a}
		\partial_s^{\ell}
	\]
	by the properties of the $V$-filtration on $\Mmod_1$. Together with the relation
	in \eqref{eq:key-relation}, this allows us to eliminate all the terms with $\ell
	\geq a+1$ from the formula for $V_{\alpha} \Mmodt_1$. 
	If $\alpha \in \N$, then $\alpha = a$, and for every integer $\ell \geq a+1$, we
	still have
	\[
		V_{m(\alpha - \ell)} \Mmod_1 \tensor \partial_s^{\ell}
		= V_{-1} \Mmod_1 \cdot t^{m-1 + m(\ell-a-1)} \tensor \partial_s^{\ell}
		= \Bigl( V_{-1} \Mmod_1 \cdot t^{m-1} \otimes 1 \Bigr) s^{\ell-a-1}
		\partial_s^{\ell}.
	\]
	With the help of the identity
	\[
		v t^{m-1} \tensor \partial_s^{\ell+1} 
		= \frac{1}{m} \Bigl( v \partial_t \tensor \partial_s^{\ell} 
		- (v \tensor \partial_s^{\ell}) \cdot \partial_t \Bigr)
	\]
	we can then again eliminate all terms with $\ell \geq a+1$ from the formula for
	$V_{\alpha} \Mmodt_1$. In both cases, the conclusion is that
	\begin{equation} \label{eq:V-filtration-bounded}
		V_{\alpha} \Mmodt_1 = \sum_{k \in \N} \sum_{\ell=0}^{\max(0, \lfloor \alpha
			\rfloor)} \Bigl( V_{m(\alpha-\ell)} \Mmod_1
		\tensor \partial_s^{\ell} \Bigr) \cdot \partial_t^k.
	\end{equation}
	Since each $V_{m(\alpha-\ell)} \Mmod_1$ is finitely generated over $V_0 \sD_{X
	\times \C_t}$, this shows that $V_{\alpha} \Mmodt_1$ is finitely generated over
	$V_0 \sD_{X \times \C_t \times \C_s}$. 

	Finally, we argue that $s \partial_s - \alpha$ acts nilpotently on
	$\gr_{\alpha}^V \Mmodt_1$. For this, it suffices to show that if $v \in
	V_{m(\alpha - \ell)} \Mmod_1$ is any local section, then 
	\[
		(v \tensor \partial_s^{\ell}) \cdot (s \partial_s - \alpha)^N
		\in V_{< \alpha} \Mmodt_1
	\]
	for $N \gg 0$. But since $t \partial_t - m(\alpha - \ell)$ acts nilpotently on
	$\gr_{m(\alpha-\ell)}^V \Mmod_1$, this is an easy consequence of the identity in
	\eqref{eq:key-relation}.
\end{proof}

The next task is to compute the intersection
\[
	F_p V_{\alpha} \Mmodt_1 = F_p \Mmodt_1 \cap V_{\alpha} \Mmodt_1.
\]
This is the content of the following lemma.

\begin{lemma}\label{lemma: Fp intersect Valpha}
	For every $\alpha < 0$ and every $p \in \Z$, one has
	\[
		F_p V_{\alpha} \Mmodt_1 = \sum_{k \in \N} \Bigl( F_{p-k} V_{m \alpha} \Mmod_1
		\tensor 1 \Bigr) \cdot \partial_t^k.
	\]
	This also holds for $\alpha = 0$, provided that multiplication by $f$ is injective
	on $\Mmod$. 
\end{lemma}

\begin{proof}
	Since the right-hand side is obviously contained in the left-hand side, it
	suffices to prove the reverse inclusion. According to
	\eqref{eq:V-filtration-bounded}, any local section $w \in V_{\alpha} \Mmodt_1$ can be
	written, for some $d \in \N$, in the form
	\[
		w = \sum_{k=0}^d \bigl( v_k \tensor 1) \cdot \partial_t^k,
	\]
	with local sections $v_0, \dotsc, v_d \in V_{m \alpha} \Mmod_1$. Using the
	formulas for the $\sD$-module structure on $\Mmodt_1$, this expression for $w$ can
	of course be rewritten as
	\[
		\sum_{\ell=0}^d u_{\ell} \tensor \partial_s^{\ell},
	\]
	which is a local section of $F_p \Mmodt_1$ exactly when $u_{\ell} \in F_{p-\ell}
	\Mmod_1$ for every $0 \leq \ell \leq d$. We now analyze these expressions from the
	top down. A short computation shows that
	\[
		u_d = (-1)^d m^d \cdot v_d t^{d(m-1)} \in F_{p-d} V_{m \alpha - d(m-1)} \Mmod_1.
	\]
	Now there are two cases. One case is $\alpha < 0$. Here multiplication
	by $t^{d(m-1)}$ is a filtered isomorphism, and therefore $v_d \in F_{p-d} V_{m
	\alpha}$. This means that $v_d \tensor \partial_s^d$ is contained in the
	right-hand side, and so we can subtract it from the expression for $w$, and
	finish the proof by induction on $d \geq 0$. The other case is $\alpha =
	0$. Here we can only conclude that $v_d t \in F_{p-d} V_{-1} \Mmod_1$. But since
	$M$ is a mixed Hodge module, the variation morphism $t
	\colon \gr_0^V \Mmod_1 \to \gr_{-1}^V \Mmod_1$ is strict with respect to the Hodge
	filtration. It follows that there is some $v_d' \in F_{p-d} V_0 \Mmod_1$ such that
	$v_d' t = v_d t$. Because multiplication by $f$ is injective on $\Mmod$,
	multiplication by $t$ is injective on $\Mmod_1$, and therefore $v_d \in
	F_{p-d} V_0 \Mmod_1$. We can then finish the proof exactly as in the case $\alpha
	< 0$.
\end{proof}

Now let us go back to the short exact sequence in \eqref{eq:Mm-complex}. Since $M$ is
a mixed Hodge module, the relative de Rham complex computing the direct image $M_m
\cong q_* \Mt_1$ is bistrict with respect to the Hodge filtration and the
$V$-filtration \cite[(3.3.3)-(3.3.5)]{Saito88}. For $\alpha \leq 0$ and $p \in \ZZ$,
we therefore get an induced short exact sequence
\begin{equation} \label{eq:Mm-complex-FV}
	\begin{tikzcd}
		0 \rar & F_{p-1} V_{\alpha} \Mmodt_1 \rar{\partial_t} & 
		F_p V_{\alpha} \Mmodt_1 \rar{\pi} & F_p V_{\alpha} \Mmod_m \rar& 0.
	\end{tikzcd}
\end{equation}
For the remainder of the argument, we shall assume that either $\alpha < 0$, or 
$\alpha = 0$ and multiplication by $f$ is injective on $\Mmod$. Under this
assumption, we have
\[
	F_p V_{\alpha} \Mmodt_1 = \sum_{k \in \N} \Bigl( F_{p-k} V_{m \alpha} \Mmod_1
	\tensor 1 \Bigr) \cdot \partial_t^k
	= F_p V_{m\alpha} \Mmod_1 \tensor 1 + F_{p-1} V_{\alpha} \Mmodt_1 \cdot
	\partial_t.
\]
We can now conclude from the exactness of \eqref{eq:Mm-complex-FV} that the morphism
\[
	\phi_{\alpha} \colon F_p V_{m\alpha} \Mmod_1 \to F_p V_{\alpha} \Mmod_m,
	\quad \phi_{\alpha}(v) = \pi(v \tensor 1),
\]
is an isomorphism of $\sO_X$-modules. From the formulas for the $\sD$-module
structure on $\Mmodt_1$, we obtain the identities
\[
	\phi_{\alpha}(v) \cdot s = \phi_{\alpha}(v \cdot t^m) \quad \text{and} \quad
	\phi_{\alpha}(v) \cdot s \partial_s = \phi_{\alpha} \left( v \cdot \frac{1}{m} t
	\partial_t \right).
\]
This proves Proposition~\ref{prop: V filtration for powers of functions}, up to the
change in notation caused by using $s$ for the coordinate function on the graph
embedding $i_m \colon X \to X \times \C_s$ by the function $f^m$. The asserted
identities for the nearby and vanishing cycles of the mixed Hodge module $M$ follow
by passing to the graded quotients $\gr_{\alpha}^V$ for $\alpha \in [-1,0]$.

\section{Restriction of nearby cycles}
We give a computation of how the nearby cycles with respect to a
normal crossing divisor behave when restricted to one of the components of the
divisor. This is used in Theorem \ref{thm: nearby cycle of hyperelliptic curve} for the computation of nearby cycle of hyperelliptic theta divisors. 

Let $X$ be a complex manifold of dimension $n+1$, and let $f \colon X \to \CC$ be a holomorphic function whose divisor
has simple normal crossings. Let us write
\[
	\dv(f) = m_0 D_0 + m_1 D_1 + \dotsb + m_r D_r,
\]
where $D_0, D_1, \dotsc, D_r$ are smooth and their union has simple normal crossings.
After restricting to $D_0$, this gives us a linear equivalence
\[
	-m_0 D_0 \vert_{D_0} \equiv m_1 D_1 \vert_{D_0} + \dotsb + m_r D_r \vert_{D_0},
\]
and therefore a section of the $m_0$-th power of the line bundle $\shO_{D_0}(-D_0)$
whose divisor again has simple normal crossings. As usual, this data determines a cyclic
covering of $U_0 \colonequals D_0 \setminus (D_1 \cup \dots \cup D_r)$ of degree $m_0$. By pushing
forward the constant sheaf from this cyclic covering, we get $m_0$ different local
systems $L_0, \dotsc, L_{m_0-1}$ of rank one on $U_0$; the local monodromy of $L_p$
around the divisor $D_0 \cap D_j$ is multiplication by the complex number $e^{2 \pi i
\, m_j p/m_0}$. Let $j_0 \colon U_0 \into D_0$ be the inclusion of this open set into
$D_0$, and let $i_0 \colon D_0 \into X$ be the inclusion of $D_0$ into $X$.

For any root of unity $\lambda \in \CC$, consider the perverse sheaf of nearby cycles $\psi_{f,
\lambda} \CC_X[n+1]$. 

\begin{lemma}\label{lemma: restriction to a divisor}
If $\lambda^{m_0} \neq 1$, then $i^{\ast}_0 \psi_{f, \lambda} \CC_X[n+1]$ is trivial. If
	$\lambda^{m_0} = 1$, let $0 \leq p \leq m_0-1$ be the unique integer such that $\lambda
	= e^{2 \pi i p/m_0}$. Then we have an isomorphism
	\begin{equation*}
		\iu_0 \psi_{f, \lambda} \CC_X[n+1] \cong \derR (j_0)_{\ast} L_p[n].
	\end{equation*}
	With this notation, the monodromy of $L_p$ around the divisor $D_0 \cap D_j$ is
	$\lambda^{m_j}$.
\end{lemma}

\begin{proof}
The isomorphism comes about as follows. If we restrict the perverse sheaf on the
left-hand side to the open subset $U_0$, we get
\[
	\ju_0 \iu_0 \psi_{f, \lambda} \CC_X[n+1] \cong L_p[n],
\]
because on the complement $X \setminus (D_1 \cup \dotsb \cup D_r)$ of the other
components, the divisor of $f$ is just $m_0 D_0$. Since $\derR (j_0)_{\ast}$ is the
right-adjoint of $\ju_0$, this gives us a morphism
\begin{equation} \label{eq:morphism}
	\iu_0 \psi_{f, \lambda} \CC_X[n+1] \to \derR (j_0)_{\ast} L_p[n],
\end{equation}
and the theorem is claiming that this morphism is an isomorphism.

Now all we need to do is check in local coordinates that \eqref{eq:morphism} is an
isomorphism. Since nothing interesting happens away from $D_0$, we can restrict our attention to
points on $D_0$. We can then choose local coordinates $x_0, x_1, \dotsc, x_n$, with
$D_0$ defined by $x_0 = 0$, such that $f$ is equal to a unit times $x_0^{m_0} x_1^{m_1}
\dotsb x_k^{m_k}$, where $m_0, \dotsc, m_k \geq 1$; if we define $m_{k+1} = \dotsb =
m_n = 0$, we can abbreviate this as $x^m$. 

Let $\Mmod$ be the right $\sD$-module underlying the nearby cycles $\psi_f \QQ_X[n+1]$. Using \cite[\S 4]{Chenqianyu2021}, in local coordinates as
above, we have
\[
	\Mmod \cong \Dmod_{\CC^{n+1}} / I,
\]
where $I$ is the right ideal generated by the monomial $x_0^{m_0} \dotsm x_k^{m_k}$,
by the vector fields $\frac{1}{m_0} x_0 \partial_0 - \frac{1}{m_j} x_j \partial_j$ for
$1 \leq j \leq k$, and by the vector fields $\partial_{k+1}, \dotsc, \partial_n$. The
restriction to the divisor $x_0 = 0$ is then given by the quotient
\[
	\Mmod / \Mmod \partial_0 \cong 
	\Dmod_{\CC^{n+1}} / \bigl( I + \Dmod_{\CC^{n+1}} \partial_0 \bigr).
\]
After playing with the relations a bit, this simplifies to $\Dmod_{\CC^n}[x_0] / J$, where $J$ is the right ideal generated by the monomial $x_0^{m_0}$, by the vector
fields $x_0^p (x_j \partial_j - \frac{m_j p}{m_0})$ for $1 \leq j \leq k$ and $0 \leq p \leq
m_0-1$, and by the vector fields $\partial_{k+1}, \dotsc, \partial_n$. This is a
direct sum of $m_0$ different $\Dmod$-modules, indexed by the integer $0 \leq p \leq
m_0-1$, which look like
\[
	\Dmod_{\CC^n} / \bigl( x_1 \partial_1 - m_1 p/m_0, \dotsc, x_k \partial_k - m_k p/m_0,
	\partial_{k+1}, \dotsc, \partial_n \bigr) \Dmod_{\CC^n}.
\]
This is a presentation for the $\ast$-extension of the local system $L_p$
across the divisor $x_1 \dotsm x_k = 0$, completing the proof of the
lemma.
\end{proof}

\bibliographystyle{abbrv}
\bibliography{higher_multiplier_ideals}

\vspace{\baselineskip}

\footnotesize{
\textsc{Department of Mathematics, Stony Brook University, Stony Brook, New York 11794, United States} \\
\indent \textit{E-mail address:} \href{mailto:cschnell@math.stonybrook.edu}{cschnell@math.stonybrook.edu}

\vspace{\baselineskip}

\textsc{Department of Mathematics, University of Kansas, 1450 Jayhawk Blvd, Lawrence, KS 66045, United States} \\
\indent \textit{E-mail address:} \href{mailto:ruijie.yang@ku.edu}{ruijie.yang@ku.edu} 
}

\end{document}